\documentclass{article}
\usepackage{amsmath,amssymb,amsthm,graphicx,epsfig,subfigure,float,url}
\usepackage[colorlinks=true]{hyperref}
\usepackage{pdfsync}
\usepackage{color}

\usepackage{verbatim}
\usepackage[applemac]{inputenc}
 
\usepackage[usenames,dvipsnames]{pstricks}
\usepackage{pst-grad} % For gradients
\usepackage{pst-plot} % For axes

\def\R{\textrm{I\kern-0.21emR}}
\def\N{\textrm{I\kern-0.21emN}}
\def\Z{\mathbb{Z}}
\newcommand{\C} {\mathbb{C}}

\renewcommand{\geq}{\geqslant}
\renewcommand{\leq}{\leqslant}

\newtheorem{theorem}{Theorem}
\newtheorem*{thmnonumbering}{Theorem}
\newtheorem{proposition}{Proposition}

\newtheorem{lemma}{Lemma}
\theoremstyle{definition}
\theoremstyle{definition}\newtheorem{remark}{Remark}

\newcommand{\Hun}{\mathbf{(H_1)}}
\newcommand{\Hdeux}{\mathbf{(H_2)}}
\newcommand{\Htrois}{\mathbf{(H_3)}}

\newcommand{\Real}{\mathrm{Re}}

\title{Optimal shape and location of sensors for parabolic equations with random initial data}

\author{Yannick Privat\footnote{CNRS, Sorbonne Universit\'es, UPMC Univ Paris 06, UMR 7598, Laboratoire Jacques-Louis Lions, F-75005, Paris, France ({\tt yannick.privat@upmc.fr}).}
	\and Emmanuel Tr\'elat\footnote{Sorbonne Universit\'es, UPMC Univ Paris 06, CNRS UMR 7598, Laboratoire Jacques-Louis Lions, Institut Universitaire de France, F-75005, Paris, France (\texttt{emmanuel.trelat@upmc.fr}).} 
        \and Enrique Zuazua\footnote{BCAM - Basque Center for Applied Mathematics, Mazarredo, 14 E-48009 Bilbao-Basque Country-Spain}\ \footnote{Ikerbasque, Basque Foundation for Science, Alameda Urquijo 36-5, Plaza Bizkaia, 48011, Bilbao-Basque Country-Spain (\texttt{zuazua@bcamath.org}).}}

\date{}

\begin{document}

\maketitle

\begin{abstract}
In this article, we consider parabolic equations on a bounded open connected subset $\Omega$ of $\R^n$.
We model and investigate the problem of optimal shape and location of the observation domain having a prescribed measure. 
This problem is motivated by the question of knowing how to shape and place sensors in some domain in order to maximize the quality of the observation: for instance, what is the optimal location and shape of a thermometer?

We show that it is relevant to consider a spectral optimal design problem corresponding to an average of the classical observability inequality over random initial data, where the unknown ranges over the set of all possible measurable subsets of $\Omega$ of fixed measure.
We prove that, under appropriate sufficient spectral assumptions, this optimal design problem has a unique solution, depending only on a finite number of modes, and that the optimal domain is semi-analytic and thus has a finite number of connected components. This result is in strong contrast with hyperbolic conservative equations (wave and Schr\"odinger) studied in \cite{PTZobsND} for which relaxation does occur.

We also provide examples of applications to anomalous diffusion or to the Stokes equations. In the case where the underlying operator is any positive (possible fractional)  power of the negative of the Dirichlet-Laplacian, we show that, surprisingly enough, the complexity of the optimal domain may strongly depend on both the geometry of the domain and on the positive power.

The results are illustrated with several numerical simulations.

\end{abstract}

\noindent\textbf{Keywords:} parabolic equations, optimal design, observability, minimax theorem.

\medskip

\noindent\textbf{AMS classification:} 93B07, 35L05, 49K20, 42B37.

\newpage
\tableofcontents
%\newpage

\section{Introduction}\label{secintro}
%\subsection{State of the art}
Given a bounded domain $\Omega$ of $\R^n$, in this paper we model and solve the problem of finding an optimal observation domain $\omega\subset\Omega$ for general parabolic equations settled on $\Omega$. 
We want to optimize not only the placement but also the shape of $\omega$, over all possible measurable subsets of $\Omega$ having a certain prescribed measure.
Such questions are frequently encountered in engineering applications but have been little treated from the mathematical point of view.
Our objective is here to provide a rigorous mathematical model and setting in which these questions can be addressed.
Our results will be established in a general parabolic framework and cover the cases of heat equations, anomalous diffusion equations or Stokes equations. For instance for the heat equation we will answer to the following question (that we will make more precise later on):
\begin{quote}
What is the optimal shape and location of a thermometer?
\end{quote}

%Let $n\geq 1$ be an integer, and let $M$ be a real analytic $n$-dimensional Riemannian manifold with metric $\mathfrak{g}$. 
%The Riemannian volume on $M$ associated with the metric metric $\mathfrak{g}$ is denoted by $V_{\mathfrak{g}}$, inducing the measure $dx$. Throughout the paper, measurable sets are considered with the measure $dx$.
%Let $\Omega$ be an open bounded connected subset of $M$.

\paragraph{Brief state of the art.}
Due to their relevance in engineering applications, optimal design problems for the placement of sensors for processes modeled by partial differential equations have been investigated in a large number of papers. Let us mention for instance the importance of the shape and placement of sensors for transport-reaction processes (see \cite{antoniades,demetriou-transport}).
Several difficulties overlap for such problems. On the one hand, the parabolic partial differential equations under consideration constitute infinite-dimensional dynamical systems, and, consequently, solutions live in infinite-dimensional spaces. On the other hand, the class of admissible designs is not closed for the standard and natural topology.
Few works take into consideration both aspects. Indeed, in many contributions, numerical tools are developed to solve a simplified version of the optimal design problem where either the partial differential equation has been replaced with a discrete approximation, or the class of optimal designs is replaced with a compact finite dimensional set (see for example \cite{armaoua,harris,vandewouwer} and \cite{morris} where such problems are investigated in a more general setting). 
In other words, in most of these applications the method consists in approximating appropriately the problem by selecting a finite number of possible optimal candidates and of recasting the problem as a finite-dimensional combinatorial optimization problem. In many studies the sensors have a prescribed shape (for instance, balls with a prescribed radius) and then the problem consists of placing optimally a finite number of points (the centers of the balls) and thus it is finite-dimensional, since the class of optimal designs is replaced with a compact finite-dimensional set.
Of course, the resulting optimization problem is already challenging. We stress however that, in the present paper, 
we want to optimize also the shape of the observation set, and  we do not make any a priori restrictive assumption to compactly the class of shapes ( $\omega$ to be of bounded variation, for instance) and the search is made over all possible measurable subsets.

From the mathematical point of view, the issue of studying a relaxed version of optimal design problems for the shape and position of sensors or actuators has been investigated in a series of articles. In \cite{munchHeat}, the authors study a homogenized version of the optimal location of controllers for the heat equation problem (for fixed initial data), noticing that such problems are often ill-posed. In \cite{allaireMunch}, the authors consider a similar problem and study the asymptotic behavior as the final time $T$ goes to infinity of the solutions of the relaxed problem; they prove that optimal designs converge to an optimal relaxed design of the corresponding two-phase optimization problem for the stationary heat equation. We also mention \cite{munchPedr} where, for fixed initial data, numerical investigations are used to provide evidence that the optimal location of null-controllers of the heat equation problem is an ill-posed problem.
In \cite{PTZobspb1} we proved that, for fixed initial data as well, the problem of optimal shape and location of sensors is always well posed for heat, wave or Schr\"odinger equations (in the sense that no relaxation phenomenon occurs); we showed that the complexity of the optimal set depends on the regularity of the initial data, and in particular we proved that, even for smooth initial data, the optimal set may be of fractal type (and there is no relaxation).
%In \cite{munchZuazua}, the authors investigate numerical approximations of exact or trajectory controls for the heat equation, by developing a numerical version of the so-called transmutation method.

A huge difference between these works and the problem addressed in this paper is that all criteria introduced in the sequel take into consideration all possible initial data. Moreover, the optimization will range over all possible measurable subsets having a given measure.
This the idea developed in \cite{PTZ_HUM1D,PTZObs1,PTZobsND}, where the problem of the optimal location of an observation subset $\omega$ among all possible subsets of a given measure or volume fraction of $\Omega$ was addressed and solved for conservative wave and Schr\"odinger equations. A relevant spectral criterion was introduced, viewed as a measure of eigenfunction concentration, in order to design an optimal observation or control set in an uniform way, independent of the data and solutions under consideration.
Such a kind of uniform criterion was earlier introduced for the one-dimensional wave equation in \cite{henrot_hebrardSCL,henrot_hebrardSICON} to investigate optimal stabilization issues.

The main difference of the previous analyses of conservative wave-like problems with respect to the present one is that, here, due to strong dissipativity of the heat equation (or of more general parabolic equations), high-frequency components are penalized in the spectral criterion, thus making optimal shapes to be determined by the low frequencies only, which, in particular, avoids spillover phenomena to occur.

\paragraph{Overview of the results of this paper.}
Let us now provide a short overview of the results of the present paper, without introducing (at this step) the whole general parabolic framework in which our results are actually valid.

\medskip

%\paragraph{The problem.}
Let $\Omega$ be an open bounded connected subset of $\R^n$.
Let $T$ be a fixed (arbitrary) positive real number.
To start with a simple model, let us consider the heat equation
\begin{equation}\label{heatEq_intro}
\partial_t y-\triangle y=0, \quad (t,x)\in (0,T)\times\Omega,
\end{equation}
with Dirichlet boundary conditions.
For any measurable subset $\omega$ of $\Omega$, we observe the solutions of \eqref{heatEq_intro} restricted to $\omega$ over the horizon of time $[0,T]$, that is, we consider the observable $z(t,x)=\chi_\omega(x)y(t,x)$, where $\chi_\omega$ denotes the characteristic function of $\omega$. The subset $\omega$ models sensors, and a natural question is to determine what is the best possible shape and placement of the sensors in order to maximize the observability in some appropriate sense, for instance in order to maximize the quality of the reconstruction of solutions. In other words, we ask the question of determining what is the best shape and placement of a thermometer in $\Omega$.

At this stage, a first challenge is to settle the problem properly, to make it both mathematically meaningful and relevant in view of practical issues.

Throughout the paper, we fix a real number $L\in(0,1)$, and we will work in a class of domains $\omega$ such that $|\omega|=L|\Omega|$. In other words the set of unknowns is
\begin{equation*}%\label{defUL_intro}
\mathcal{U}_L = \{ \chi_\omega \in L^\infty(\Omega;\{0,1\}) \ \vert\ \omega\ \textrm{is a measurable subset of}\ \Omega \ \textrm{of Lebesgue measure}\ \vert\omega\vert=L\vert\Omega\vert\}.
\end{equation*}
This is done to model the fact that the quantity of sensors to be employed is limited and, hence, that we cannot measure the solution over $\Omega$ in its whole.

We stress again that we do not make any restriction on the regularity or shape of the subsets $\omega$. We are trying to determine whether or not there exists an "absolute" optimal observation domain.
We will see that such a domain exists in the parabolic case under slight assumptions on the operator and on the domain $\Omega$ (in contrast to the case of hyperbolic equations studied in \cite{PTZobsND}).

\medskip

Let us now define the observability problem under consideration.

Recall that, for a given measurable subset $\omega$ of $\Omega$, the heat equation \eqref{heatEq_intro} is said to be \textit{observable} on $\omega$ in time $T$ whenever there exists $C>0$ such that
\begin{equation}\label{ineqobs_intro}
C \int_\Omega y(T,x)^2\, dx
\leq \int_0^T\int_\omega y(t,x)^2 \,dx \, dt,
\end{equation}
for every solution of \eqref{heatEq_intro} such that $y(0,\cdot)\in \mathcal{D}(\Omega)$ (the set of functions defined on $\Omega$, that are smooth and of compact support). It is well known that, if $\Omega$ is $C^2$, then this observability inequality holds true (see \cite{EFC,FursikovImanuvilov,LebeauRobbiano,TucsnakWeiss}). Note that this result has been recently extended in \cite{AEWZ} to the case where $\Omega$ is bounded Lipschitz and locally star-shaped.

The observability constant $C_T(\chi_\omega)$ is defined as the largest possible constant $C>0$ such that \eqref{ineqobs_intro} holds. 
%, that is,
%\begin{equation}\label{defCT_intro}
%C_T(\chi_\omega)=\inf\left\{ \frac{ \int_0^T\int_\omega y(t,x)^2\,dx \, dt }{\int_\Omega y(T,x)^2\, dx} \ \big\vert\  y(0,\cdot)\in \mathcal{D}(\Omega) \setminus\{0\} \right\}.
%\end{equation}
This constant gives an account for the well-posedness of the inverse problem of reconstructing the solutions from measurements over $[0,T]\times\omega$ (see, e.g., the textbook \cite{Chou} for such inverse problems). Of course, the larger the constant $C_T(\chi_\omega)$ is, the more stable  the inverse problem will be.

 Hence it is  natural to model the problem of best observation for the heat equation \eqref{heatEq_intro} as the problem of maximizing the functional $C_T(\chi_\omega)$ over the set $\mathcal{U}_L$, that is,
\begin{equation}\label{supCT_intro}
\sup_{\chi_\omega\in\mathcal{U}_L} C_T(\chi_\omega).
\end{equation}
Such a problem is however very difficult due to the presence of crossed terms at the right-hand side of \eqref{ineqobs_intro} when considering spectral expansions (see Section \ref{sec2.1} for details). On the other hand, actually, the observability constant $C_T(\chi_\omega)$ %defined by \eqref{defCT_intro} 
is (by nature) pessimistic in the sense that it corresponds to a worst possible case, and in practice it is expected that the worst case will not occur very often. 
In practice, to reconstruct solutions one is often led to achieve a large number of measurements, and in the problem of finding a best observation domain it is reasonable to design a set that will optimize the observability only in average. 

In view of that, we define an \textit{averaged version} of the observability inequality, where the average runs over \textit{random initial data}. This procedure, described in detail in Section \ref{sec2.1}, consists of randomizing the Fourier coefficients of the initial data. To explain it with few words, let us fix an orthonormal Hilbert basis $(\phi_j)_{j\in\N^*}$ of $L^2(\Omega)$ consisting of eigenfunctions of the (negative of) Dirichlet-Laplacian associated with the positive eigenvalues $(\lambda_j)_{j\in\N^*}$, with $\lambda_1\leq\cdots\leq \lambda_j\rightarrow+\infty$.
Every solution of \eqref{heatEq_intro} can be expanded as
$$y(t,x)=\sum_{j=1}^{+\infty}a_je^{-\lambda_jt}\phi_j (x),$$
We randomize the solutions (actually, their initial data) by considering
$$y_\nu(t,x)=\sum_{j=1}^{+\infty}\beta_j^\nu  a_j e^{-\lambda_jt}\phi_j(x),$$
for every event $\nu\in\mathcal{X}$, where $(\beta_{j}^\nu)_{j\in\N^*}$ is a sequence of independent real random variables on a probability space $(\mathcal{X},\mathcal{A},\mathbb{P})$ having mean equal to $0$, variance equal to $1$, and a super exponential decay (for instance, Bernoulli laws).
The randomized version of the observability inequality \eqref{ineqobs_intro} is then defined as
$$C_{T,\textrm{rand}}(\chi_\omega) \int_\Omega y(T,x)\, dx \ \leq\ \mathbb{E} \int_0^T\int_\omega y_\nu(t,x)^2 \,dx\,  dt ,$$
where the expectation $\mathbb{E}$ ranges over the space $\mathcal{X}$ with respect to the probability measure $\mathbb{P}$. Here, $C_{T,\textrm{rand}}(\chi_\omega)$ is defined as the largest possible constant such that this randomized observability inequality holds, and is called \textit{randomized observability constant}.
It is easy to establish that
\begin{equation}\label{CTrand_intro}
C_{T,\textnormal{rand}}(\chi_\omega)= \inf_{j\in\N^*} \frac{e^{2\lambda_jT}-1}{2\lambda_j} \int_\omega \phi_j(x)^2 \, dx,
\end{equation}
for every measurable subset $\omega$ of $\Omega$. Moreover, note that $0\leq C_{T,\textnormal{rand}}(\chi_\omega)\leq C_T(\chi_\omega)$ (and the second inequality may be strict, as we will see further).

\medskip

Following the previous discussion, instead of considering as a criterion the deterministic observability constant $C_T(\chi_\omega)$ (and then, the problem \eqref{supCT_intro}), we find more relevant to model the problem of best observation domain as the problem of maximizing the functional $C_{T,\textrm{rand}}(\chi_\omega)$ over the set $\mathcal{U}_L$, that is the problem
\begin{equation}\label{pb_intro}
\sup_{\chi_\omega\in\mathcal{U}_L} C_{T,\textnormal{rand}}(\chi_\omega)
= \sup_{\chi_\omega\in\mathcal{U}_L} \inf_{j\in\N^*} \frac{e^{2\lambda_jT}-1}{2\lambda_j} \int_\omega \phi_j(x)^2 \, dx .
\end{equation}
This spectral model is discussed and settled in a more general parabolic framework in %Sections \ref{sec_framework} and 
Section \ref{sec2.1}.
As a particular case of our main results established in Section \ref{sec_poser}, we have the following result for the heat equation \eqref{heatEq_intro} with homogeneous Dirichlet boundary conditions.

\begin{thmnonumbering}%(Theorem \ref{propLBSC1}) %\label{thm1}
Let $T>0$ arbitrary. Assume that $\partial\Omega$ is piecewise $C^1$.
There exists a unique\footnote{Here, it is understood that the optimal set $\omega^*$ is unique within the class of all measurable subsets of $\Omega$ quotiented by the set of all measurable subsets of $\Omega$ of zero measure.} optimal observation measurable set $\omega^*$, solution of \eqref{pb_intro}. Moreover:
\begin{itemize}
\item $C_{T}(\chi_{\omega^*})<C_{T,\textrm{rand}}(\chi_{\omega^*})$.
\item The optimal set $\omega^*$ is open and semi-analytic. In particular, it has a finite number of connected components and $\vert\partial\omega^*\vert=0$.
\item The optimal set $\omega^*$ is completely characterized from a finite-dimensional spectral approximation, by keeping only a finite number of modes. More precisely, for every $N\in \N^*$, there exists a unique measurable set $\omega^N$ such that $\chi_{\omega^N}\in\mathcal{U}_L$ maximizes the functional
$$
\chi_\omega \longmapsto \inf_{1\leq j\leq N} \frac{e^{2\lambda_jT}-1}{2\lambda_j} \int_{\omega} \phi_j(x)^2\, dx
$$
over $\mathcal{U}_L$. Moreover $\omega^N$ is open and semi-analytic. Furthermore, the sequence of optimal sets $\omega^N$ is stationary, and there exists $N_0\in\N^*$ such that $\omega^N=\omega^*$ for every $N\geq N_0$.
The stationarity integer $N_0$ decreases as $T$ increases and $N_0=1$  whenever $T$ is large enough. In that case, the optimal shape is completely determined by the first eigenfunction.
\end{itemize}
\end{thmnonumbering}

A more general result (Theorem \ref{mainTheo}) will be established in a general parabolic framework.
In the case of the heat equation, one of the important ingredients of the proof is a fine lower bound estimate (stated in \cite{AEWZ}) of the spectral quantities $\int_{\omega}\phi_{j}(x)^2\, dx$, which is uniform over measurable subsets $\omega$ of a given measure.

Note that this existence and uniqueness result holds for every orthonormal basis of eigenfunctions of the Dirichlet-Laplacian, but the optimal set depends, in principle, on the specific choice of the basis. Of course, for $T>0$ large enough, the optimal set is independent of the basis since it is completely determined by the first eigenfunction.

%It is remarkable that the optimal observation set $\omega^*$ can be recovered from a finite number of modes.
%In particular, this implies that \textit{the optimal domain $\omega^*$ is open and semi-analytic} as well and hence has a \textit{finite number of connected components} and the measure of its boundary is zero (Jordan measurable set). 
These properties, stated here for the heat equation \eqref{heatEq_intro} (and proved more generally for parabolic equations under an appropriate spectral assumption, see further) are in strong contrast with the results of \cite{PTZObs1,PTZobspb1,PTZobsND} established for conservative wave and Schr\"odinger equations. In that context of wave-like equations it was proved that:
\begin{itemize}
\item when considering the problem with fixed initial data, the optimal set could be of Cantor type (hence, $\vert\partial\omega\vert>0$) even for smooth initial data;
\item the corresponding randomized observability constant is equal to
$\inf_{j\in\N^*} \int_\omega \phi_j(x)^2\, dx$, and, with respect to \eqref{CTrand_intro}, the evident difference is that all weights are equal to $1$. This is not surprising in view of the conservative properties of the wave or Schr\"odinger equation, however the fact that all frequencies have the same weight causes a strong instability of the optimal sets $\omega^N$ (maximizers of the corresponding spectral approximation). It was proved in \cite{henrot_hebrardSICON,PTZObs1} that the best possible set $\omega^N$ for $N$ modes is actually the worst possible one when considering $N+1$ modes (\textit{spillover} phenomenon).
\end{itemize}

In contrast, for the parabolic problems under consideration, we prove that this instability phenomenon does not occur, and that the sequence of maximizers $\omega^N$ is constant for $N$ large enough, equal to the optimal set $\omega^*$.
This stationarity property is of particular interest in view of designing the best observation set $\omega^*$ in practice.

\medskip

In Section \ref{sec_poser} we provide more details on these results, and state them in a far more general setting, involving in particular the Stokes equation and anomalous diffusion equations (with fractional Laplacian). 
For the Stokes equation
\begin{equation}\label{Stokesequation}
\partial_t y -\triangle y + \nabla p=0,\quad\mathrm{div}\, y=0,
\end{equation}
considered on the unit disk with Dirichlet boundary conditions, we establish that there exists a unique optimal observation set in $\mathcal{U}_L$, sharing nice regularity properties as above.

\smallskip

Let us mention a striking feature occuring for the anomalous diffusion equation
\begin{equation}\label{anomalouseq}
\partial_t y + (-\triangle)^\alpha y=0 ,
\end{equation}
considered on some domain $\Omega$, where $(-\triangle)^\alpha$ is some positive power of the Dirichlet-Laplacian.
Note that such equations are well recognized as being relevant models in many problems encountered in physics (plasma with slow or fast diffusion, aperiodic crystals, spins, etc), in biomathematics, in economy, also in imaging sciences (see for instance \cite{Metzler,Miller,Sokolov}). Hence they provide an important class of parabolic equations entering into the general framework developed in the paper.

Given $T>0$ arbitrary, we prove that if $\partial\Omega$ is piecewise $C^1$ and if $\alpha>1/2$ (or if $\alpha=1/2$ and $T$ is large enough) then there exists a unique optimal observation domain, independently on the Hilbert basis of eigenfunctions under consideration.
Furthermore, we prove the unexpected facts that:
\begin{itemize}
\item in the Euclidean square $\Omega=(0,\pi)^2$, when considering the usual Hilbert basis of eigenfunctions consisting of products of sine functions, for every $\alpha>0$ there exists a unique optimal set in $\mathcal{U}_{L}$ (as in the theorem), which is moreover open and semi-analytic and thus has a finite number of connected components (and this, whatever the value of $\alpha>0$ may be);
\item in the Euclidean disk $\Omega=\{ x\in\R^2\ \vert\ \Vert x\Vert< 1\}$, when considering the usual Hilbert basis of eigenfunctions parametrized in terms of Bessel functions, for every $\alpha>0$ there exists a unique optimal set $\omega^*$ (as in the theorem), which is moreover open, radial, with the following  additional property:
\begin{itemize}
\item if $\alpha>1/2$ then $\omega^*$ consists of a finite number of concentric rings that are at a positive distance from the boundary;
\item if $\alpha<1/2$ (or if $\alpha=1/2$ and $T$ is small enough) then $\omega^*$ consists of an infinite number of concentric rings accumulating at the boundary!
\end{itemize}
\end{itemize}
This surprising result shows that the complexity of the optimal shape does not only depend on the operator but also on the geometry of the domain $\Omega$.

It must be underlined that the proof of these properties (done in Section \ref{ex4}) is lengthy and particularly difficult in the case $\alpha<1/2$. It requires the development of very fine estimates for Bessel functions, combined with the use of quantum limits (semi-classical measures) in the disk, nontrivial minimax arguments and analyticity considerations.

\medskip

Several numerical simulations based on the spectral approximation described previously are provided in Section \ref{sec_anomalous}.
They show in particular what is the optimal shape and location of a thermometer in a square or in a disk.

\medskip

The paper is structured as follows.

Section \ref{sec:obsCst} is devoted to model and solve the problem of finding a best observation domain for parabolic equations. The model is discussed and defined in Section \ref{sec2.1}, based on the introduction of the randomized observability inequality. The problem is solved in a general parabolic setting in Section \ref{sec_poser}, where it is shown that, under an appropriate spectral assumption, there exists a unique optimal observation set, which can moreover be recovered from a finite dimensional spectral approximation problem. 
Section \ref{sec_Stokes} is devoted to the application to the Stokes equation on the unit disk. In Section \ref{sec_anomalous}, we study the case of anomalous diffusion equations and then we provide several numerical simulations illustrating our results and in particular the stationarity feature of the sequence of optimal sets.
Further comments  on the spectral assumption are presented in Section \ref{sec_LBSC}, from a semi-classical analysis viewpoint. 

All results are proved in Section \ref{sec_proofs}.
It must be underlined that the proof concerning the anomalous diffusion equations, in particular in the case $\alpha<1/2$, is long and very technical. It is actually unexpectedly difficult. The proof concerning the Stokes equation is as well for a large part based on facts derived in the previous proof.

Section \ref{sec_ccl} provides a conclusion and several further comments and open problems.

\section{Optimal sensor shape and location / optimal observability}\label{sec:obsCst}
%\subsection{General framework}\label{sec_framework}
Let $\Omega$ be an open bounded connected subset of $\R^n$.
Throughout the paper we consider the problem of determining the optimal observation domain for the abstract parabolic model
\begin{equation}\label{heatEq}
\partial_t y+A_{0} y=0,
\end{equation}
where $A_{0}:D(A_0)\rightarrow L^2(\Omega,\C)$ be a densely defined operator. Precise assumptions on $A_0$ will be done further.
As the main reference, we can keep in mind the typical example of the heat equation with Dirichlet boundary conditions overviewed in the introduction. But our analysis and results will be established for a large class of parabolic operators.

At this stage all what we need to assume, in order to establish the model that we will study, is that there exists a normalized 
%\red{Schauder\footnote{\red{In our classes of examples, this will be either a Hilbert basis or a Riesz basis. We refer the reader to \cite{Gohberg} for sufficient conditions ensuring that there exists a Schauder basis of eigenfunctions.}}}
Hilbert basis $(\phi_j)_{j\in\N^*}$ of $L^2(\Omega,\C)$ consisting of (complex-valued) eigenfunctions of $A_{0}$, associated with the (complex) eigenvalues $(\lambda_j)_{j\in\N^*}$.

\subsection{The model}\label{sec2.1}
The aim of this section is to introduce and define a relevant mathematical model of the problem of best observation.
The first ingredient is the notion of observability inequality.

\paragraph{Observability inequality.}
For every $y^0\in D(A_{0})$, there exists a unique solution $y\in C^0(0,T;D(A_{0}))\cap C^1(0,T;L^2(\Omega))$ of \eqref{heatEq} such that $y(0,\cdot)=y^0(\cdot)$.
For every measurable subset $\omega$ of $\Omega$, the equation \eqref{heatEq} is said to be \textit{observable} on $\omega$ in time $T$ if there exists $C>0$ such that
\begin{equation}\label{ineqobs}
C \Vert y(T,\cdot)\Vert_{L^2(\Omega)}^2
\leq \int_0^T\int_\omega |y(t,x)|^2 \,dx \, dt,
\end{equation}
for every solution of \eqref{heatEq} such that $y(0,\cdot)\in D(A_{0})$.
This inequality is called \textit{observability inequality}, and the constant defined by
\begin{equation}\label{defCT}
C_T(\chi_\omega)=\inf\left\{ \frac{ \int_0^T\int_\omega |y(t,x)|^2\,dx \, dt }{\Vert y(T,\cdot)\Vert^2_{L^2(\Omega)}} \ \big\vert\  y^0\in D(A_{0}) \setminus\{0\} \right\},
\end{equation}
is called the \textit{observability constant}. It is the largest possible nonnegative constant for which \eqref{ineqobs} holds. In other words, the equation \eqref{heatEq} is observable on $\omega$ in time $T$ if and only if $C_T(\chi_\omega)>0$.

\begin{remark}
It is well known that, if $A_{0}$ is the negative of the Dirichlet, or Neumann, or Robin Laplacian, 
then the equation \eqref{heatEq} is observable (see \cite{EFC,FursikovImanuvilov,LebeauRobbiano,TucsnakWeiss}), for every open subset $\omega$ of $\Omega$.
The observability property holds as well, e.g., for the linearized Cahn-Hilliard operator corresponding to $\Omega\subset \R^n$, $A_{0}=(-\triangle)^2$, with the boundary conditions $y_{\vert\partial\Omega}=\triangle y_{\vert\partial\Omega}=0$
(see \cite{TucsnakWeiss}). For the Stokes operator, the observability property follows from \cite[Lemma 1]{FCGIP}.\footnote{More precisely, in order to derive the usual observability inequality from the Carleman estimate proved in this reference, it suffices to estimate from below the left-hand side weight on $[T/4,3T/4]$, to estimate from above the right-hand weight, and to use the fact that the function $t\mapsto\Vert y(t,\cdot) \Vert_{L^2}$ is nonincreasing.}
\end{remark}

As explained in the introduction, throughout the paper we fix a real number $L\in(0,1)$ and we will search an optimal domain in the set
\begin{equation}\label{defUL}
\mathcal{U}_L = \{ \chi_\omega \in L^\infty(\Omega;\{0,1\}) \ \vert\ \omega\ \textrm{is a measurable subset of}\ \Omega \ \textrm{of Lebesgue measure}\ \vert\omega\vert=L\vert\Omega\vert\}.
\end{equation}
This gives an account for the fact that we can measure the solutions only over a part of the whole domain $\Omega$.

Having in mind the observability inequality \eqref{ineqobs}, it is a priori natural to model the question of the optimal location of sensors in terms of maximizing the observability constant $C_T(\chi_\omega)$ over the set $\mathcal{U}_L$ defined by \eqref{defUL}, where $T>0$ is  fixed. Actually, when implementing a reconstruction method, the observability constant $C_T(\chi_\omega)$ gives an account for the well-posedness of the corresponding inverse problem. More precisely, the larger the observability constant is, and the better conditioned the inverse problem is.

However at this stage two remarks are in order.

Firstly, settled as such, the problem is difficult to handle, due to the presence of crossed terms at the right-and side of \eqref{ineqobs} when considering spectral expansions. This problem, which has been discussed thoroughly in \cite{PTZObs1,PTZobsND}, is quite similar to the open problem of determining the best constants in Ingham’s inequalities (see \cite{Ingham,JaffardTucsnakZuazua}). Here, one is faced with the problem of determining the infimum of eigenvalues of an infinite dimensional symmetric nonnegative matrix (namely, the Gramian, see below). Although this criterion has a clear sense, it leads to an optimal design problem which does not seem to be easily tractable.

Secondly, even though the problem of maximizing the observability constant seems natural at the first glance, it is actually not so relevant with respect to the practical issues that we have in mind. Indeed in practice one is led to deal with a large number of solutions: when implementing a reconstruction process, one has to carry out in general a very large number of measures; likewise, when implementing a control procedure, the control strategy is expected to be efficient in general, but maybe not exactly for all cases. The issue that we raise here is the fact that the above observability inequality \eqref{ineqobs} is \textit{deterministic}, and thus the observability constant $C_T(\chi_\omega)$ is pessimistic since it corresponds to a worst possible case. It is likely that in practice this worst case will not occur very often, and hence the deterministic observability constant is not a relevant criterion when realizing a large number of experiments. Instead of that, we are going to propose an averaged version of the observability constant, better suited to our purposes, and defined in terms of probabilistic arguments.

We next describe this procedure, inspired by \cite{BurqTzvetkov1} and which has been used as well in \cite{PTZobsND} to deal with wave and Schr\"odinger equations. 

\paragraph{Randomized observability inequality.}
Every solution $y(\cdot)$ of \eqref{heatEq} such that $y(0,\cdot)=y_0(\cdot)$ can be expanded as
\begin{equation}\label{yDecomp}
y(t,x)=\sum_{j=1}^{+\infty}a_je^{-\lambda_jt}\phi_j (x),
\end{equation}
where 
\begin{equation}\label{defajbj}
a_j = \int_\Omega y^0(x) \overline{\phi}_j (x)\, dx,
\end{equation}
for every $j\in\N^*$.
Using this spectral decomposition, %it follows that
%\begin{equation}\label{GT1}
%\int_0^T\int_\omega |y(t,x)|^2\,dx \, dt
%= \int_0^T  \int_\omega\Big\vert\sum_{j=1}^{+\infty}a_je^{-\lambda_jt}\phi_j(x)  \Big\vert^2 dx \, dt 
%= \sum_{j,k=1}^{+\infty}\alpha_{jk} \int_\omega \phi_j(x) \overline{\phi}_k(x) \, dx
%\end{equation}
%where
%\begin{equation}\label{defalphaij}
%\alpha_{jk} = a_j\bar{a}_k\int_0^T e^{-(\lambda_j+\bar{\lambda}_k)t} \, dt =\frac{a_j\bar{a}_k}{\lambda_j+\bar{\lambda}_k}(1-e^{-(\lambda_j+\bar{\lambda}_k)T}).
%\end{equation}
%Therefore,
%$$
%C_T(\chi_\omega)=\inf_{(a_j)_{j\in\N^*}\in\ell^2(\C)\setminus \{0\}}\frac{\int_0^T  \int_\omega\left|\sum_{j=1}^{+\infty}a_je^{-\lambda_jt}\phi_j(x)  \right|^2 dx \, dt}{\int_\Omega \left|\sum_{j=1}^{+\infty}a_je^{-\lambda_jT}\phi_j(x)\right|^2 dx},
%$$
%and using 
the change of variable $b_j=a_je^{-\lambda_jT}$ and an easy density argument, we get
\begin{equation}\label{defCTdeterministic}
C_T(\chi_\omega)=\inf_{\sum_{j=1}^{+\infty}|b_j|^2=1}\ \  \int_0^T  \int_\omega\left|\sum_{j=1}^{+\infty}b_je^{\lambda_j t}\phi_j(x)  \right|^2 dx \, dt.
\end{equation}
As briefly explained previously, $C_T(\chi_\omega)$ appears as the infimum of the eigenvalues of a Gramian operator, which is the infinite-dimensional Hermitian nonnegative matrix
\begin{equation}\label{defGT}
G_T(\chi_\omega) = \left(\frac{e^{(\lambda_{j}+\bar{\lambda}_k)T}-1}{\lambda_{j}+\bar{\lambda}_k}\int_{\omega}\phi_{j}(x)\overline{\phi}_k(x)\, dx\right)_{ j,k\geq 1}.
\end{equation}
Due to the crossed terms appearing when expanding the square in \eqref{defCTdeterministic}, the resulting optimal design problem, consisting of maximizing $C_T(\chi_\omega)$ over the set $\mathcal{U}_L$, is not easily tractable, at least in view of deriving theoretical results.
Moreover, from the practical point of view the problem of modeling the best observation has to be done, having in mind that the best observation domain should be designed to be the best possible in average, that is, over a large number of experiments. The observability constant $C_T(\chi_\omega)$ above is by definition deterministic, and thus pessimistic in the sense that is gives an account for the worst possible case. In practice, when carrying out a large number of experiments, it can however be expected that the worst possible case does not occur very often.
Having this remark in mind, we next define a new notion of observability inequality by considering an \textit{average over random initial data}. We then define below a notion of \textit{randomized observability constant}, which is in our view better suited to the model of best observation.
We follow \cite{PTZobsND}, accordingly to early ideas developed in \cite{PaleyZygmund} for harmonic analysis issues and recently in \cite{Burq,BurqTzvetkov1} in view of ensuring the probabilistic well-posedness of classically ill-posed supercritical wave or Schr\"odinger equations.

For any given $y^0\in D(A_{0})$, the Fourier coefficients of $y^0$, defined by \eqref{defajbj}, are randomized by defining $a_j^\nu=\beta^\nu_{j}a_j$ for every $j\in\N^*$, where $(\beta_{j}^\nu)_{j\in\N^*}$ is a sequence of independent real random variables on a probability space $(\mathcal{X},\mathcal{F},\mathbb{P})$ having mean equal to $0$, variance equal to $1$, and a super exponential decay (for instance, independent Bernoulli random variables, see \cite{Burq,BurqTzvetkov1} for more details on randomization possibilities and properties).
For every $\nu\in\mathcal{X}$, the solution corresponding to the initial data $y^0_\nu = \sum_{j=1}^{+\infty}\beta_j^\nu a_j \phi_j$ is then
$y_\nu(t,\cdot)=\sum_{j=1}^{+\infty}\beta_j^\nu  a_j e^{-\lambda_jt}\phi_j(\cdot)$.
Instead of considering the deterministic observability inequality \eqref{ineqobs}, we define the \textit{randomized observability inequality} by
\begin{equation}\label{ineqobsrand}
C_{T,\textrm{rand}}(\chi_\omega) \Vert y(T,\cdot)\Vert_{L^2(\Omega)}^2
\ \leq\ \mathbb{E} \int_0^T\int_\omega |y_\nu(t,x)|^2 \,dx\,  dt ,
\end{equation}
for every solution $y$ of \eqref{heatEq} such that $y(0)\in D(A_0)$, where $\mathbb{E}$ is the expectation over the space $\mathcal{X}$ with respect to the probability measure $\mathbb{P}$. The nonnegative constant $C_{T,\textrm{rand}}(\chi_\omega)$ is called \textit{randomized observability constant} and is defined (by density) by
\begin{equation}\label{defCTrandomized}
C_{T,\textrm{rand}}(\chi_\omega)
=
\inf_{\sum_{j=1}^{+\infty} |b_j|^2=1}\ 
\mathbb{E} \int_0^T  \int_\omega \Big\vert\sum_{j=1}^{+\infty}\beta_j^\nu  b_j e^{\lambda_j t}\phi_j(x)  \Big\vert^2 dx \, dt  .
\end{equation}
It is the randomized counterpart of the deterministic constant $C_T(\chi_\omega)$ defined by \eqref{defCTdeterministic}.
Note that
\begin{equation}\label{ineqrand}
0\leq C_T(\chi_\omega)\leq C_{T,\textnormal{rand}}(\chi_\omega),
\end{equation}
for every measurable subset $\omega$ of $\Omega$. The inequalities can be strict (see Theorem \ref{mainTheo} further).

\begin{proposition}\label{lemma:random}
Let $T>0$ arbitrary. For every measurable subset $\omega$ of $\Omega$, we have
$$
C_{T,\textnormal{rand}}(\chi_\omega)= \inf_{j\in\N^*}\gamma_j(T)\int_\omega |\phi_j(x)|^2 \, dx,
$$
with
\begin{equation}\label{defgammaj}
\gamma_j(T)=
\left\{ \begin{array}{lll}
\displaystyle\frac{e^{2\Real(\lambda_j)T}-1}{2\Real(\lambda_j)} & \textrm{if} & \Real(\lambda_j)\neq 0, \\
T & \textrm{if} & \Real(\lambda_j)=0 .
\end{array}\right.
\end{equation}
\end{proposition}

\begin{proof}
Using the Fubini theorem and the independence of the random laws, one has
\begin{eqnarray*}
C_{T,\textrm{rand}}(\chi_\omega) & = & \inf_{\sum_{j=1}^{+\infty}|b_j|^2=1}\ \int_0^T\int_\omega \sum_{j,k=1}^{+\infty}\mathbb{E}(\beta_j^\nu \beta_j^\nu)b_j\bar{b}_k e^{(\lambda_j+\bar{\lambda}_k)t}\phi_j(x)\overline{\phi}_k(x)\, dx \, dt\\
  & = & \inf_{\sum_{j=1}^{+\infty}|b_j|^2=1}\ \sum_{j=1}^{+\infty}|b_j|^2\int_0^T e^{2\Real(\lambda_j)t}\, dt \int_\omega |\phi_j(x)|^2\, dx,
\end{eqnarray*}
and the conclusion follows easily.
\end{proof}

This result clearly shows how the randomization procedure rules out the off-diagonal terms in the Gramian \eqref{defGT}.

\paragraph{Conclusion: the optimal shape design problem.}
For every measurable subset $\omega$ of $\Omega$, we set
\begin{equation}\label{defJ}
J(\chi_\omega)= C_{T,\textrm{rand}}(\chi_\omega) = \inf_{j\in\N^*}\gamma_j(T) \int_\omega \vert\phi_j(x)\vert^2 \, dx,
\end{equation}
Throughout the paper, we will consider the problem of maximizing the functional $J$ over the set $\mathcal{U}_L$ defined by \eqref{defUL}, where the coefficients $\gamma_j(T)$ are defined by \eqref{defgammaj}.
In other words, we consider the problem
\begin{equation}\label{defpb}
\boxed{
\sup_{\chi_\omega\in\mathcal{U}_L} J(\chi_\omega) = \sup_{\chi_\omega\in\mathcal{U}_L} \inf_{j\in\N^*}\gamma_j(T)\int_\omega \vert\phi_j(x)\vert^2 \, dx.
}
\end{equation}
According to the previous discussion, this optimal shape design problem models the best sensor shape and location problem for the parabolic equation \eqref{heatEq}. 

The functional $J$ defined by \eqref{defJ} corresponds to an energy concentration measure.
As we will see, solving this problem requires spectral assumptions.

\subsection{The main result}\label{sec_poser}
In our main result below, it will be useful to consider the functional $J_N$ defined by
\begin{equation}\label{defJN}
J_N(\chi_\omega) = \inf_{1\leq j\leq N} \gamma_j(T) \int_{\omega} \vert\phi_j(x)\vert^2\, dx,
\end{equation}
for every measurable subset $\omega$ of $\Omega$, for every $N\in\N^*$.
The functional $J_N$ is the spectral truncation of the functional $J$ to the $N$ first terms.
We consider as well the shape optimization problem
\begin{equation}\label{pb_max_JN}
\sup_{\chi_\omega\in {\mathcal{U}}_L} J_N(\chi_\omega),
\end{equation}
which is a spectral approximation of the problem \eqref{defpb}. We call it the truncated problem.

Let us now provide the general parabolic framework and the required spectral assumptions.

\paragraph{Framework and assumptions.}
Let $\Omega$ be an open bounded connected subset of $\R^n$, and let $L\in(0,1)$ and $T>0$ be arbitrary.
Let $A_{0}:D(A_0)\rightarrow L^2(\Omega,\C)$ be a densely defined operator, generating a strongly continuous semigroup on $L^2(\Omega,\C)$.
We assume that there exists a Hilbert basis $(\phi_j)_{j\in\N^*}$ of $L^2(\Omega,\C)$ consisting of (complex-valued) eigenfunctions of $A_{0}$, associated with (complex) eigenvalues $(\lambda_j)_{j\in\N^*}$ such that $\Real(\lambda_1)\leq\cdots\leq \Real(\lambda_j)\leq\cdots$, and such that the following assumptions are satisfied:
\begin{itemize}
%\item[$\Hun$] $\Real(\lambda_1)\leq\cdots\leq \Real(\lambda_j)\rightarrow+\infty$;
\item[$\Hun$] (\textit{Strong Conic Independence Property}) If there exists a subset $E$ of $\Omega$ of positive Lebesgue measure, an integer $N\in\N^*$, a $N$-tuple $(\alpha_{j})_{1\leq j\leq N}\in (\R_{+})^N$, and $C\geq 0$ such that $\sum_{j=1}^N \alpha_j \vert\phi_j(x)\vert^2 = C$ almost everywhere on $E$, then there must hold $C = 0$ and $\alpha_j = 0$ for every $j\in \{1,\cdots,N\}$.
\item[$\Hdeux$] For every $a\in L^\infty(\Omega;[0,1])$ such that $\int_{\Omega}a(x)\, dx=L|\Omega|$, one has
\begin{equation*}%\label{LBSC}
\liminf_{j\rightarrow+\infty} \ \gamma_j(T) \int_\Omega a(x) \vert\phi_j(x)\vert^2\, dx > \gamma_1(T) ;
\end{equation*}
\item[$\Htrois$] The eigenfunctions $\phi_j$ are analytic in $\Omega$.
\end{itemize}

We start with a simple preliminary result for the truncated problem.

\begin{proposition}\label{truncTheo}
Under $\Hun$, for every $N\in\N^*$, the truncated problem \eqref{pb_max_JN} has a unique\footnote{Here and in the sequel, it is understood that the optimal set is unique within the class of all measurable subsets of $\Omega$ quotiented by the set of all measurable subsets of $\Omega$ of zero measure.} solution $\chi_{\omega^N}\in\mathcal{U}_L$.
Moreover, under $\Htrois$, $\omega^N$ is an open semi-analytic\footnote{A subset $\omega$ of a real analytic finite dimensional manifold $M$ is said to be semi-analytic if it can be written in terms of equalities and inequalities of analytic functions.
We recall that such semi-analytic subsets are stratifiable in the sense of Whitney (see \cite{Hardt,Hironaka}), and enjoy local finitetess properties, such that: local finite perimter, local finite number of connected components, etc.} set, and thus, in particular, it has a finite number of connected components.
\end{proposition}

\begin{remark}\label{remark4}
If $A_0$ is defined on a domain $D(A_0)$ such that the eigenfunctions $\phi_j$ vanish on $\partial\Omega$ (Dirichlet boundary conditions), then moreover there exists $\eta^N>0$ such that the (Euclidean) distance between $\omega^N$ and $\partial\Omega$ is larger than $\eta^N$. 
\end{remark}

Our main result is the following.

\begin{theorem}\label{mainTheo}
Under $\Hun$ and $\Hdeux$, the optimal shape design problem \eqref{defpb} has a unique solution $\chi_{\omega^*}\in\mathcal{U}_L$. 

Moreover, there exists a smallest integer $N_0(T)$ such that
\begin{equation*}%\label{cclmaintheo}
J(\chi_{\omega^*})=\max_{\chi_\omega\in\mathcal{U}_L}J(\chi_\omega)=
\max_{\chi_\omega \in\mathcal{U}_L}J_N(\chi_\omega),
\end{equation*}
for every $N\geq N_0(T)$. In other words, the sequence $(\chi_{\omega^N})_{N\in\N^*}$ of maximizers of $J_N$ is stationary, that is, $\omega^*=\omega^{N_0(T)}=\omega^N$ for $N\geq N_0(T)$.
\\
The function $T\mapsto N_{0}(T)$ is nonincreasing, and if $\Real(\lambda_j)\rightarrow+\infty$ as $j\rightarrow+\infty$  then $N_0(T)=1$ whenever $T$ is large enough.

Under the additional assumption $\Htrois$, we have moreover that:
\begin{itemize}
\item $C_T(\chi_{\omega^*})< C_{T,\textnormal{rand}}(\chi_{\omega^*})$;
\item the optimal observation set $\omega^*$ is an open semi-analytic set and thus it has a finite number of connected components. 
\end{itemize}
\end{theorem}

Proposition \ref{truncTheo} and Theorem \ref{mainTheo} are proved in Section \ref{sec_proofs}.

\begin{remark}\label{remvanish}
In the next sections we will comment in detail on the assumptions done in the theorem, and provide classes of examples where they are satisfied (note however that proving their validity is far from obvious): heat and anomalous diffusion equations, Stokes equation.

We can however note, at this stage, that these assumptions are of different natures. 

The assumption $\Hun$ will be treated essentially with analyticity considerations. Indeed note that  $\Hun$ holds true as soon as the eigenfunctions $\phi_j$ are analytic in $\Omega$ (that is, under the assumption $\Htrois$) and vanish along $\partial\Omega$. This is often the case, for instance, for elliptic operators with analytic coefficients.  
It can be noted that a generalization of the property $\Hun$ has been studied for the Dirichlet-Laplacian in \cite{PriSiga}, where the $\alpha_{j}$ are arbitrary real numbers, and is proved to hold generically with respect to the domain $\Omega$.
The validity of $\Hun$ in general (for instance, in for Neumann boundary conditions) is an open problem.

The assumption $\Hdeux$, which can as well be seen from a semi-classical point of view (see comments in Section \ref{sec_LBSC} further) is related with nonconcentration properties of eigenfunctions. For instance proving it for heat-like equations will require the use of fine recent results providing lower bound estimates that are uniform with respect to the observation domain $\omega$.
\end{remark}

Before coming to these applications, several remarks are in order.

\begin{remark}
The fact that the sequence $(\chi_{\omega^N})_{N\in\N^*}$ of optimal sets of the truncated problem \eqref{truncoptdesignpb} is stationary is in strong contrast with the results of \cite{henrot_hebrardSCL, henrot_hebrardSICON, PTZ_HUM1D, PTZObs1, PTZobsND} in which such optimal design problems have been investigated for conservative wave or Schr\"odinger equations. In these references it was observed and proved that the corresponding maximizing sequence of subsets does not converge in general, except in very particular cases. Moreover, in dimension one, this sequence of sets has an instability property known as \textit{spillover phenomenon}. Namely, the best possible set for $N$ modes is actually the worst possible one when considering $N+1$ modes. This instability property has negative consequences in view of practical issues for designing a relevant notion of optimal set. 

In contrast, Theorem \ref{mainTheo} shows that, for the parabolic equation \eqref{heatEq}, the maximizing sequence of subsets is stationary, and hence only a finite number of modes is enough in order to capture all the information necessary to design the true optimal set. In other words, higher modes play no role.
Although this result can appear as intuitive because we are dealing with a parabolic equation, deriving such a property however requires the spectral property $\Hdeux$, which is commented and analyzed further.
\end{remark}

\begin{remark}
The fact that the optimal set $\omega^*$ is semi-analytic is a strong (and desirable) regularity property. In addition to the fact that $\omega^*$ has a finite number of connected components, this implies also that $\omega^*$ is Jordan measurable, that is, $\vert\partial\omega^*\vert=0$.
This is in contrast with the already mentioned fact that, for wave-like equations, when maximizing the energy for \textit{fixed data}, the optimal set may be a Cantor set of positive measure, even for smooth initial data (see \cite{PTZobspb1}).
\end{remark}

\begin{remark}[A convexified formulation of \eqref{defpb}]\label{rem:introrelax}
It is standard in shape optimization to introduce a convexified version of a maximization problem, since it may fail to have some solutions because of hard constraints. This is what is usually referred to as relaxation (see, e.g., \cite{BucurButtazzo}). 

Since the set $\mathcal{U}_L$ (defined by \eqref{defUL}) does not share nice compactness properties, we consider the convex closure of $\mathcal{U}_L$ for the weak star topology of $L^\infty$, which is
\begin{equation}\label{defALbar}
\overline{\mathcal{U}}_L = \left\{ a\in L^\infty(\Omega;[0,1])\ \vert\ \int_{\Omega} a(x)\,dx=L\vert\Omega\vert\right\}.
\end{equation}
Such a relaxation was used as well in \cite{munchHeat,PTZObs1,PTZobsND}.
Replacing $\chi_\omega\in\mathcal{U}_L$ with $a\in\overline{\mathcal{U}}_L$, we define a relaxed formulation of the optimal shape design problem \eqref{defpb} by
\begin{equation}\label{defJa}
\sup_{a\in \overline{\mathcal{U}}_L} J(a) ,
\end{equation}
where the functional $J$ is naturally extended to $\overline{\mathcal{U}}_L$ by
\begin{equation}\label{defJrelax}
J(a)=\inf_{j\in \N^*}\gamma_j(T) \int_{\Omega}a(x)\vert \phi_j(x)\vert^2\, dx,
\end{equation}
for every $a\in \overline{\mathcal{U}}_L$. Moreover, one has the following existence result.

\begin{lemma}\label{propExistRelax}
For every $L\in (0,1)$, the relaxed problem \eqref{defJa} has at least one solution $a^*\in\overline{\mathcal{U}}_L$.
\end{lemma}

\begin{proof}[Proof of Lemma \ref{propExistRelax}]
For every $j\in\N^*$, the functional $a\in\overline{\mathcal{U}}_L\mapsto \gamma_j \int_{\Omega}a(x) \vert\phi_j(x)\vert^2\, dx$ is linear and continuous for the weak star topology of $L^\infty$. Hence $J$ is upper semicontinuous as the infimum of continuous linear functionals. Since $\overline{\mathcal{U}}_L$ is compact for the weak star topology of $L^\infty$, the lemma follows.
\end{proof}

Note that, obviously,
$$
\sup_{\chi_\omega\in {\mathcal{U}}_L} J(\chi_\omega) \leq
\sup_{a\in \overline{\mathcal{U}}_L} J(a) = J(a^*).
$$
But, in fact, from Theorem \ref{mainTheo} (and from its proof) we deduce that the two suprema coincide,  and that the problem \eqref{defpb} and the relaxed problem \eqref{defJa} have the same (unique) solution. This means two things. First, there is no gap between the optimal values of the problem \eqref{defpb} and its relaxed formulation \eqref{defJa}. A similar result was established   in \cite{PTZobsND} for wave and Schr\"odinger like equations under spectral assumptions on the domain $\Omega$. But, in contrast to these hyperbolic equations where relaxation occurs except for some very distinguished discrete values of $L$, here, in the parabolic setting, relaxation does not occur, at least under the assumption $\Hdeux$, which is fulfilled for the Dirichlet-Laplacian for piecewise $C^1$ domains $\Omega$ (see Theorem \ref{propLBSC1} further).

In particular, in the parabolic setting, contrarily to what happens in wave-like equations, the constant function $a=L$ is not an optimal solution.
Note that this constant function corresponds intuitively (at the weak limit) to equi-distribute the sensors over the domain $\Omega$. This strategy is however not optimal for parabolic problems.
\end{remark}

\begin{remark}\label{rem4}
The assumption $\Hdeux$ can be actually weakened (as can be easily seen from the proof of the theorem). To ensure that the conclusion of Theorem \ref{mainTheo} holds, it is sufficient to assume that
\begin{equation}\label{LBSC_weakened}
\liminf_{j\rightarrow+\infty}\ \gamma_j(T)\int_{\Omega} a^*(x)\vert\phi_j(x)\vert^2\, dx >\gamma_1(T),
\end{equation}
where $a^*\in\overline{\mathcal{U}}_L$ is any optimal solution of the relaxed problem \eqref{defJa}. In other words, it is sufficient to restrict the assumption $\Hdeux$ to the sole $a^*$. Note that such an assumption is impossible to check since $a^*$ is not known a priori, but this remark will however be useful in Section \ref{ex4}.

Note that, since $J(L)=L\gamma_1$, it follows that $J(a^*)\geq L\gamma_1$, and hence in particular there always holds
$$
\liminf_{j\rightarrow+\infty}\ \gamma_j(T)\int_{\Omega} a^*(x)\vert\phi_j(x)\vert^2\, dx \geq L\gamma_1(T).
$$
\end{remark}

\begin{remark}
The existence and uniqueness of an optimal set, stated in Theorem \ref{mainTheo}, holds true for any Hilbert basis of eigenfunctions of $A_0$ as soon as this basis satisfies the assumptions $\Hun$, $\Hdeux$ and $\Htrois$. However the optimal set $\omega^*$ may depend on the specific choice of the basis.
\end{remark}

\begin{remark}\label{rem9}
As noted before, the issue of solving the optimal design problem
$$
\sup_{\chi_{\omega}\in\mathcal{U}_{L}}C_{T}(\chi_{\omega})
$$
where $C_{T}(\chi_{\omega})$ is the observability constant of the parabolic equation \eqref{heatEq} defined by \eqref{defCT}, is natural and interesting, although this problem is very difficult to handle from the theoretical point of view, even for the truncated criterion, and not as much relevant as the one we consider here, from the practical point of view (as already discussed).

 Note that the truncated version of the criterion $C_{T,\textnormal{rand}}(\chi_{\omega})$ is the lowest eigenvalue of the diagonal matrix $\textrm{diag}\left(\gamma_{j}(T)\int_{\omega}\vert\phi_{j}(x)\vert^2\, dx\right)_{1\leq j\leq N}$, whereas the truncated version $C_{T,N}(\chi_{\omega})$ of the criterion $C_{T}(\chi_{\omega})$ is the lowest eigenvalue of the Gramian matrix
\begin{equation}\label{defdGTN}
G_{T,N}(\chi_\omega) = \left(\frac{e^{(\lambda_{j}+\bar{\lambda}_k)T}-1}{\lambda_{j}+\bar{\lambda}_k}\int_{\omega}\phi_{j}(x)\overline{\phi}_k(x)\, dx\right)_{1\leq j,k\leq N},
\end{equation}
which is the truncation of the Gramian $G_T(\chi_\omega)$ defined by \eqref{defGT}.
Under the conditions of Theorem \ref{mainTheo}, the sequence of the minimizers over $\mathcal{U}_{L}$ of the truncated version of the randomized constant $C_{T,\textnormal{rand}}(\chi_{\omega})$ is stationary. An interesting problem consists of investigating theoretically or numerically whether this stationarity property holds true or not for the truncated version $C_{T,N}(\chi_{\omega})$ of the observability constant $C_{T}(\chi_{\omega})$. 

Notice that, extending the definition of $C_{T}(\chi_{\omega})$ to the functions $a\in L^\infty(\Omega;[0,1])$ by 
\begin{equation*}\label{defCTa}
C_T(a)=\inf\left\{ \frac{ \int_0^T\int_\Omega a(x)|y(t,x)|^2\,dx \, dt }{\Vert y(T,\cdot)\Vert^2_{L^2(\Omega)}} \ \big\vert\  y^0\in D(A_{0}) \setminus\{0\} \right\},
\end{equation*}
one gets easily that the optimal design problem of maximizing $C_{T}(a)$ over $\overline{\mathcal{U}}_L$ has at least one solution. Furthermore, it is interesting to note that, by adapting the proof of \cite[Proposition 2]{PTZObs1}, we get the following partial result.

\begin{lemma}
For every $L\in (0,1)$ and every $T>0$, the constant function $\overline{a}(\cdot)=L$ is not a maximizer of the functional $a\mapsto C_{T}(a)$ over $\overline{\mathcal{U}}_L$.
\end{lemma}

%This result emphasizes the intrinsic difference between such optimal design problems for the wave and the heat equations. Indeed, in the case of the wave equation, it has been proved in \cite{PTZObs1} that the constant function $\overline{a}(\cdot)=L$ solves the problem of maximizing the observability constant for the one-dimensional wave equation over $\overline{\mathcal{U}}_L$.
\end{remark}

\begin{remark}\label{rem_largetime}
Finally, let us comment on the role of the time $T$. Recall that $T>0$ has been arbitrarily fixed at the beginning of the analysis. Its role is in the weights $\gamma_j(T)$ coming into play in the definition of the functional $J$ (defined by \eqref{defJ}). If the eigenvalues are such that $\Real(\lambda_j)\rightarrow+\infty$, then the larger $T$ is, and the quicker the weights tend to $+\infty$.
As a consequence, as stated in Theorem \ref{mainTheo}, the integer $N_0(T)$ decreases as $T$ increases, and if $T$ is large enough then $N_0(T)=1$.
This says that, if one can observe the solutions of the equation over a large enough horizon of time, then the optimal observation domain can be designed from the first mode only. This fact is in accordance with the strong damping properties of a parabolic equation, at least, under the assumption $\Hdeux$. In large time the energy of the solutions is essentially carried by the first mode.
\end{remark}

%%%%%%%%%%%%%%%%%%%%%%%%%%%%%%%%%%%%%%%%%%%%%
%%%%%%%%%%%%%%%%%%%%%%%%%%%%%%%%%%%%%%%%%%%%%
\subsection{Application to the Stokes equation in the unit disk}\label{sec_Stokes}
In this section, we assume that $\Omega=\{ x\in\R^2\ \vert\ \Vert x\Vert< 1\}$ is the Euclidean unit disk of $\R^2$, and we consider the Stokes equation \eqref{Stokesequation} in the unit disk of $\R^2$, with Dirichlet boundary conditions.

Note that the Stokes system does not exactly enter in the framework defined in Section \ref{sec_poser}, but it suffices to make the following very slight modification.
The Stokes operator $A_0:D(A_0)\rightarrow H$ is defined by $A_0=-\mathcal{P}\triangle$, with $D(A_0)=\{ y\in V \ \vert\ A_{0}y\in H \}$, 
$V = \{ y\in (H^1_0(\Omega))^2\ \vert\ \mathrm{div}\, y=0 \}$,
$H = \{ y\in (L^2(\Omega))^2\ \vert\ \mathrm{div}\, y=0 ,\ y_{\vert\partial\Omega}.n=0  \}$,
and $\mathcal{P}:(L^2(\Omega))^2 = H\overset{\perp}{\oplus}H^\perp\rightarrow H$ is the Leray projection. %Note that the space $H$ is a Hilbert for the norm induced by $(L^2(\Omega))^d$.
Then $A_0$ is an unbounded operator in the Hilbert space $H$ (and not on $L^2$), endowed with the $L^2$-norm (see \cite{Boyer}).

We consider here the Hilbert basis of $H$ of eigenfunctions, indexed by $j\in\Z$, $k\in\N^*$ and $m=1,2$, defined by 
\begin{equation}\label{Stokes_eig0}
\phi_{0,k}(r,\theta) = \frac{-J_0'(\sqrt{\lambda_{0,k}} r) }{ \sqrt{\pi} \sqrt{\lambda_{0,k}} \vert J_0(\sqrt{\lambda_{0,k}})\vert }  \begin{pmatrix} -\sin\theta \\ \cos\theta \end{pmatrix} ,
\end{equation}
and
\begin{equation}\label{Stokes_eig1}
\begin{split}
\phi_{j,k,m}(r,\theta) \ =\ &
\frac{J_j(\sqrt{\lambda_{j,k}} r)-J_j(\sqrt{\lambda_{j,k}})r^j}{\lambda_{j,k}\vert J_j(\sqrt{\lambda_{j,k}})\vert r} j (-1)^{m+1} Y_{j,m}(\theta) \begin{pmatrix} \cos\theta \\ \sin\theta \end{pmatrix} \\
& +
\frac{-\sqrt{\lambda_{j,k}}J_j'(\sqrt{\lambda_{j,k}} r) + j J_j(\sqrt{\lambda_{j,k}}) r^{j-1} }{ \lambda_{j,k} \vert J_j(\sqrt{\lambda_{j,k}})\vert } Y_{j,m+1}(\theta) \begin{pmatrix} -\sin\theta \\ \cos\theta \end{pmatrix}
\end{split}
\end{equation}
whenever $j\neq 0$, where $(r,\theta)$ are the usual polar coordinates (see \cite{Kelliher,LeeRummler}).
The functions $Y_{j,m}(\theta)$ are defined by $Y_{j,1}(\theta)=\frac{1}{\sqrt{\pi}}\cos(j\theta)$ and $Y_{j,2}(\theta)=\frac{1}{\sqrt{\pi}}\sin(j\theta)$, with the agreement that $Y_{j,3}=Y_{j,1}$, and $J_j$ is the Bessel function of the first kind of order $j$. Denoting by $z_{j,k}>0$ is the $k^\textrm{th}$ positive zero of $J_{j}$, the eigenvalues of $A_{0}$ are the doubly indexed sequence $(-\lambda_{j,k})_{j\in\Z,k\in\N^*}$, where $\lambda_{j,k}=z_{\vert j\vert+1,k}^{2}$ is of multiplicity $1$ if $j=0$, and $2$ if $j\neq 0$. 

Note that, in $\Hun$, $\Hdeux$, and in the definition \eqref{defJ} of the functional $J$, we replace $\vert\cdot\vert$ with the Euclidean norm of $\R^2$.

\begin{theorem}\label{theo:stokes}
%Let $L\in (0,1)$. Assume that $\Omega=\{ x\in\R^2\ \vert\ \Vert x\Vert< 1\}$, and consider the Hilbert basis of eigenfunctions given by \eqref{Stokes_eig0}-\eqref{Stokes_eig1}. Then,
The assumptions $\Hun$, $\Hdeux$ and $\Htrois$ are satisfied. Then Theorem \ref{mainTheo} implies that there exists a unique optimal observation domain $\omega^*$ (solution of the problem \eqref{defpb}), which is moreover open and semi-analytic.
\end{theorem}

This result is proved in Section \ref{Sec:proofstokes}.
The proof is technically based on the explicit form of the basis of eigenfunctions under consideration, and we did not investigate what can happen in higher dimension.
Also, what can happen for more general domains is not known.

%\paragraph{Periodic Stokes equation.}
%Consider as in \cite{Temam} the periodic Stokes equation in $\R^n$
%$$ \partial_t y -\triangle y + \nabla p = 0,\quad \mathrm{div}\, y=0,$$
%with the periodicity conditions $y(t,x+\pi e_i)=y(t,x)$, for every $t>0$ and every $x\in\R^n$, where $(e_1,\ldots,e_n)$ is the canonical basis of $\R^n$. Consider the usual orthonormal basis of eigenfunctions of the corresponding Stokes operator in $\Omega=(0,\pi)^n$, given by
%$$
%\phi_{j,m}(x) = \left( e_m - \frac{j_m}{\vert j\vert^2} j\right) \mathrm{e}^{2ij.x},
%$$
%for every multi-index $j=(j_1,\ldots,j_n)\in \N^n$ and every $m\in\{1,\ldots,n\}$, associated with the eigenvalues $\lambda_{j,m}=4\vert j\vert^2$.
%Then the assumption $\Hdeux$ is clearly satisfied and hence Theorem \ref{mainTheo} can be applied as well to this case.

%%%%%%%%%%%%%%%%%%%%%%%%%%%%%%%%%%%%%%%%%%%%%%%%%

\subsection{Application to anomalous diffusion equations}\label{sec_anomalous}
In this section, we assume that $\Omega$ is Lipschitz, and we consider the Dirichlet-Laplacian $\triangle$ defined on its domain $D(\triangle)=\{ y\in H^1_0(\Omega)\ \vert\ \triangle y\in L^2(\Omega) \}$. 
Note that if $\partial\Omega$ is $C^2$ then $D(\triangle)= H^1_0(\Omega)\cap H^2(\Omega)$.

We set $A_0=(-\triangle)^\alpha$ (where $\triangle$ is the Dirichlet-Laplacian), with $\alpha>0$ arbitrary, defined spectrally, based on the spectral decomposition of the Dirichlet-Laplacian. 
This case corresponds to the anomalous diffusion equation \eqref{anomalouseq}. 

To be more precise with the functional framework, the domain of the operator $A_0$, as an unbounded operator in $L^2(\Omega)$, is defined as follows. If $\alpha\in(0,1)\setminus\{1/4\}$, then $D(A_0)=H_0^{2\alpha}(\Omega)$;  if $\alpha=1/4$ then $D(A_0)=H_{00}^{1/2}(\Omega)$ (Lions-Magenes space), and if $1/4<\alpha<1$ then $ D(A_0)= H^1_0(\Omega)\cap H^{2s}(\Omega)$ (see \cite{LM} or \cite[Appendix]{BT}).

For $\alpha >1$ the operator is defined by composing integer powers of $-\triangle$ with the fractional powers above.
For instance one can take $A_{0}=(-\triangle)^2$ with the boundary conditions 
$y_{\vert\partial\Omega}=\triangle y_{\vert\partial\Omega}=0$: in that case \eqref{heatEq} corresponds to a linearized model of Cahn-Hilliard type.

In the general case $\alpha>0$, the equation \eqref{heatEq} models a physical process exhibiting anomalous diffusion (see for instance \cite{Metzler,Miller,Sokolov}).
Of course if $\alpha=1$ then \eqref{heatEq} is the heat equation with Dirichlet boundary conditions, as overviewed in the introduction.

Note that the eigenfunctions of $A_0$ are those of the Dirichlet-Laplacian, and therefore the assumptions $\Hun$ and $\Htrois$ are satisfied. Only the assumption $\Hdeux$ has to be discussed in the sequel.
We have the following three results.

\subsubsection{A general result}

\begin{theorem}\label{propLBSC1}
Assume that $\partial\Omega$ is piecewise\footnote{Actually a more general assumption can be done: $\Omega$ is Lipschitz and locally star-shaped (see \cite{AEWZ} for the definition and details).} $C^1$. If $\alpha>1/2$, then the assumption $\Hdeux$ is satisfied for any Hilbert basis of eigenfunctions of $A_0$.
For $\alpha=1/2$, the conclusion holds true as well provided that $T$ is moreover large enough.

Under these conditions, Theorem \ref{mainTheo} can be applied and implies that there exists a unique optimal set $\omega^*$ (solution of the problem \eqref{defpb}), which is open and semi-analytic.
\end{theorem}

In order to prove that the uniform lower bound assumption $\Hdeux$ holds true, the main ingredient is a lower bound estimate (stated in \cite{AEWZ}) of the spectral quantities $\int_{\omega}\phi_{j}(x)^2\, dx$, which is uniform over measurable subsets $\omega$ of a given measure.

It can be noted that the number $N_0(T)$ of relevant modes needed to compute the optimal set depends on the speed of convergence of $j^{\frac{2\alpha-1}{n}}T$ to $+\infty$ (this follows from the proof of Theorem \ref{propLBSC1}, by using Weyl's asymptotics).

\subsubsection{Case of the $n$-dimensional orthotope}
Assume that $\Omega=(0,\pi)^n$, for $n\in\N^*$. %Consider the usual Hilbert basis of eigenfunctions consisting of products of sine functions. 
We consider the usual Hilbert basis consisting of products of sine eigenfunctions, given by
\begin{equation}\label{basis_orthotope}
\phi_{j_1,\dots,j_n}(x)=\left(\frac{2}{\pi}\right)^{n/2}\prod_{k=1}^n\sin (j_kx_k),
\end{equation}
for all $(j_1,\dots,j_n)\in{\N^*}^n$, with corresponding eigenvalues
$\lambda_{(j_1,\dots,j_n)}=\left(\sum_{k=1}^n j_k^2\right)^\alpha$.

\begin{theorem}\label{propLBSC2}
The assumption $\Hdeux$ is satisfied, whatever the value of $\alpha>0$ may be. Then, Theorem \ref{mainTheo} implies that there exists a unique optimal observation domain $\omega^*$ (solution of the problem \eqref{defpb}), which is open and semi-analytic.
\end{theorem}

Note that it is not clear whether this result is satisfied or not for any Hilbert basis of eigenfunctions, at least for $\alpha<1/2$ (indeed the case $\alpha>1/2$ is solved with Theorem \ref{propLBSC1}).

As shown in the proof of Theorem \ref{propLBSC2}, it can be also noted that the conclusion of Theorem \ref{mainTheo} holds with $N_0(T)$ defined as the lowest multi-index $(j_1,\ldots,j_n)$ (in lexicographical order) such that
$$
\frac{e^{2\lambda_{(j_1,\ldots,j_n)}T}-1}{2\lambda_{(j_1,\dots,j_n)}} \geq \frac{e^{2\lambda_{(1,\ldots,1)}T}-1}{2\lambda_{(1,\ldots,1)} F^{[n]}(L\pi^n)}.
$$
where $F$ is the function defined on $[0,\pi]$ by $F(s)= \frac{1}{\pi}(s-\sin s)$, and $F^{[n]}$ is the composition of $F$ with itself, $n$ times. %In particular, in accordance with Remark \ref{rem_largetime}, the integer $N_0(T)$ decreases as the time $T$ increases.

\begin{remark}\label{rem_Neumann}
Note that the result of Theorem \ref{propLBSC2} holds true as well for the Neumann-Laplacian $A_0=-\triangle$ defined on the domain %(see Remark \ref{sec_Neumann} further for comments on this domain)
$
D(A_0) = %\left
\{y \in H^2(\Omega,\C)\ \vert\ \int_\Omega y = 0\textrm{ and }\frac{\partial y}{\partial n}=0 \textrm{ on }\partial\Omega%\right
\} ,
$
with the usual Hilbert basis of eigenfunctions consisting of products of cosine functions (it is indeed easy to see that the assumption $\Hdeux$ is satisfied).
Fractional operators can be as well defined out of this Neumann-Laplacian.
%(it is not clear if this property holds true in more general domains)
The reason to consider the Neumann-Laplacian on functions that are of zero average (which is standard for observability issues) is due to the fact that, if we do not make this restriction then $\lambda_1=0$ (and $\lambda_j>0$ for every $j\geq 2$) and $\phi_1=1/\sqrt{\vert\Omega\vert}$, and $\Hun$ fails. 
%Let us however investigate what would happen in the analysis of the truncated optimal design problem \eqref{truncoptdesignpb}, by keeping this domain without taking the zero average.
%Since there holds $\gamma_1(T)\int_{\Omega}a(x)|\phi_1(x)|^2\, dx = LT$ for every $a\in\overline{\mathcal{U}}_{L}$, it follows that
%$$
%J_N(a)=\min\left(LT,\inf_{j\geq 2}\gamma_{j}(T)\int_{\Omega}a(x)|\phi_j(x)|^2\, dx\right).
%$$
%Considering the constant function $a=L$, we note that
%$$
%\sup_{a\in\overline{\mathcal{U}}_L}\inf_{j\geq 2}\gamma_{j}(T)\int_{\Omega}a(x)|\phi_j(x)|^2\, dx \geq L\inf_{j\geq 2}\gamma_{j}(T) =L\gamma_2(T)>LT,
%$$
%and therefore,
%$$
%\sup_{a\in\overline{\mathcal{U}}_L} J_N(a) = LT,
%$$
%for every $N\in\N^*$. Note that the constant function $a=L$ is a solution. The conclusion of Proposition \ref{truncTheo} (in particular, the uniqueness of an optimal set) does not hold true anymore, and then the problem is not very interesting in this case.
%
%This is due to the fact that the constant function is a solution of the heat equation with Neumann boundary conditions, and is compatible with the observability inequality \eqref{ineqobs}. Besides, since every solution has a constant average (this follows obviously from the dynamics), observing the constant mode has no interest.
%
%This explains why it is relevant, in the Neumann case, to consider only initial conditions having a zero average.
%Then $\bf (H_3)$ is satisfied (see Section \ref{sec_framework}) and then Proposition \ref{truncTheo} can be applied.
\end{remark}

\subsubsection{Case of the two-dimensional disk}
Assume that $\Omega=\{ x\in\R^2\ \vert\ \Vert x\Vert< 1\}$ is the Euclidean unit disk of $\R^2$. 
We consider the Hilbert basis of eigenfunctions defined by the triply indexed sequence of functions
\begin{equation}\label{basis_disk}
\phi_{j,k,m}(r,\theta) = 
\left\{ \begin{array}{ll}
R_{0,k}(r)/\sqrt{2\pi} & \ \textrm{if}\ j=0,\\
R_{j,k}(r)Y_{j,m}(\theta) & \ \textrm{if}\ j\geq 1,
\end{array} \right.
\end{equation}
for $j\in\N$, $k\in\N^*$ and $m=1,2$, where $(r,\theta)$ are the usual polar coordinates.
The functions $Y_{j,m}(\theta)$ are defined by $Y_{j,1}(\theta)=\frac{1}{\sqrt{\pi}}\cos(j\theta)$ and $Y_{j,2}(\theta)=\frac{1}{\sqrt{\pi}}\sin(j\theta)$, and the functions $R_{j,k}$ are defined by
\begin{equation}\label{def_Rjk}
R_{j,k}(r) = \sqrt{2}\,\frac{J_j(z_{j,k}r)}{\vert J'_{j}(z_{j,k}) \vert} = \frac{J_j(z_{j,k}r)}{\sqrt{\int_0^1 J_j(z_{j,k}r)^2 r\, dr}} ,
\end{equation}
where $J_j$ is the Bessel function of the first kind of order $j$, and $z_{j,k}>0$ is the $k^\textrm{th}$ positive zero of $J_{j}$.
The corresponding eigenvalues consist of a doubly indexed sequence $(-\lambda_{j,k})_{j\in\N,k\in\N^*}$, where $\lambda_{j,k}=z_{j,k}^{2\alpha}$ is of multiplicity $1$ if $j=0$, and $2$ if $j\geq 1$.

\begin{theorem}\label{theoLBSC3}
%Assume that $\Omega=\{ x\in\R^2\ \vert\ \Vert x\Vert< 1\}$ is the Euclidean unit disk of $\R^2$. Consider the usual Hilbert basis of eigenfunctions parametrized in terms of Bessel functions. 
For every $\alpha>0$, the optimal design problem \eqref{defpb} has a unique solution $\chi_{\omega^*}\in\mathcal{U}_L$, where $\omega^*$ is moreover open and radial. Furthermore:
\begin{itemize}
\item If $\alpha>1/2$ then the assumption $\Hdeux$ is satisfied\footnote{This is a consequence of Theorem \ref{propLBSC1}.} and $\omega^*$ consists of a finite number of concentric rings that are at a positive distance from the boundary. Additionally, we have
$\lim_{j\to +\infty} \int_{\omega^*} \phi_{j,k,m}(x)^2 \, dx = 0$, for every $k\in\N^*$.

\item If $0<\alpha <1/2$, or if $\alpha=1/2$ and $T$ is small enough, then neither the assumption $\Hdeux$ nor its weakened version \eqref{LBSC_weakened} are satisfied. The optimal set $\omega^*$ consists of an infinite number of concentric rings accumulating at the boundary. However the number of connected components of $\omega^*$ intersected with any proper compact subset of $\Omega$ is finite.
\end{itemize}
\end{theorem}

\medskip

Theorems \ref{propLBSC1}, \ref{propLBSC2} and \ref{theoLBSC3} are proved in Section \ref{sec_proofs}.

Theorem \ref{theoLBSC3} is probably the most difficult result of the paper. Its contents contrast with those of Theorem \ref{propLBSC2}. 
Indeed, for instance in the two-dimensional square, the optimal observation domain consists of a finite number of connected components, which are at a positive distance from the boundary (since we are in the Dirichlet case), and this whatever the value of $\alpha>0$ may be. In the two-dimensional disk, we have a similar conclusion for $\alpha>1/2$, but if $\alpha<1/2$ then the optimal domain is much more complex and has an infinite number of connected components.
This surprising result shows that the complexity of the optimal domain $\omega^*$ depends on the geometry of the whole $\Omega$. 

Note that, in Theorem \ref{theoLBSC3}, the result of Theorem \ref{mainTheo} cannot be applied if $\alpha<1/2$. We are however able to prove the existence and the uniqueness of an optimal domain.
The proof relies on a nontrivial minimax argument combined with fine properties of Bessel functions, analyticity considerations, and the use of quantum limits in the disk.

\begin{remark}
For $\alpha>1/2$, the fact that $\liminf_{j+k\to +\infty} \int_{\omega^*} \phi_{j,k,m}(x)^2 \, dx = 0$ in spite of $\Hdeux$ is in contrast with the results given in Theorems \ref{propLBSC1} and \ref{propLBSC2} where this limit was positive. 

At this step the role of the weights $\gamma_{j,k}(T,\alpha)$ must be underlined. Indeed, in the disk there is the well-known whispering gallery phenomenon, according to which a subsequence of the probability measures $\phi_{j,k,m}^2dx$ converges vaguely to the Dirac along the boundary (this property is recalled in a precise way in the proof of Lemma \ref{lemma:tech_ineq} in terms of semi-classical limits). 
Note that the whispering gallery concentration phenomenon is however not strong enough to imply the failure of $\Hdeux$ if $\alpha>1/2$, due to the exponential increase of the coefficients $\gamma_{j,k}(T,\alpha)$ as $j+k$ tends to $+\infty$ (see also Section \ref{sec_LBSC}).

In contrast, if $\alpha \in (0,1/2)$ then the increase of the coefficients $\gamma_{j,k}(T,\alpha)$ is not strong enough, which is in accordance with the fact that $\Hdeux$ fails.
\end{remark}

\begin{remark}\label{rem_strict}
As noted in the inequality \eqref{ineqrand}, there holds $C_T(\chi_\omega)\leq C_{T,\textnormal{rand}}(\chi_\omega)$, for every measurable subset $\omega$ of $\Omega$. The last part of Theorem \ref{mainTheo} states that the inequality is strict for the optimal set $\omega^*$.
Combining this remark with Theorem \ref{propLBSC2}, it is interesting to note the following fact. 

Assume that $\Omega=(0,\pi)^n$, for some $n\in\N^*$, and fix an arbitrary $\alpha\in (0,1/2)$. According to Theorem \ref{propLBSC2}, there exists a unique optimal set, and moreover one has $C_T(\chi_{\omega^*})< C_{T,\textnormal{rand}}(\chi_{\omega^*})$.
According to \cite{MZ1,Miller}, the anomalous diffusion equation \eqref{anomalouseq} is not exactly null controllable for $\alpha<1/2$, and therefore (by duality) $C_T(\chi_{\omega^*})=0$.
Hence, we have here an example where $C_T(\chi_{\omega^*})=0$ whereas $C_{T,\textnormal{rand}}(\chi_{\omega^*})>0$.
\end{remark}

\subsubsection{Several numerical simulations}%\label{sec_num}
We provide hereafter several numerical simulations, illustrating the above results.
The truncated problem of order $N$ is obtained by considering all couples $(j,k)$ such that $j\leq N$ and $k\leq N$. 
The simulations are made with a primal-dual approach combined with an interior point line search filter method\footnote{More precisely, we used the optimization routine IPOPT (see \cite{IPOPT}) combined with the modeling language AMPL (see \cite{AMPL}) on a standard desktop machine.}.

On Figure \ref{figpb2D} (resp., on Figure \ref{figpb2N}), we compute the optimal domain $\omega^N$ for the operator $A_0=-\triangle$, the Dirichlet-Laplacian (resp., the Neumann-Laplacian on the domain defined with zero average) on the square ${\Omega}=(0,\pi)^2$. We can observe the expected stationarity property of the sequence of optimal domains $\omega^N$ from $N=4$ on (i.e., $16$ eigenmodes).

Note that, in the numerical simulations, we have taken $T=0.05$, that is, a small value. Indeed, in accordance with Remark \ref{rem_largetime}, if we take $T$ too large then the stationarity property is observed from $N=1$ on, and then the numerical simulations are not very meaningful.

\begin{figure}[H]
\begin{center}
\includegraphics[width=12cm]{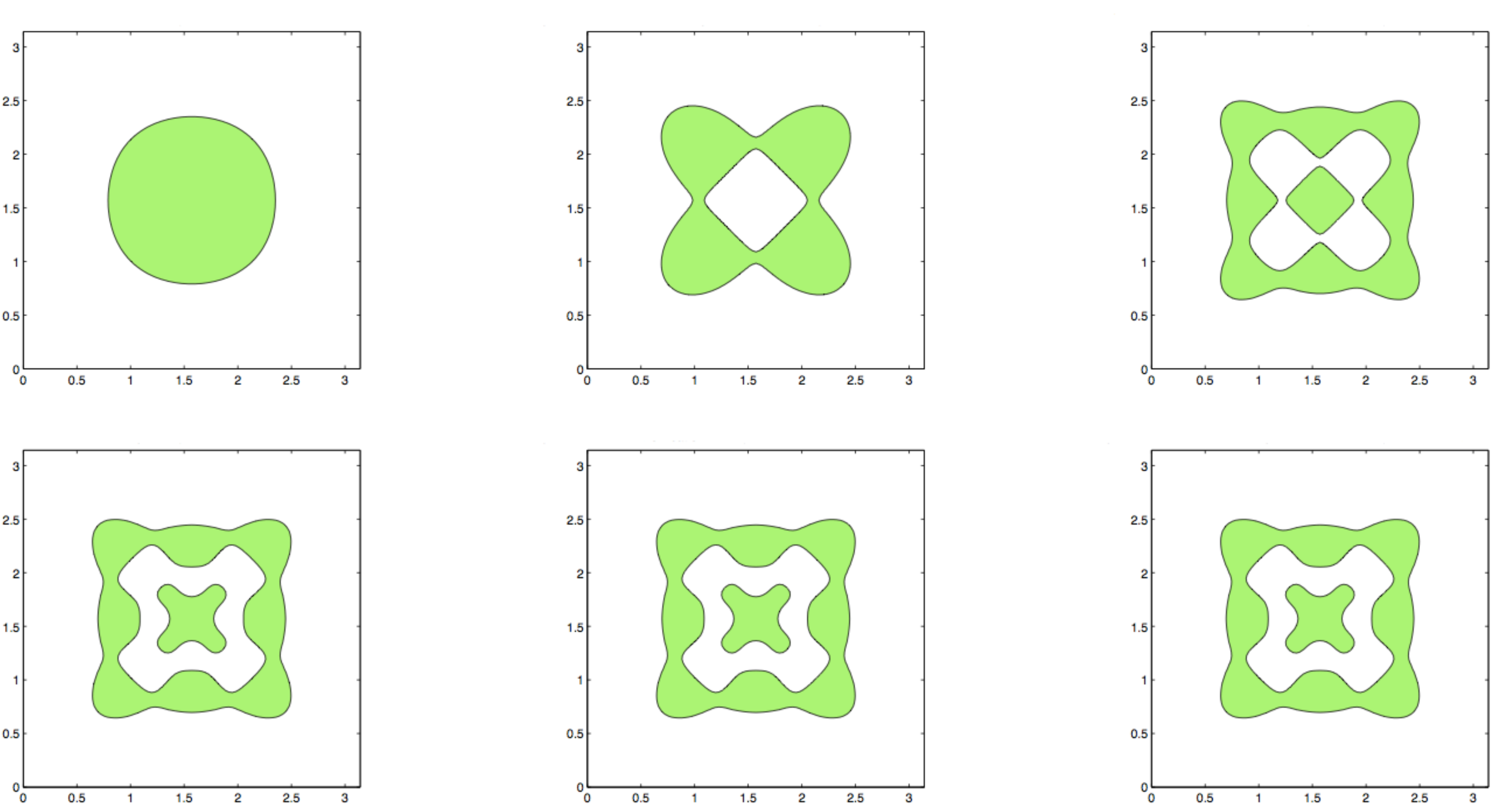}
\caption{On this figure, ${\Omega}=(0,\pi)^2$, $L=0.2$, $T=0.05$, and $A_{0}$ is the Dirichlet-Laplacian. Row 1, from left to right: optimal domain $\omega^N$ (in green)  for $N=1$, $2$, $3$. Row 2, from left to right: optimal domain $\omega^N$ (in green) for $N=4$, $5$, $6$.}\label{figpb2D}
\end{center}
\end{figure}

\begin{figure}[H]
\begin{center}
\includegraphics[width=12cm]{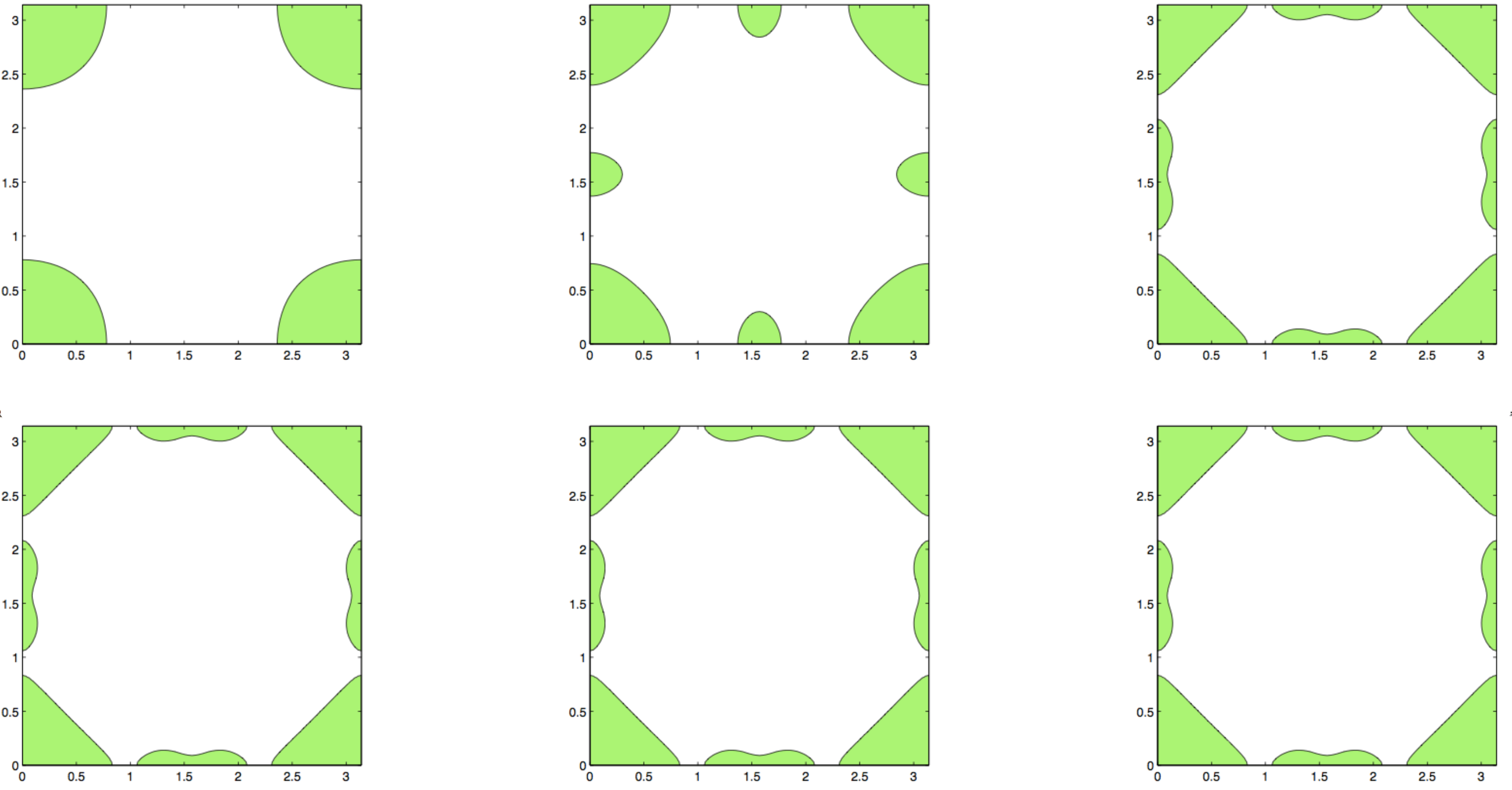}
\caption{On this figure, ${\Omega}=(0,\pi)^2$, $L=0.2$, $T=0.05$, and $A_{0}=-\triangle$ is the Neumann-Laplacian defined on the domain $D(A_0) = \{y \in H^2(\Omega,\C)\ \vert\ \int_\Omega y = 0\textrm{ and }\frac{\partial y}{\partial n}=0 \textrm{ on }\partial\Omega\}$. %(see Remark \ref{sec_Neumann}). 
Row 1, from left to right: optimal domain $\omega^N$ (in green)  for $N=1$, $2$, $3$. Row 2, from left to right: optimal domain $\omega^N$ (in green) for $N=4$, $5$, $6$.}\label{figpb2N}
\end{center}
\end{figure}

On Figures \ref{figpb2Ddisk} and \ref{figpb2Ddiskfrac}, we compute the optimal domain $\omega^N$ for the operator $A_0=(-\triangle)^\alpha$, the fractional Dirichlet-Laplacian on the unit disk $\Omega=\{x\in\R^2\mid \Vert x\Vert < 1\}$, for $\alpha=1$ and $\alpha=0.15$. The numerical simulations illustrate the result stated in Theorem \ref{theoLBSC3}. Indeed, in the case $\alpha=1$, we can observe the expected stationarity property of the sequence of optimal domains $\omega^N$ from $N=3$ on (i.e., $9$ eigenmodes). In the case $\alpha=0.15$, the numerical simulations provide evidence of the accumulation of concentric rings at the boundary (as expected); they are done with values of $N$ between $1$ and $15$ (i.e., $225$ eigenmodes).

\begin{figure}[H]
\begin{center}
\includegraphics[width=12cm]{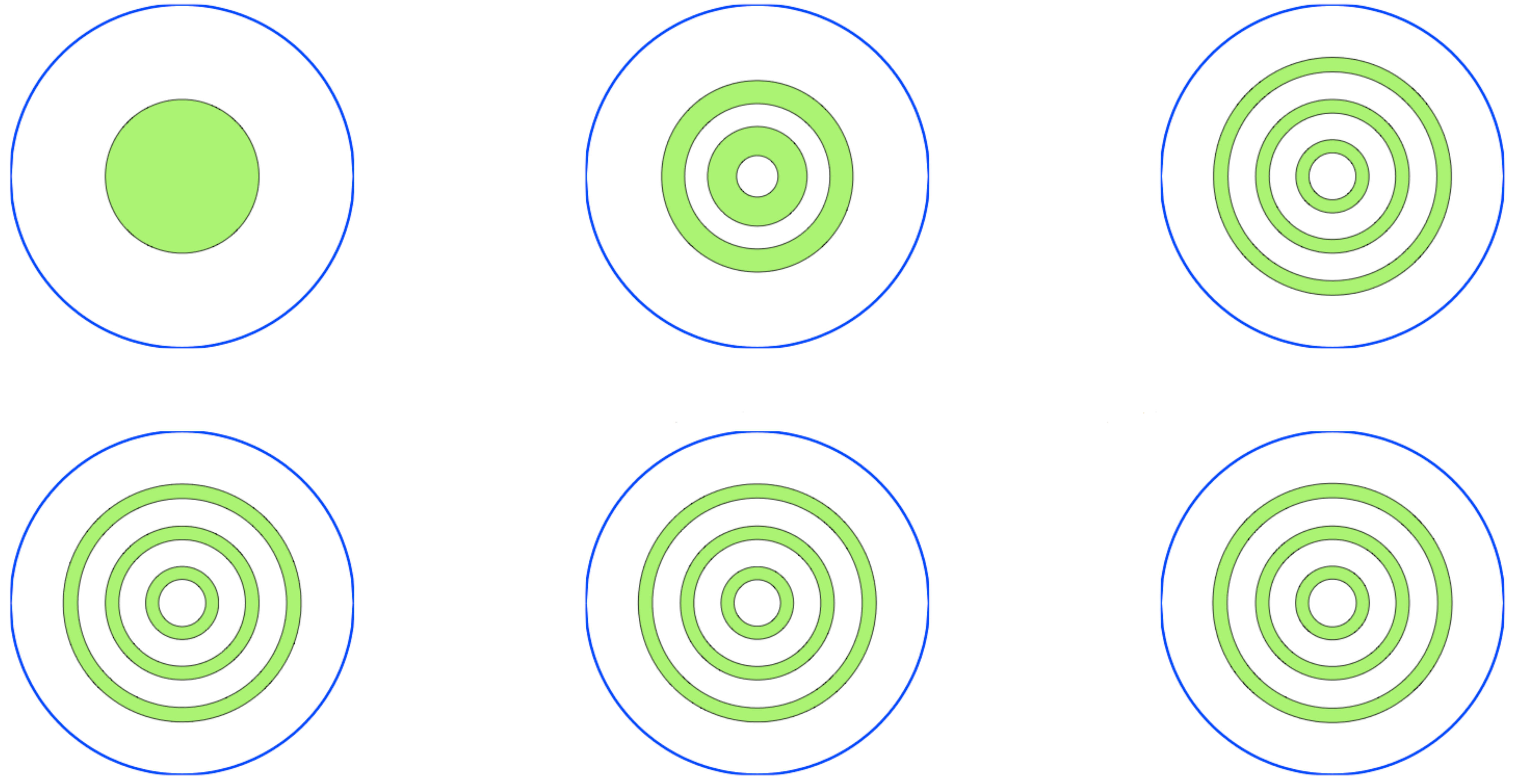}
\caption{On this figure, ${\Omega}=\{x\in\R^2\mid \Vert x\Vert < 1\}$, $L=0.2$, $T=0.05$, and $A_{0}$ is the Dirichlet-Laplacian. Row 1, from left to right: optimal domain $\omega^N$ (in green)  for $N=1$, $2$, $3$. Row 2, from left to right: optimal domain $\omega^N$ (in green) for $N=4$, $5$, $6$.}\label{figpb2Ddisk}
\end{center}
\end{figure}

\begin{figure}[H]
\begin{center}
\includegraphics[width=12cm]{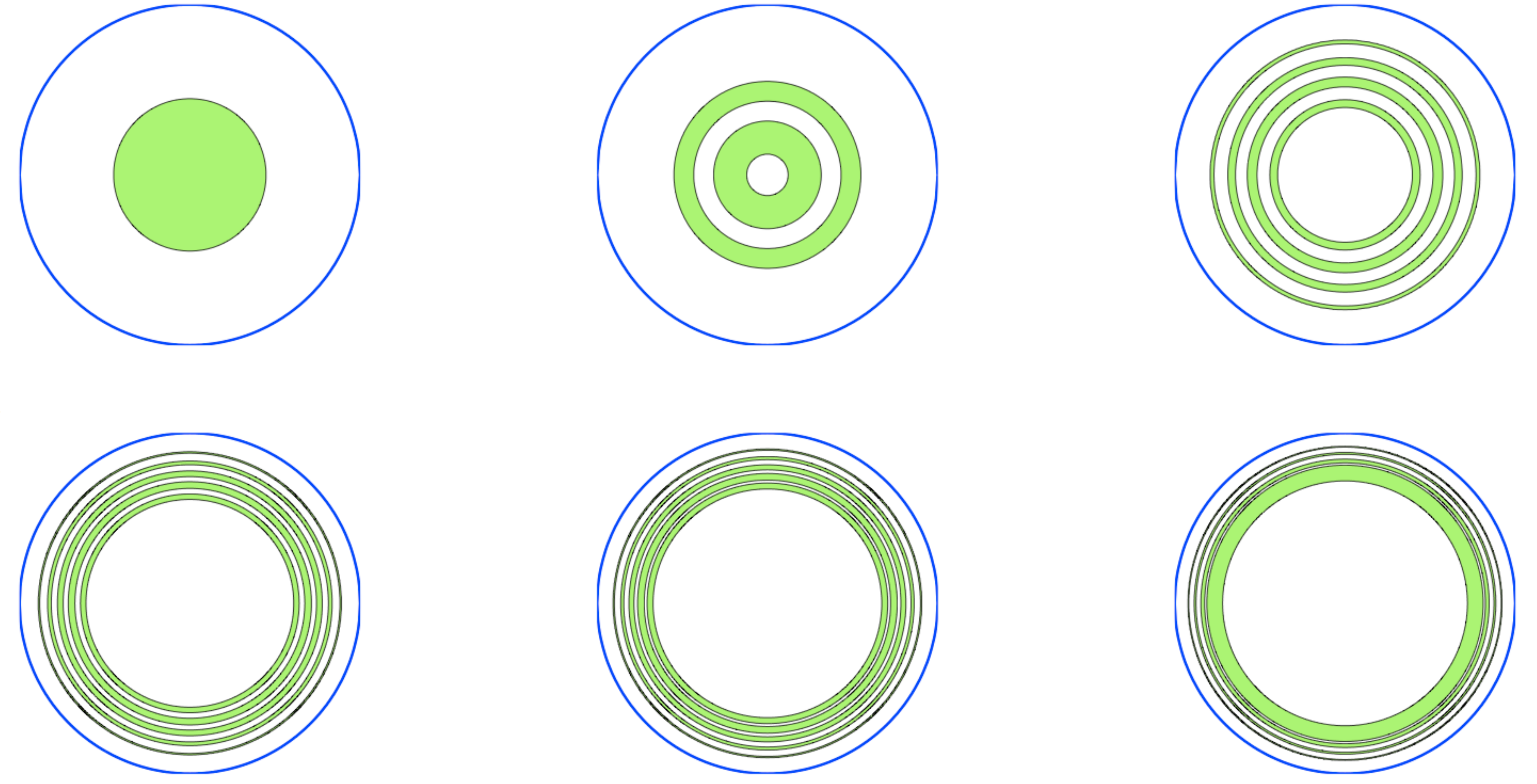}
\caption{On this figure, ${\Omega}=\{x\in\R^2\mid \Vert x\Vert < 1\}$, $L=0.2$, $T=0.05$, and $A_{0}=(-\triangle)^\alpha$ is the fractional Dirichlet-Laplacian with $\alpha=0.15$. Row 1, from left to right: optimal domain $\omega^N$ (in green)  for $N=1$, $2$, $5$. Row 2, from left to right: optimal domain $\omega^N$ (in green) for $N=10$, $12$, $15$.}\label{figpb2Ddiskfrac}
\end{center}
\end{figure}

These figures show what must be the optimal shape and placement of a thermometer in a square domain or in a disk (for the corresponding boundary conditions), when the observation is made over the horizon of time $[0,T]$.

\begin{remark}
If $\alpha=1$ then we are in the framework of Theorem \ref{mainTheo} and hence $N_0(T)=1$ if $T$ is large enough. With respect to what is drawn on Figure \ref{figpb2Ddisk}, this means that if $T$ is large enough then the optimal set is simply the central disk. The situation is however much more complicated if $\alpha<1/2$ (as on Figure \ref{figpb2Ddiskfrac}), since it is proved that a finite number of modes is never sufficient in order to recover the optimal set. In that case, for every value of $T$ the optimal set will always consist of an infinite number of concentric rings accumulating at the boundary, and it is an open and interesting question to investigate how the optimal set behaves when $T$ tends to $+\infty$.
\end{remark}

%%%%%%%%%%%%%%%%%%%%%%%%%%%%%%%%%%%%%%%%%%%%%%%%%

\subsection{Further comments from a semi-classical analysis viewpoint}\label{sec_LBSC}
The assumption $\Hdeux$ is of a spectral nature and can be seen from a semi-classical analysis viewpoint as follows.
The probability measure $\mu_{j}=\phi_j(x)^2 \, dx$ is interpreted (in quantum mechanics) as the probability of being in the state $\phi_j$ with an energy $\lambda_j$. Every closure point or weak limit for the vague topology of the sequence of probability measures $(\mu_j)_{j\in\N^*}$ is called a \textit{semi-classical measure} or a \textit{quantum limit} (the general definition is however in the phase space).
In this sense, the assumption $\Hdeux$ can be called a "lower-bound semi-classical assumption".

The question of determining the set of quantum limits is widely open in general. One is able to compute them only in very particular cases. In the standard round sphere (in any dimension) any geodesic invariant measure is a quantum limit (see \cite{JakobsonZelditch}), hence in particular the Dirac along any geodesic circle is a quantum limit.
This provides an account for possible strong concentrations of eigenfunctions. Similarly, in the disk with Dirichlet boundary conditions, the Dirac along the boundary is a quantum limit (accounting for the already mentioned whispering galleries phenomenon)
In contrast, in the flat torus (in any dimension) all quantum limits are absolutely continuous (see \cite{Jakobson}).

In some sense the assumption $\Hdeux$ stipulates that there is no very strong concentration phenomenon. 
To be more precise, we claim that:
\begin{quote}
\textit{The assumption $\Hdeux$ holds true if one is able to establish that the eigenfunctions $\phi_j$ are uniformly bounded in $L^\infty$ and that every semi-classical measure (weak limit of the probability measures $\mu_j$ for the vague topology) is absolutely continuous and the corresponding densities are positive over the whole domain $\Omega$.}
\end{quote}
This claim easily follows from the Portmanteau theorem (see also Remark \ref{rem21:00} further), because then, using the fact that $\gamma_j(T)$ is exponentially increasing, it follows that $\gamma_j(T)\int_{\Omega} a(x)\phi_j(x)^2\, dx\rightarrow +\infty$ for every $a\in\overline{\mathcal{U}}_L$.

Unless the case of flat tori mentioned above, we are not aware of existing results establishing exactly such a property, however results in this direction can be found in \cite{AnantharamanNonnenmacher,BurqZworski2}. Note that this property holds true for square domains (as explained previously). 

\medskip

In general, there are many possible quantum limits. The most natural one is the uniform measure, and it is indeed an important issue in quantum physics is to determine appropriate assumptions on $\Omega$ under which the probability measures $\mu_j$ tend to equidistribute as $j$ converges to $+\infty$. 
The famous Schnirelman theorem (see \cite{CdV2,gerard,HZ,Sh1,ZZ}) states that, if $\Omega$ is  ergodic with a piecewise smooth boundary, then\footnote{Note that the results established in these references are actually stronger and derive the QE property, not only "on the base" (that is, in the configuration space $\Omega$), but in the unit cotangent bundle $S^*{\Omega}$ of $\Omega$, in the framework of pseudo-differential operators. Here, we are concerned only with weak limits in $\Omega$, and following \cite{Zel3} we use the wording "on the base".}
 there exists a subsequence of $(\mu_j)_{j\in\N^*}$ of density one converging vaguely to the uniform measure $\frac{1}{\vert{\Omega}\vert}dx$ (\textit{Quantum Ergodicity on the base}).
Here, \textit{density one} means that there exists $\mathcal{I}\subset \N^*$ such that $\# \{j\in\mathcal{I}\ \vert\  j\leq N\} /N$ converges to $1$ as $N\to+\infty$, and the manifold is seen as a billiard where the geodesic flow moves at unit speed and bounces at the boundary according to the Geometric Optics laws. 

The Shnirelman theorem lets however open the possibility of having an exceptional sequence of measures $\mu_j$ converging vaguely, e.g., to an invariant measure carried by unstable closed geodesic orbits or on some invariant tori formed by such orbits. This kind of semi-classical measure is referred to as a \textit{scar} and accounts for an energy concentration phenomenon.

\medskip

Then, with respect to our discussion concerning the validity of the assumption $\Hdeux$, the worst possible case is when there exist a quantum limit which is completely concentrated, such as a scar.
In this sense, the assumption $\Hdeux$ is a "non-scarring" assumption.

\begin{remark}\label{rem21:00}
In the claim above (and in Theorem \ref{mainTheo}) we have assumed that the eigenfunctions are uniformly bounded in $L^\infty({\Omega})$. This strong assumption holds true in domains that are Cartesian products of one-dimensional domains, but for example if $\Omega$ is a ball then the eigenfunctions of the Dirichlet-Laplacian are not uniformly bounded.

It is interesting to understand why we add the strong assumption of $L^\infty$ uniform boundedness. It is needed in the application of the Portmanteau theorem, for the following reason.
In semi-classical analysis the vague topology for measures is usually employed.
Assuming that the quantum limits under consideration are absolutely continuous, the convergence in vague topology means that (up to subsequence)
$$
\lim_{j\rightarrow+\infty}\int_\omega \phi_j^2 \, dx = \int_\omega \phi^2\, dx\qquad
\forall\omega\ \textrm{measurable s.t.}\ \vert\partial\omega\vert=0,
$$
that is, the convergence holds on every Jordan measurable set.
In contrast, the convergence in $L^1$ weak topology means that
$$
\lim_{j\rightarrow+\infty}\int_\omega \phi_j^2 \, dx = \int_\omega \phi^2\, dx\qquad
\forall\omega\ \textrm{measurable},
$$
that is, the convergence does hold true as well for those measurable subsets whose boundary has a positive measure. Both convergence properties do coincide as soon as we add the $L^\infty$ boundedness assumption.
This explains why we added such a strong assumption. Indeed our aim is to be able to capture any possible measurable subset.
\end{remark}

%%%%%%%%%%%%%%%%%%%%%%%%%%%%%%%%%%%%%%%%%
%%%%%%%%%%%%%%%%%%%%%%%%%%%%%%%%%%%%%%%%%

\section{Proofs}\label{sec_proofs}
This section is devoted to prove Proposition \ref{truncTheo}, Theorems \ref{mainTheo}, \ref{propLBSC1}, \ref{propLBSC2} and \ref{theoLBSC3}, and finally (in this order), Theorem \ref{theo:stokes}.

\subsection{Proof of Proposition \ref{truncTheo}}\label{sec_proof_truncTheo}
For every $N\in\N^*$, the functional $J_N$ defined by \eqref{defJN} on $\mathcal{U}_L$ is extended to $\overline{\mathcal{U}}_L$ (see Remark \ref{rem:introrelax}) by setting
\begin{equation}\label{defJNa}
J_N(a) = \inf_{1\leq j\leq N} \gamma_j(T) \int_{\Omega} a (x)\vert\phi_j(x)\vert^2\, dx,
\end{equation}
for every $a\in \overline{\mathcal{U}}_L$. 
We consider the relaxed truncated problem
\begin{equation}\label{truncoptdesignpb}
\sup_{a\in \overline{\mathcal{U}}_L} J_N(a) .
\end{equation}
Using the same arguments as in the proof of Lemma \ref{propExistRelax}, it is clear that the problem \eqref{truncoptdesignpb} has at least one solution $a^N\in\overline{\mathcal{U}}_L$. Let us prove that $a^N$ is the characteristic function of a set $\omega^N$ such that $\chi_{\omega^N}\in\mathcal{U}_L$. Define the simplex set
$$
\mathcal{S}_N=\Big\{\alpha=(\alpha_j)_{1\leq j\leq N}\in \R_+^N\ \Big\vert\ \sum_{j=1}^N\alpha_j=1\Big\}.
$$
It follows from the Sion minimax theorem (see \cite{Sion}) that
\begin{eqnarray*}
\sup_{a\in \overline{\mathcal{U}}_L} \min_{1\leq j\leq N} \gamma_j(T) \int_{\Omega}a(x) \vert\phi_j(x)\vert^2\,dx & = & \max_{a\in \overline{\mathcal{U}}_L} \min_{\alpha\in\mathcal{S}_N}\int_{\Omega}a(x) \sum_{j=1}^N\alpha_j \gamma_j(T) \vert\phi_j(x)\vert^2\,dx\\
 & = &  \min_{\alpha\in\mathcal{S}_N}\max_{a\in \overline{\mathcal{U}}_L}\int_{\Omega}a(x) \sum_{j=1}^N\alpha_j \gamma_j(T) \vert\phi_j(x)\vert^2\,dx ,
\end{eqnarray*}
and that there exists $\alpha^N\in \mathcal{S}_N$ such that $(a^N,\alpha^N)$ is a saddle point of the functional
$$
(a,\alpha)\in \overline{\mathcal{U}}_L\times \mathcal{S}_N\longmapsto\sum_{j=1}^N\alpha_j  \gamma_j(T) \int_{\Omega}a(x)\vert\phi_j(x)\vert^2\,dx.
$$
Therefore, $a^N$ is solution of the optimal design problem
$$
\max_{a\in \overline{\mathcal{U}}_L}\int_{\Omega}a(x) \sum_{j=1}^N\alpha_j^N \gamma_j(T) \vert\phi_j(x)\vert^2\,dx.
$$
Set $\varphi_N(x)=\sum_{j=1}^N\alpha_j^N \gamma_j(T) \vert\phi_j(x)\vert^2$, for every $x\in\Omega$.
It follows from $\Hun$ that $\varphi_N$ is never constant on any subset of $\Omega$ of positive measure. 
Therefore, there exists $\lambda^N$ such that $a^N(x)=1$ whenever $\varphi_N(x)\geq\lambda_N$, and $a^N(x)=0$ otherwise. In other words, $a^N=\chi_{\omega^N}\in\mathcal{U}_L$, with $\omega^N=\{x\in\Omega\ \vert\ \varphi_N(x)>\lambda_N\}$. 

The uniqueness of $a_N$ follows from the fact that, as proved above, any optimal solution is a characteristic function. Indeed if there were two optimal sets, then any convex combination would also be an optimal solution because $J_N$ is concave. This raises a contradiction since any maximizer has to be a characteristic function.

Under the additional assumption $\Htrois$, the function $\varphi_N$ is analytic in $\Omega$ and therefore $\omega^N$ is an open semi-analytic set.

\subsection{Proof of Theorem \ref{mainTheo}}
According to Lemma \ref{propExistRelax}, the relaxed optimal design problem \eqref{defJa} has at least one solution $a^*\in\overline{\mathcal{U}}_L$.
The assumption $\Hdeux$ applied to $a^*$ implies that there exists $N_0\in\N^*$ such that
\begin{equation}\label{propN0}
\inf_{j>N_0}\gamma_j(T)\int_{\Omega} a^*(x) |\phi_j(x)|^2\, dx > \gamma_1(T) .
\end{equation}
Since there holds in particular $J_{N_0}(a^*)\leq\gamma_1(T)\int_{\Omega} a^*(x)\vert\phi_1(x)\vert^2\, dx\leq\gamma_1(T)$, we infer from \eqref{propN0} that
$$
J(a^*)  = \min \left(J_{N_0}(a^*),\inf_{j>N_0}\gamma_j(T)\int_{\Omega} a^*(x) |\phi_j(x)|^2\, dx\right) = J_{N_0}(a^*) .
$$
Using $\Hun$ and Proposition \ref{truncTheo}, let  $a^{N_0}\in\mathcal{U}_L$ be the maximizer of $J_{N_0}$.
Let us prove that  $J(a^*) = J_{N_0}(a^{N_0})$. Since $a^{N_0}$ maximizes $J_{N_0}$ over $\overline{\mathcal{U}}_L$, one has $J(a^*) =J_{N_0}(a^*) \leq J_{N_0}(a^{N_0})$.
Let us argue by contradiction and assume that $J_{N_0}(a^*) < J_{N_0}(a^{N_0})$.
For every $t\in[0,1]$, we set $ a_t = a^* + t ( a^{N_0} - a^*)$.
Since $J_{N_0}$ is concave (as an infimum of linear functionals), we get
$$
J_{N_0}(a_t) \geq (1-t)J_{N_0}(a^*) + t J_{N_0}(a^{N_0}) > J_{N_0}(a^*) = J(a^*),
$$
for every $t\in(0,1]$, which means that
\begin{equation}\label{train17h39}
\inf_{1\leq j\leq N_0} \gamma_j(T) \int_{\Omega} a_t(x) |\phi_j(x)|^2 \, dx > \inf_{1\leq j\leq N_0} \gamma_j(T) \int_{\Omega} a^*(x) |\phi_j(x)|^2 \, dx \geq J(a^*),
\end{equation}
for every $t\in(0,1]$. 
Besides, for every $\varepsilon>0$ there exists $t>0$ small enough such that
$$ \gamma_j(T) \int_{\Omega} a_t(x) |\phi_j(x)|^2 \, dx  \geq (1-t)\gamma_j(T) \int_{\Omega} a^*(x)|\phi_j(x)|^2 \, dx  \geq   \gamma_1(T) +\varepsilon,
$$
for every $j>N_0$. Therefore, 
\begin{equation}\label{train17h44}
\inf_{j> N_0}  \gamma_j(T) \int_{\Omega} a_t(x) |\phi_j(x)|^2 \, dx  > \gamma_1(T) .
\end{equation}
Since there holds in particular $J_{N_0}(a_t)\leq\gamma_1(T)$, we infer from \eqref{train17h39} and \eqref{train17h44} that $J(a_t) = J_{N_0}(a_t) > J(a^*)$, which contradicts the optimality of $a^*$.

Therefore $J_{N_0}(a^*) =J(a^*)= J_{N_0}(a^{N_0})$, whence the result.

\medskip

The function $T\mapsto N_{0}(T)$ is clearly nonincreasing since the function $T\mapsto \gamma_{j}(T)$ is increasing for every $j\in\N^*$.
The fact that $N_0(T)=1$ for $T$ large enough is an obvious consequence of the fact that $\gamma_2(T)/\gamma_1(T)\rightarrow+\infty$ as $T\rightarrow+\infty$ under the additional assumption that $\Real(\lambda_j)\rightarrow+\infty$ as $j\rightarrow+\infty$.

\medskip

It remains to prove that $C_T(\chi_{\omega^*})< C_{T,\textnormal{rand}}(\chi_{\omega^*})$ (assuming $\Htrois$).
In the conditions of Theorem \ref{mainTheo}, there exists $j_0\in\{1,\ldots,N_0\}$ such that
$$
C_{T,\textrm{rand}}(\chi_{\omega^*}) = \gamma_{j_0}(T) \int_\Omega \chi_{\omega^*}(x)\vert\phi_{j_0}(x)\vert^2\, dx = \min_{1\leq j\leq N} \gamma_j(T) \int_\Omega \chi_{\omega^*}(x)\vert\phi_j(x)\vert^2\, dx,
$$
for every $N\geq N_0$.

\begin{lemma}\label{lemm4}
There exists an integer $k_0\neq j_0$ such that $\int_\Omega \chi_{\omega^*}(x) \overline\phi_{j_0}(x) \phi_{k_0}(x)\, dx \neq 0$.
\end{lemma}

\begin{proof}
We argue by contradiction. Assume that $\int_\Omega  \chi_{\omega^*}(x)\overline\phi_{j_0}(x) \phi_{k}(x)\, dx =0$, for every integer $k\neq j_0$. Since $(\phi_k)_{k\in\N^*}$ is a Hilbert basis of $L^2(\Omega)$, it follows that there exists a constant $c\in\C$ such that $\chi_{\omega^*}(x)\phi_{j_0}(x) = c\, \phi_{j_0}(x)$, for every $x\in\Omega$. In particular, this implies that $\phi_{j_0}$ must be equal to $0$ on the nonempty open set $\Omega\setminus\bar\omega^*$. Since $\phi_{j_0}$ is analytic, it must be identically zero on $\Omega$. This is a contradiction.
\end{proof}

From now on, let us fix an integer $N$ such that $N\geq N_0$ and $N\geq k_0$.

Let $(e_i)_{i\in\N^*}$ be the canonical basis of $\ell^2(\C)$. Using the notations of Remark \ref{rem9}, we have
$$ C_T(\chi_{\omega^*}) = \inf \{ \langle G_T(\chi_{\omega^*})X,X\rangle\ \vert\ X\in\ell^2(\C),\ \Vert X\Vert_{\ell^2}=1\},$$
and
$$ C_{T,N}(\chi_{\omega^*}) = \inf \{ \langle G_{T,N}(\chi_{\omega^*})X,X\rangle\ \vert\ X\in\C^N,\ \Vert X\Vert_2=1\}.$$
It follows that $C_T(\chi_{\omega^*})\leq C_{T,N}(\chi_{\omega^*})$.

Besides, taking $X=e_{j_0}$ yields that $C_{T,N}(\chi_{\omega^*}) \leq \gamma_{j_0}(T) \int_\Omega \chi_{\omega^*}(x)\vert\phi_{j_0}(x)\vert^2\, dx = C_{T,\textrm{rand}}(\chi_\omega^*) $.
Let us now show that the latter inequality is actually strict.
Consider the integer $k_0$ of Lemma \ref{lemm4}, and take
$X=\cos(\alpha)\, e_{j_0} + \mathrm{e}^{i\beta}\sin(\alpha)\, e_{k_0}$. 
Denoting by $g_{ij}$ the coefficients of the matrix $G_{T,N}(\chi_{\omega^*})$, we then have
$$
\langle G_{T,N}(\chi_{\omega^*})X,X\rangle = \cos^2(\alpha) \, g_{j_0j_0} + \sin^2(\alpha)\, g_{k_0k_0} + \Real(g_{j_0k_0}\mathrm{e}^{i\beta}\sin(2\alpha)),
$$
and hence, at the first order in $\alpha$ as $\alpha$ tends to $0$, we get $\langle G_{T,N}(\chi_{\omega^*})X,X\rangle = g_{j_0j_0} + 2\alpha \Real(g_{j_0k_0}\mathrm{e}^{i\beta})$.
Choosing $\alpha$ and $\beta$ such that $\alpha \Real(m_{k1}\mathrm{e}^{i\beta})<0$, it follows that 
$$C_T(\chi_{\omega^*})\leq C_{T,N}(\chi_{\omega^*})\leq \langle G_{T,N}(\chi_{\omega^*})X,X\rangle < g_{j_0j_0}= C_{T,\textrm{rand}}(\chi_{\omega^*}).$$

\subsection{Proof of Theorem \ref{propLBSC1}}

%\begin{lemma}\label{prop_Escauriaza}
%Assume that $\Omega$ is Lipschitz and locally star-shaped in $\R^n$. Then, for every $\alpha>1/2$, $\Hdeux$ holds true, for any Hilbert basis of eigenfunctions.
%\end{lemma}

%\begin{proof}
To avoid any confusion, we denote by $(\mu_j)_{j\in\N^*}$ the (positive) eigenvalues of the negative of the Dirichlet-Laplacian. With this notation, the eigenvalues of $A_0=(-\triangle)^\alpha$ are given by $\lambda_j=\mu_j^\alpha$, for every $j\in\N^*$. Let $(\phi_j)_{j\in\N^*}$ be an arbitrary Hilbert basis of eigenfunctions of $A_0$ (and of $\triangle$).

It is well known that the eigenfunctions (which are real-valued) are analytic, and hence $\Htrois$ is satisfied.

\medskip

Let us prove that the assumption $\Hun$ holds true.
Let $N\in\N^*$, $(\alpha_{j})_{1\leq j\leq N}\in (\R_{+})^N$ and $C\geq 0$ be such that $\sum_{j=1}^N \alpha_j \phi_j(x)^2 = C$ almost everywhere on some subset $E$ of positive measure. By analyticity and by continuity, the function $x\mapsto\sum_{j=1}^N \alpha_j \phi_j(x)^2$ must be constant on $\bar\Omega$ on its whole, and $\Hun$ follows since the functions $\phi_j$ vanish on $\partial\Omega$. 

\medskip

Let us prove that
\begin{equation*}
\liminf_{j\rightarrow+\infty} \ \frac{e^{2\lambda_j T}-1}{2\lambda_j} \int_\Omega a(x)\phi_j(x)^2\, dx = +\infty ,
\end{equation*}
for every $a\in\overline{\mathcal{U}}_L$, which will imply $\Hdeux$.

For every $a\in\overline{\mathcal{U}}_L$, there exist $\varepsilon>0$ and a measurable subset $E$ of $\Omega$ with positive measure such that $a\geq \varepsilon\chi_E$. Therefore,
$$
\int_\Omega a(x)\phi_j(x)^2\, dx \geq \varepsilon\int_E\phi_j(x)^2\, dx,
$$
for every $j\in\N^*$.
Moreover, it can be assumed that there exist $x_0\in\Omega$ and $R>0$ such that $E\subset  B(x_0,R)\subset \Omega$.
This last technical assumption is required to apply results of \cite{AEWZ}.
Now, it follows from \cite{AEWZ} that, under the regularity assumptions on $\Omega$, there exists a positive constant $C$ (depending on $\Omega$, $R$, $\vert E\vert / \vert B(x_0,R)\vert$) such that
$$
\int_E \vert\phi_j(x)\vert\, dx \geq \frac{e^{-C\sqrt{\mu_j}}}{C} ,
$$
for every $j\in\N^*$, and thus, from Cauchy-Schwarz inequality,
$$
\int_E \phi_j(x)^2 \, dx \geq \frac{e^{-2C\sqrt{\mu_j}}}{C^2\vert E\vert} ,
$$
for every $j\in\N^*$.
Therefore,
$$
\liminf_{j\rightarrow+\infty} \ \gamma_j(T) \int_\Omega a(x)\phi_j(x)^2\, dx \geq \frac{\varepsilon}{2C^2\vert E\vert} \liminf_{j\rightarrow+\infty}  \frac{e^{2\mu_j^\alpha T-2C\sqrt{\mu_j}}}{\mu_j^\alpha} = +\infty
$$
since $\alpha> 1/2$.
%\end{proof}

%The first point of Proposition \ref{propLBSC} is proved.

\subsection{Proof of Theorem \ref{propLBSC2}: the $n$-dimensional orthotope}\label{sec_proofpropLBSC}
We proceed in two steps, studying first the case $n=1$, and then the case $n\geq 2$.

\paragraph{Case $n=1$.}
The eigenelements of $A_0$ are given by $\lambda_j=j^{2\alpha}$ and $\phi_j(x)=\sqrt{\frac{2}{\pi}}\sin jx$, for every $j\in\N^*$, and every $x\in[0,\pi]$. The assumption $\Hdeux$ is then satisfied, as a direct consequence of the following lemma whose proof can be found in \cite{periago,PTZ_HUM1D}.

\begin{lemma}\label{lemmeP}
Let $\rho\in L^\infty(0,\pi)$ be a nonnegative fonction. There holds
$$
\int_0^\pi \rho(x)\sin^2(jx)\, dx\geq \frac{1}{2}\left(\int_0^\pi \rho(x)\, dx-\sin \left(\int_0^\pi \rho(x)\, dx\right)\right),
$$
for every $j\in\N^*$.
\end{lemma}

Indeed, it follows from this lemma that $\int_0^\pi a(x)\phi_j(x)^2\, dx \geq \frac{L\pi-\sin (L \pi)}{\pi}$, for every $j\in\N^*$, which clearly implies that $\Hdeux$ holds true, since $\gamma_j(T)\rightarrow +\infty$ as $j\rightarrow+\infty$ (independently on the value of $\alpha>0$). Moreover, the conclusion of Theorem \ref{mainTheo} holds true with $N_0$ defined as the lowest integer $j$ such that
$$
\frac{e^{2j^{2\alpha}T}-1}{2j^{2\alpha}} \geq \frac{\pi}{2}\frac{e^{2T}-1}{L\pi-\sin (L\pi)}.
$$

%Clearly, similar considerations hold as well for the Neumann-Laplacian operator.

\paragraph{Case $n\geq 2$.}
We consider the Hilbert basis of eigenfunctions given by \eqref{basis_orthotope}.

\begin{lemma}\label{lemmasquare}
Let $\rho\in L^\infty((0,\pi)^n)$ be a nonnegative function such that $\int_{(0,\pi)^n} \rho(x)\, dx>0$. We have
$$
\inf_{(j_1,\dots,j_n)\in{\N^*}^n}\int_{(0,\pi)^n} \rho(x)\phi_{j_1,\dots,j_n}(x)^2\, dx\geq F^{[n]} \left(\int_{(0,\pi)^n} \rho(x)\, dx\right)>0,
$$
where $F$ is the function defined on $[0,\pi]$ by $F(s)= \frac{1}{\pi}(s-\sin s)$, and $F^{[n]}$ is the composition of $F$ with itself, $n$ times.
\end{lemma}

\begin{proof}[Proof of Lemma \ref{lemmasquare}]
Using the Fubini theorem and Lemma \ref{lemmeP}, we infer that
\begin{eqnarray*}
\int_{\Omega} \rho(x)\phi_{j_1,\dots,j_n}(x)^2\, dx & = &\left(\frac{2}{\pi}\right)^{n}\int_0^\pi \sin^2(j_nx_n)\int_{[0,\pi]^{n-1}}\rho(x)\prod_{k=1}^{n-1}\sin ^2(j_kx_k)\, dx_1\dots\, dx_{n-1} \, dx_n\\
 & \geq & F\left(\left(\frac{2}{\pi}\right)^{n-1}\int_{[0,\pi]^{n-1}}\rho(x)\prod_{k=1}^{n-1}\sin ^2(j_kx_k)\, dx_1\dots\, dx_{n-1} \right) ,
\end{eqnarray*}
and the conclusion follows from a simple induction argument.
\end{proof}

It follows from this lemma that $\int_{(0,\pi)^n}a(x)\phi_{j_1,\dots,j_n}(x)^2\, dx \geq F^{[n]}(L\pi^n)$, for all $(j_1,\dots,j_n)\in{\N^*}^n$. Therefore the assumption $\Hdeux$ holds true since $\gamma_{(j_1,\dots,j_n)}(T)\rightarrow +\infty$ as $\vert(j_1,\dots,j_n)\vert\rightarrow+\infty$. Moreover, the conclusion of Theorem \ref{mainTheo} holds with $N_0$ defined as the lowest multi-index $(j_1,\ldots,j_n)$ (in lexicographical order) such that
$$
\frac{e^{2\lambda_{(j_1,\ldots,j_n)}T}-1}{2\lambda_{(j_1,\dots,j_n)}} \geq \frac{e^{2\lambda_{(1,\ldots,1)}T}-1}{2\lambda_{(1,\ldots,1)} F^{[n]}(L\pi^n)}.
$$

%Clearly, similar conclusions hold as well in the Neumann case ($\partial_{n}\phi_{j}=0$ on $\partial{\Omega}$), for which we choose as a Hilbert basis of eigenfunctions the sequence of functions defined by
%$$
%\phi_{j}(x)=\left(\frac{2}{\pi}\right)^{n/2}\prod_{k=1}^n\cos \left(j_{k}x_{k}\right).
%$$

%At this step, we have proved that $\Hdeux$ holds if $\Omega=(0,\pi)^n$, with $n\in\N^*$.

\subsection{Proof of Theorem \ref{theoLBSC3}: the unit disk of the Euclidean plane}\label{ex4}
According to Lemma \ref{propExistRelax}, let $a^*$ be a maximizer of $J$ over $\overline{\mathcal{U}}_L$. Our objective is to prove that $a^*$ is unique and is the characteristic function of a subset $\omega^*$ sharing the properties announced in the statement of Theorem \ref{theoLBSC3}.

In order to underline the dependence on $\alpha$, throughout the proof we use the notation
\begin{equation}\label{gammajkdisk}
\gamma_{j,k}(T,\alpha)=
\left\{ \begin{array}{lll}
\displaystyle\frac{e^{2\lambda_{j,k}T}-1}{2\lambda_{j,k}} & \textrm{if} & \lambda_{j,k}\neq 0, \\
T & \textrm{if} & \lambda_{j,k}=0 .
\end{array}\right.
\end{equation}
Setting $\mathcal{I}=\N\times \N^*\times \{1,2\}$, using the expression \eqref{basis_disk} of the eigenfunctions, we have
\begin{equation}\label{Jdisk}
J(a) = \inf_{(j,k,m)\in\mathcal{I}}  \gamma_{j,k}(T,\alpha)\int_0^{2\pi}\int_0^1 a(r,\theta)R_{j,k}(r)^2 Y_{j,m}(\theta)^2\, r\, dr\, d\theta ,
\end{equation}
for every $a\in\overline{\mathcal{U}}_L$, with $Y_{j,1}(\theta)^2=\frac{\cos^2(j\theta)}{\pi}$ and $Y_{j,2}(\theta)^2=\frac{\sin^2(j\theta)}{\pi}$.

To facilitate the reading of the proof, we split it into several steps.
We first introduce an associated radial problem, with a functional $J_r$ corresponding to the functional above restricted to radial functions. We prove that $J$ and $J_r$ have the same maxima (not necessarily the same maximizers). Then we distinguish between two cases: 1) $\alpha>1/2$, 2) $0<\alpha<1/2$ or $\alpha=1/2$ and $T$ small enough. In contrast to the first case, which can be tackled directly using Theorem \ref{propLBSC1}, the second case is much more difficult to treat. We apply a refined version of the minimax theorem in order to prove that the optimal domain exists and is unique. This requires to prove that a certain (switching) function is analytic, which necessitates a very careful and technical analysis using in an instrumental way the knowledge of some quantum limits (semi-classical measures) of the eigenfunctions and of some fine asymptotic properties of Bessel functions. Actually, the proof of the analyticity, which is very lenghty, takes the major part of the section.

\subsubsection{Associated radial problem}
For every $ b\in L^\infty(0,1)$, we set
$$
J_r(b) = \inf_{\substack{j\in\N\\ \ k\in\N^*}}\gamma_{j,k}(T,\alpha)\int_{0}^1b(r ) R_{j,k}(r )^2 r\, dr .
$$
We define the set
$$
{\mathcal{U}}_L^r = \left\{ \chi_{\omega_r}\in L^{\infty}(0,1;\{0,1\}) \mid \omega_r\ \textrm{is a measurable subset of $(0,1)$ of Lebesgue measure}\ \vert\omega_r\vert = \frac{L}{2}  \right\}.
$$
Its weak star convex closure is
$
\overline{\mathcal{U}}_L^r = %\left
\{ b\in L^\infty(0,1;[0,1])\ \vert\ \int_{0}^1b(r) r\, dr =\frac{L}{2}  %\right
\}.
$
We consider the problem 
\begin{equation}\label{radialpb}
\sup_{b\in \overline{\mathcal{U}}_L^r} J_r(b)
\end{equation}
of maximizing $J_r$ over the set ${\mathcal{U}}_L^r$.

%Let us prove that the optimal set $\omega^*$ is radial, in other words that there exists a measurable subset $\omega_{r}^*\subset [0,1]$ such that $\omega^*=\omega_{r}^*\times [0,2\pi]$ in polar coordinates. Then as a matter of fact $\omega_r^*$ will consist of a countable number of open intervals, with the property that $\omega_r^*\cap[0,1-h]$ is the union of a finite number of open intervals, for every $h>0$.
%

\begin{lemma}\label{lem6}
The problem \eqref{radialpb} has at least one solution $b^*$, and
\begin{equation}\label{21:06}
J(a^*) = %J(\chi_{\omega^*})= 
\max_{a\in \overline{\mathcal{U}}_L} J(a) =\max_{b\in \overline{\mathcal{U}}_L^r} J_r(b) = J_r(b^*) .
\end{equation}
%Moreover, $b^*$ is the characteristic function of a subset $\omega_{r}^*\subset[0,1]$, and $\omega_r^*$ is the union of a countable number of open intervals, with the additional property that $\omega_r^*\cap[0,1-h]$ is the union of a finite number of open intervals, for every $h>0$.
%
Besides, if $a^*\in \overline{\mathcal{U}}_L$ is a maximizer of $J$, then the (radial) function $\bar a^*\in \overline{\mathcal{U}}_L$ defined by $\bar a^*(r,\theta) = \frac{1}{2\pi}\int_0^{2\pi} a^*(r,\Theta)\, d\Theta$ (which does not depend on $\theta$) is as well a maximizer of $J$, and the function $b^*\in \overline{\mathcal{U}}_L^r$ defined by $b^*(r) = \bar a^*(r,0)$ is a maximizer of $J_r$.
\end{lemma}

\begin{proof}[Proof of Lemma \ref{lem6}]
Since the functional $J_r$ is concave and upper semi-continuous (as the infimum of continuous linear functionals) for the weak star topology of $L^\infty$, and since ${\mathcal{U}}_L^r$ is compact for this topology, it follows that the problem \eqref{radialpb} has at least one solution $b^*$.

First of all, let us note that, if a function $a\in \overline{\mathcal{U}}_L$ does not depend on $\theta$, then, setting $b(r)=a(r,0)$, the constraint $\int_0^{2\pi}\!\!\int_0^1 a(r,\theta) r\, dr \, d\theta = L\pi$ yields $\int_0^1 b(r) r\, dr = \frac{L}{2}$, that is $b\in \overline{\mathcal{U}}_L^r$, and using \eqref{Jdisk} and the Fubini theorem, we get clearly the equality $J(a) = J_r(b)$.
Therefore, we get
$$
\sup_{a\in \overline{\mathcal{U}}_L} J(a) \geq \sup_{b\in \overline{\mathcal{U}}_L^r} J_r(b) .
$$

Let us prove the converse inequality. Let $a\in \overline{\mathcal{U}}_L$ arbitrary. Settting $b(r)=\frac{1}{2\pi}\int_0^{2\pi}a(r,\theta)\,d\theta$, we have clearly $b\in \overline{\mathcal{U}}_L^r$. On the one hand, we can write
\begin{multline*}
J(a) =   \inf \Bigg( 
 \inf_{k\geq 1} \frac{\gamma_{0,k}(T)}{2\pi} \int_0^{2\pi}\!\!\! \int_0^1 a(r,\theta) R_{0,k}(r)^2 r\, dr \, d\theta , \\
 \inf_{\substack{j,k\geq 1\\ t\in [0,1]}} \frac{\gamma_{j,k}(T,\alpha)}{\pi} \int_0^{2\pi}\!\!\! \int_0^1 a(r,\theta) R_{j,k}(r)^2(t\cos^2(j\theta)+(1-t)\sin^2(j\theta)) r\, dr \, d\theta  \Bigg) ,
\end{multline*}
and on the other hand, we have
\begin{equation}\label{deuxineq}
\begin{split}
& \inf_{\substack{j,k\geq 1\\ t\in [0,1]}} \frac{\gamma_{j,k}(T,\alpha)}{\pi} \int_0^{2\pi}\!\!\! \int_0^1 a(r,\theta) R_{j,k}(r)^2(t\cos^2(j\theta)+(1-t)\sin^2(j\theta)) r\, dr \, d\theta \\
\leq\ & \inf_{j,k\geq 1} \frac{\gamma_{j,k}(T,\alpha)}{\pi} \int_0^{2\pi}\!\!\! \int_0^1 a(r,\theta) R_{j,k}(r)^2\left(\frac{1}{2}\cos^2(j\theta)+\frac{1}{2}\sin^2(j\theta)\right) r\, dr \, d\theta \\
\leq\ & \inf_{j,k\geq 1} \gamma_{j,k}(T,\alpha) \int_0^1 b(r) R_{j,k}(r)^2 r\, dr \, d\theta .
\end{split}
\end{equation}
We infer that $J(a) \leq J_r(b)$, and then the converse inequality indeed follows.

We have proved \eqref{21:06}.

%Note that, with the latter reasoning, we have proved as well the following fact: for every $b\in \overline{\mathcal{U}}_L^r$, any function $a\in \overline{\mathcal{U}}_L$ such that $b(r) = \frac{1}{2\pi} \int_0^{2\pi} a(r,\theta) r \, dr\, d\theta$ (there may exist an infinite number of such functions $a$) is such that $J(a) \leq J_r(b)$.

Now, let $a^*\in \overline{\mathcal{U}}_L$ be a maximizer of $J$. We define the function $\bar a^*\in \overline{\mathcal{U}}_L$ by $\bar a^*(r,\theta) = \frac{1}{2\pi}\int_0^{2\pi} a^*(r,\Theta)\, d\Theta$. The function $\bar a^*$ does actually not depend on $\theta$, and we define also the function $\bar b\in \overline{\mathcal{U}}_L^r$ by $\bar b(r) = \bar a^*(r,0)$.
Using \eqref{deuxineq}, we get $J(a^*) \leq J(\bar a^*)= J_r(\bar b)\leq J_r(b^*)$. The statement follows.
%
%It remains to prove that $b^*$ is a characteristic function, and is unique. We proceed as previously, noticing that
%$$
%J_r(b) = \inf_{\substack{\beta_{j,k}\geq 0\\ \ \sum_{j,k}\beta_{j,k}=1}} \sum_{\substack{j\in\N\\ \ k\in\N^*}}\gamma_{j,k}(T,\alpha)\beta_{j,k}\int_{0}^1b(r ) R_{j,k}(r )^2 r\, dr,
%$$
%and hence that the problem of maximizing $J_r$ can be as well written as
%$$
%\max_{\substack{b\in L^\infty(0,1;[0,1])\\ \int_{0}^1b(r) r\, dr =\frac{L}{2}}} 
% \inf_{\substack{\beta_{j,k}\geq 0\\ \ \sum_{j,k}\beta_{j,k}=1}} G(b,\beta),
%$$
%with
%$$
%G(b,\beta) = \sum_{\substack{j\in\N\\ \ k\in\N^*}}\gamma_{j,k}(T,\alpha)\beta_{j,k}\int_{0}^1b(r ) R_{j,k}(r )^2 r\, dr.
%$$
%The minimax argument then goes as previously, and leads to the conclusion that there exist $\tilde\xi>0$ and $\beta^*$ such that
%$$ b^*(r) = \left\{ \begin{array}{rcl}
%1 & \textrm{if} & \tilde\psi(r)>\tilde\xi,\\
%0 & \textrm{if} & \tilde\psi(r)<\tilde\xi,
%\end{array}\right.$$
%where
%$$
%\tilde\psi(r) = \sum_{\substack{j\in\N\\ \ k\in\N^*}}\gamma_{j,k}(T,\alpha)\beta_{j,k}^* R_{j,k}(r)^2 r .
%$$
%Similarly as in Proposition \ref{technicalLemma_AnaDisk}, the function $\tilde\psi$ is analytic in $[0,1)$, and it follows that $b^*$ is the characteristic function of a measurable subset $\omega_r^*\subset[0,1]$. The uniqueness is proved as before. The complexity properties of $\omega_r^*$ are as well a consequence of the analyticity.
\end{proof}

%As a consequence of this lemma, by uniqueness, we conclude that the set $\omega^*$ is radial, and moreover that $\omega^*=\omega_{r}^*\times [0,2\pi]$ in polar coordinates.

%At this step, the first part of Theorem \ref{theoLBSC3} is proved.

\begin{remark}\label{rem_rad}
In addition to Lemma \ref{lem6}, we note that, if the radial problem \eqref{radialpb} has a unique solution, which is moreover the characteristic function of some measurable subset $\omega_r^*$ of $[0,1]$ (and this is what we will prove in the sequel), then necessarily the functional $J$ has a unique maximizer as well, which is the characteristic function of the set $\omega^*=\omega_{r}^*\times [0,2\pi]$ in polar coordinates.

Indeed, let $a^*\in\overline{\mathcal{U}}_L$ be a maximizer of $J$. Then, according to Lemma \ref{lem6}, the function $b^*\in\overline{\mathcal{U}}_L^r$ defined by $b^*(r)=\frac{1}{2\pi}\int_0^{2\pi} a^*(r,\theta) r \, dr \, d\theta$ is a maximizer of $J_r$, and therefore $b^*=\chi_{\omega_r^*}$. Then, for almost every $r\in[0,1]$ we have $\frac{1}{2\pi}\int_0^{2\pi} a^*(r,\theta) r \, dr \, d\theta = \chi_{\omega_r^*}(r)$, and since $0\leq a^*(r,\theta)\leq 1$ it follows that $a^*(r,\theta) = \chi_{\omega_r^*}(r)$. In other words, we have $a^*=\chi_{\omega^*}$ with $\omega^*=\omega_{r}^*\times [0,2\pi]$.
\end{remark}

Note that, at least at this step, Lemma \ref{lem6} does not imply that any maximizer of $J$ is radial; but it implies that there always exists a radial maximizer, that is, a function maximizing $J$ and that does not depend on $\theta$.

However, in what follows, we will eventually prove that the radial problem \eqref{radialpb} has indeed a unique solution, which is moreover the characteristic function of a set $\omega_r^*$.
Then, according to Remark \ref{rem_rad}, this will finally imply that $J$ has a unique maximizer $a^*=\chi_{\omega^*}$ with $\omega^*=\omega_{r}^*\times [0,2\pi]$ in polar coordinates.
The properties of $\omega^*$ stated in Theorem \ref{theoLBSC3} will then follow from the properties of the set $\omega_{r}^*$ that we will establish hereafter.

\medskip

We distinguish between two cases, depending on the value of $\alpha$. The case $\alpha>1/2$ is much easier to treat.

%\paragraph{Complexity of the optimal set.} Let us finally prove the two last statements of Theorem \ref{theoLBSC3}.

\subsubsection{Case $\alpha>1/2$}
Although we could make a direct proof, we already know, according to Theorem \ref{propLBSC1}, that the assumption $\Hdeux$ holds true. 
Then, according to Theorem \ref{mainTheo}, we have $a^*=\chi_{\omega^*}$ with $\chi_{\omega^*}\in\mathcal{U}_L$. The fact that $\omega^*$ has a finite number of connected components also follows from Theorem \ref{mainTheo}.

The same arguments can be applied to the radial problem \eqref{radialpb}.
More precisely, note first that we have
\begin{equation}\label{LBSCradial}
\liminf_{j+k\rightarrow +\infty}\gamma_{j,k}(T,\alpha)\int_{0}^1b(r ) R_{j,k}(r )^2 r\, dr > \gamma_{1,1}(\alpha,T) .
\end{equation}
Indeed it suffices to apply $\Hdeux$ to a radial function. In other words, \eqref{LBSCradial} is the radial version of $\Hdeux$. Then, under this condition, the proof of Theorem \ref{mainTheo} can be straightforwardly adapted to the radial problem and leads to the following conclusion: the maximizer $b^*$ of $J_r$ is unique and is the characteristic function of a measurable subset $\omega_r^*$ of $[0,1]$, with $\chi_{\omega_r^*} \in \mathcal{U}_L^r$. Moreover there exist $N_{0}\in \N^*$, nonnegative real numbers $(\alpha_{jk}^*)_{0\leq j\leq N_{0},1\leq k\leq N_{0}}$ of sum $1$, and $\lambda^*>0$ such that
$$
\omega^*_r=\left\{\sum_{j=0}^{N_{0}}\sum_{k=1}^{N_{0}}\gamma_{j}(T)\alpha_{jk}^*R_{j,k}(r )^2> \lambda^*\right\}.
$$
Since the functions $R_{j,k}$ are analytic and vanish at $r=1$, the set $\omega_r^*$ is the union of a finite number of intervals that are at a positive distance of $1$.
Moreover, the Lebesgue measure of $\partial\omega^*_r$ is equal to $0$ (indeed $\omega^*_r$ is a finite union of intervals).

By uniqueness, we conclude that $\omega^*=\omega_{r}^*\times [0,2\pi]$ in polar coordinates.

Besides, recall that, for every $k\in\N^*$, the sequence of probability measures $R_{j,k}(r)^2r\, dr$ converges vaguely to the Dirac at $r=1$ as $j$ tends to $+\infty$.
This fact accounts for the phenomenon of whispering galleries, and says that the Dirac along the boundary is a semi-classical measure (quantum limit) in the disk.
Then, from the Portmanteau theorem (note that the Lebesgue measure of $\partial\omega^*_r$ is equal to $0$), we get 
$$
\lim_{j\rightarrow +\infty} \int_{\omega^*_r} R_{j,k}(r)^2 r\, dr = 0,
$$
for every $k\in\N^*$. The additional property $\liminf_{j\to +\infty} \int_{\omega^*} \phi_{j,k,m}(x)^2 \, dx = 0$ for every $k\in\N^*$ follows.

\subsubsection{Case $0<\alpha< 1/2$ (or $\alpha=1/2$ and $T$ small enough).}
This is the most difficult case to deal with.

\paragraph{First estimates.}
Let us first prove the following lemma, providing an exponential estimate of the functions $R_{j,k}$ (in the spirit of estimates derived in \cite{grebenkov}).

\begin{lemma}\label{lem1_Rjk}
For every $h\in (0,1)$, for every $k\in\N^*$ there exists a constant $C_k>0$ such that
\begin{equation}\label{estim1_Rjk}
R_{j,k}(r)^2\leq C_k  j^{4/3} \exp\left( -C_kjh^{3/2} \right) .
\end{equation}
for every $r\in [0,1-h]$, and for every $j\in \N^*$.
\end{lemma}

Note that the estimate \eqref{estim1_Rjk} provides an account for the whispering galleries phenomenon, according to which the eigenfunctions of the Dirichlet-Laplacian in the unit disk tend to concentrate along the boundary of the disk as the index $j$ tends to $+\infty$.
This estimate says that this concentration is exponential.

Note that we will later need to extend the result of that lemma (see Lemma \ref{lem9} further), by proving that the estimate \eqref{estim1_Rjk} actually holds true for a larger set of indices. But for the moment this statement is enough.

\begin{proof}[Proof of Lemma \ref{lem9}]
We will use the so-called Kapteyn inequality, proved in \cite{Siegel}, and stating that
\begin{equation}\label{Kapteyn}
J_j(jy) \leq \exp ( jg(y) ) ,
\end{equation}
for every $y\in[0,1]$, with
$$
g(y) = \sqrt{1-y^2}-\log\frac{1+\sqrt{1-y^2}}{y}.
$$
The function $g:(0,1]\rightarrow (-\infty,0]$ is smooth, increasing, and $g(1)=0$.
Besides, for every $j\in\N^*$ the Bessel function $x\mapsto J_j(x)$ is known to be increasing on the interval $[0,z'_{j,1}]$, where $z'_{j,1}$ is the first positive zero of $J_j'$. Moreover it is known that
$
z'_{j,1} = j + \gamma_1' j^{1/3} + o(j^{1/3}),
$
with $\gamma_1'>0$ (see \cite{Olver}), and that
$\frac{C_1}{j^{1/3}}\leq J_j(j)\leq 1$
(see \cite{Watson}). It follows that $\frac{C_1}{j^{1/3}}\leq J_j(x)$ whenever $j\leq x\leq z'_{j,1}$, and thus that
\begin{equation}\label{tech19:22}
J_j(z_{j,k}r) \geq \frac{C_1}{j^{1/3}}\qquad\forall r\in\left[ \frac{j}{z_{j,k}},\frac{z'_{j,1}}{z_{j,k}}\right].
\end{equation}
Using the inequality $z_{j,k}\leq \pi(j+k)$ (see \cite[Lemma 5]{Kelliher}), we infer from \eqref{tech19:22} that
\begin{equation}\label{ALR1}
\int_0^1 J_j(z_{j,k}r)^2 r dr \geq \frac{C_1}{j^{2/3}} \frac{{z'_{j,1}}^2-j^2}{z_{j,k}^2} \geq \frac{2C_1\gamma_1'}{\pi^2} \frac{j^{2/3}}{(j+k)^2} .
\end{equation}
Besides, recall that, for every $k\in\N^*$ fixed, we have
\begin{equation*}
z_{j,k} = j + \delta_k j^{1/3} + o(j^{1/3}),
\end{equation*}
with $\delta_k>0$ (see \cite{Olver}). Then, for every $r\in [0,1-h]$, we write $z_{j,k}r = j y$ with $y=\frac{z_{j,k}}{j}r$, and we get
$$
y = \frac{z_{j,k}}{j}r \leq (1-h) \left( 1+\frac{\delta_k}{j^{2/3}}+o(j^{-2/3}) \right)
\leq 1-\frac{h}{2}
$$
whenever $j$ is large enough. Therefore, if $j$ is large enough then we get, using the Kapteyn inequality \eqref{Kapteyn} and the fact that $g$ is increasing, that
$$
\vert J_j(z_{j,k} r)\vert \leq \exp \left( j g \Big( 1-\frac{h}{2} \Big) \right),
$$
for every $r\in[0,1-h]$.
Using an asymptotic expansion of $g$, we get that
\begin{equation}\label{ALR2}
\vert J_j(z_{j,k} r)\vert \leq \exp \left( -\frac{1}{3} j h^{3/2} + o(j h^{3/2}) \right),
\end{equation}
for every $r\in[0,1-h]$.
Since $R_{j,k}(r)^2 = J_j(z_{j,k} r)^2 / \int_0^1 J_j(z_{j,k}r)^2 r \, dr$, the estimate \eqref{estim1_Rjk} of the lemma finally follows by combining \eqref{ALR1} with \eqref{ALR2}.
\end{proof}

In what follows, we consider a maximizer $b^*\in\overline{\mathcal{U}}_L^r$ of $J_r$.

\begin{lemma}\label{lemALR}
Neither the assumption $\Hdeux$ nor its weakened version \eqref{LBSC_weakened} are satisfied for the radial problem \eqref{radialpb}. 
\end{lemma}

\begin{proof}[Proof of Lemma \ref{lemALR}.]
We argue by contradiction, assuming that
\begin{equation}\label{21:11}
\liminf_{j+k\rightarrow +\infty} \gamma_{j,k}(T,\alpha)\int_0^1 b^*(r) R_{j,k}(r )^2 r\, dr>\gamma_{1,1}(T,\alpha). 
\end{equation}
Using the same arguments as for the case $\alpha>1/2$, it follows that the maximizer $b^*$ of $J_r$ is unique and is the characteristic function of a measurable subset $\omega_r^*$ of $[0,1]$, with $\chi_{\omega_r^*} \in \mathcal{U}_L^r$. Moreover the optimal set $\omega_r^*\subset[0,1]$ must consist of a finite number of intervals that are at a positive distance from $1$. The important fact that we note here is the fact that there exists $h>0$ such that $\omega_r^*\subset(0,1-h)$.

From the expansion $z_{j,k} = j + \delta_k j^{1/3} + o(j^{1/3})$ (already used), it follows that, for $k$ fixed and $j$ large enough, we have $j \leq z_{j,k} \leq 2j$. Then, using the estimate \eqref{estim1_Rjk} of Lemma \ref{lem1_Rjk}, the expression \eqref{gammajkdisk} of $\gamma_{j,k}(T,\alpha)$, and the inequalities $j^{2\alpha}<\lambda_{j,k}=z_{j,k}^{2\alpha}<(2j)^{2\alpha}$ for $j$ large enough and $k$ fixed, we infer that
\begin{equation*}
%\begin{split}
\gamma_{j,k}(T,\alpha)\int_{\omega_r^*} R_{j,k}(r )^2 r\, dr  \leq
\frac{e^{2T(2j)^{2\alpha}}}{j^{2\alpha}}  C\frac{e^{-Cjh^{3/2}}}{j^{4/3}}
 \leq  C \exp\left(2T(2j)^{2\alpha} - Cjh^{3/2} \right),
%\end{split}
\end{equation*}
and therefore, since we have either $\alpha<1/2$ or $\alpha=1/2$ and $T$ small enough, we get
$$
\lim_{j\to +\infty}\gamma_{j,k}(T,\alpha)\int_{\omega_r^*} R_{j,k}(r )^2 r\, dr=0 ,
$$
for every $k\in\N^*$, which raises a contradiction with \eqref{21:11}. It follows that neither the assumption $\Hdeux$ nor its weakened version \eqref{LBSC_weakened} are satisfied. 
\end{proof}

\begin{lemma}\label{lem_adherence1}
For every $h\in(0,1)$, the restriction of $b^*$ to the interval $[1-h,1]$ is nontrivial.
\end{lemma}

\begin{proof}[Proof of Lemma \ref{lem_adherence1}.]
We argue by contradiction. Assume that there exists $h\in(0,1)$ such that $b^*(r)=0$ for every $r\in[1-h,1]$.
Then, as in the proof of Lemma \ref{lemALR}, we get
\begin{equation*}
%\begin{split}
\gamma_{j,k}(T,\alpha)\int_0^1 b^*(r) R_{j,k}(r )^2 r\, dr
\leq \gamma_{j,k}(T,\alpha)\int_0^{1-h} R_{j,k}(r )^2 r\, dr
 \leq  C \exp\left(2T(2j)^{2\alpha} - Cjh^{3/2} \right),
%\end{split}
\end{equation*}
which converges to $0$ as $j$ tends to $+\infty$. It follows that $J_r(b^*)=0$, which is absurd.
%\red{Remarque de Yannick: il suffirait de majorer $b^*$ sur $[0,1-h]$ par $1$, et d'appliquer Portmanteau avec les whispering galleries.}
\end{proof}

%Assume now by contradiction that the number of connected components of $\omega_{r}^*$ is finite. Then, there must exist $\eta\in (0,1)$ such that $(1-\eta,1)\subset \omega_{r}$ and we deduce from \cite[Lemma 3.1, (3.11)]{Lagnese} that\footnote{Note that this can be as well straightforwardly inferred from the quantum limits mentioned at the beginning of the proof of Lemma \ref{lemma:tech_ineq}.}
%$$
%\inf_{(j,k)\in \N\times \N^*}\int_{\omega_r^*} R_{j,k}(r )^2 \, rdr >0.
%$$
%Since the coefficients $\gamma_{j,k}$ grow exponentially, it follows immediately that the assumption \eqref{LBSC_weakened} is satisfied, which raises a contradiction. As a result, the set $\omega_{r}^*$ has an infinite number of connected components.

We are next going to prove that $b^*$ is unique, and is the characteristic function of some subset.

\paragraph{Existence and uniqueness of an optimal domain for the radial problem.}
First of all, setting
$$
\mathcal{S}=\left\{\beta=(\beta_{j,k})_{(j,k)\in\N\times\N^*}\in \ell^1(\R_{+}) \mid \sum_{(j,k)\in\N\times\N^*}\beta_{j,k}=1\right\} ,
$$
we clearly have the equality (by `convexifying" the infimum over discrete indices)
$$
J_r(b) = \inf_{(j,k)\in\N\times\N^*}  \gamma_{j,k}(T,\alpha)\int_0^1 b(r) R_{j,k}(r)^2\, rdr
=  \inf_{\beta\in\mathcal{S}}F(b,\beta),
$$
with
$$
F(b,\beta)=\sum_{(j,k)\in\N\times\N^*}\gamma_{j,k}(T,\alpha)\beta_{j,k}\int_0^1 b(r)R_{j,k}(r)^2\, rdr ,
$$
for every $b\in \overline{\mathcal{U}}_L^r$.
Therefore, we have
$$
\sup_{b\in \overline{\mathcal{U}}_{L}^r}J_r(b)=\sup_{b\in \overline{\mathcal{U}}_{L}^r}\inf_{\beta\in\mathcal{S}}F(b,\beta).
$$

We are going to apply a minimax theorem to the functional $F$.
Clearly, the function $F$ is upper semi-continuous with respect to its first variable, lower semi-continuous with respect to its second variable and concave-convex.
To derive the existence of a saddle point, some compactness properties are required. The set $\overline{\mathcal{U}}_L^r$ is ($L^\infty$ weakly star) compact, however the set $\mathcal{S}$ is not compact, so there is a difficulty here. This difficulty can however be overcome by using an extension of Sion's minimax theorem due to \cite{hartung}, by noticing the fact that, although $\mathcal{S}$ is not compact, the function $F$ is however inf-compact. Indeed for $b(\cdot)=L$, one has
$$F(L,\beta)=L\sum_{(j,k)\in\N\times\N^*}\gamma_{j,k}(\alpha,T)\beta_{j,k},$$
and then, using the fact that $\lambda_{j,k}=z_{j,k}^{2\alpha}\geq (j+k)^{2\alpha}$ (see \cite[Lemma 5]{Kelliher} for the latter inequality) and thus that the coefficients $\gamma_{j,k}(T,\alpha)$ have an exponential increase, it is easy to prove that the set
$
\{\beta \in \mathcal{S}\mid F(L,\beta)\leq \lambda\}
$
is compact in $\ell^1(\R)$, for every $\lambda\in\R$. This is the inf-compactness property.
Then, it follows from \cite[Theorem 1]{hartung} that there exists a saddle point $(b^*,\beta^*)\in\overline{\mathcal{U}}_L^r\times\mathcal{S}$ of the functional $F$, which implies in particular that
\begin{equation}\label{saddleDisk}
J_r(b^*)= \max_{b\in \overline{\mathcal{U}}_{L}^r}F(b,\beta^*) 
= \max_{b\in \overline{\mathcal{U}}_{L}^r} \int_0^1  \psi(r) b(r) r\, dr ,
\end{equation}
with the function $\psi$ defined by
\begin{equation}\label{def_psi}
\psi(r) = \sum_{(j,k)\in\N\times\N^*}\gamma_{j,k}(T,\alpha)\beta_{j,k}^* R_{j,k}(r)^2 .
\end{equation}
In other words, the function $r\mapsto b^*(r) r$ has to maximize a given integral (under a volume constraint), and therefore is characterized in terms of the level sets of the function $\psi$. More precisely, there exists a unique $\xi>0$ (which can be interpreted as a Lagrange multiplier, as in \cite[Theorem 1]{PTZobspb1}) such that
$$ b^*(r) = \left\{ \begin{array}{rcl}
1 & \textrm{if} & \psi(r)>\xi,\\
0 & \textrm{if} & \psi(r)<\xi,
\end{array}\right.$$
and the values of $b^*(r)$ are not determined by such first-order conditions on subsets of positive measure along which $\psi(r)=\xi$. Another way of expressing the latter case is to write that $\psi(r)=\xi$ on the set $\bigcup_{\varepsilon\in (0,1)}\{\varepsilon \leq b^*\leq 1-\varepsilon\}$.

\begin{remark}\label{remark_minimax}
Another consequence of the minimax theorem is that
$$
J_r(b^*) = \min_{\beta\in\mathcal{S}} \sum_{(j,k)\in\N\times\N^*} \beta_{j,k} \gamma_{j,k}(T,\alpha) \int_0^1 b^*(r) R_{j,k}(r)^2 r\, dr ,
$$
from which it follows that, if $\beta_{j,k,m}^*>0$, then necessarily there must hold 
$$\gamma_{j,k}(T,\alpha) \int_0^1 b^*(r) R_{j,k}(r)^2 r \, dr  = J_r(b^*),$$
and, conversely, if $\gamma_{j,k}(T,\alpha) \int_0^1 b^*(r) R_{j,k}(r)^2 r \, dr > J_r(b^*)$ then there must hold $\beta_{j,k}^*=0$.
In other words, the support of the Lagrange multipliers $\beta_{j,k}^*$ coincides with the set of active constraints, as is well known in constrained optimization. This remark will be useful in the sequel.
\end{remark}

It can be noted that the above minimax argument could have been applied as well to the case $\alpha>1/2$ but then it does not give any additional information. Here, this argument is instrumental in order to prove that $b^*$ is a characteristic function, as proved in what follows.

Let us come back to the expression of $b^*$ in terms of the level sets of the function $\psi$.
As a consequence of that expression, if we are able to state that the function $\psi$ cannot be constant on any subset of positive measure, then the function $a^*$ can only take the values $0$ and $1$, and therefore $b^*=\chi_{\omega_r^*}$ is the characteristic function of some subset $\omega_r^*$ such that $\chi_{\omega_r^*}\in\mathcal{U}_L^r$.
Typically this nondegeneracy assumption is satisfied as soon as the function $\psi$ is analytic.
And indeed we have the following result.

\begin{proposition}\label{technicalLemma_AnaDisk}
The function $\psi$ defined by \eqref{def_psi} is analytic in $(0,1)$.
\end{proposition}

%The proof of Proposition \ref{technicalLemma_AnaDisk} is lengthy and technical, and is postponed to the end of the section.

%Let us thus admit provisionally this lemma, from which 
With this proposition, it follows that, necessarily, $b^*=\chi_{\omega_r^*}\in\mathcal{U}_L^r$.
Hence, at this step, we can say that there exists an optimal domain for the radial problem \eqref{radialpb}, and that any maximizer of \eqref{radialpb} is a characteristic function.

Let us prove that the optimal domain is unique (and thus, that \eqref{radialpb} has a unique maximizer). The functional $b\mapsto J_r(b)$ is concave on $\overline{\mathcal{U}}_{L}^r$, since it is defined as the infimum of linear functionals. Therefore, if there were to exist two distinct maximizers $\chi_{\omega_r^1}$ and $\chi_{\omega_r^2}$, then, for every $t\in (0,1)$, the function $t\mapsto \chi_{\omega_r^1}+(1-t)\chi_{\omega_r^2}$ would be a maximizer of $J_r$ as well. But this contradicts the fact that any maximizer of \eqref{radialpb} is a characteristic function.

We have thus proved that the radial problem \eqref{radialpb} has a unique optimal domain $\omega_r^*$. Moreover, since $\psi$ is analytic in $(0,1)$, $\omega_r^*$ is semi-analytic in $(0,1)$ and thus $\omega_r^*$ intersected with any proper compact subset of $(0,1)$ has a finite number of connected components.

Therefore, according to Remark \ref{rem_rad}, this implies that $J$ has a unique maximizer $a^*=\chi_{\omega^*}$ with $\omega^*=\omega_{r}^*\times [0,2\pi]$ in polar coordinates.
In particular, the intersection of $\omega^*$ with any compact ring which is a proper subset of the unit disk is the union of a finite number of rings.
Note that there remains a problem in the neighborhood of $r=0$. We will tackle this problem later. Indeed at this step it could happen that there is an accumulation of rings at $r=0$. We will see later that this is not the case.

Let us next prove Proposition \ref{technicalLemma_AnaDisk}. Since its proof is very lengthy and technical, we encapsulate it in the next paragraph.

\paragraph{Proof of Proposition \ref{technicalLemma_AnaDisk}: the function $\psi$ is analytic in $\Omega$.}
The proof necessitates the use of fine properties and estimates of the eigenfunctions. We split this proof in several lemmas.
%Recall that, according to \eqref{basis_disk}, for $j\geq 1$ we have $\phi_{j,k,m}(r,\theta)^2 = R_{j,k}(r)^2 Y_{j,m}(\theta)^2$, with $Y_{j,1}(\theta)^2 = \frac{1}{\pi} \cos^2(j\theta)$ and $Y_{j,2}(\theta)^2 = \frac{1}{\pi} \sin^2(j\theta)$.

We will use in an instrumental way the following asymptotic properties of the functions $R_{j,k}$ defined by \eqref{def_Rjk}, already mentioned and used in \cite{PTZobsND}:
\begin{itemize}
\item for every $j\in\N$, the sequence of probability measures $R_{j,k}(r)^2r \,dr$ converges vaguely to $dr$ as $k$ tends to $+\infty$,
\item for every $k\in\N^*$, the sequence of probability measures $R_{j,k}(r)^2r \,dr$ converges vaguely to the Dirac at $r=1$ as $j$ tends to $+\infty$.
\item when taking the limit of $R_{j,k}(r)^2r \, dr$ with a fixed ratio $j/k$, and making this ratio vary, we obtain the family of probability measures
\begin{equation}\label{mesure_Burq}
\mu_s = f_s(r)\, dr=\frac{ 1} { \sqrt{ 1 - s^2}} \frac{ r } { \sqrt{ r^2 - s^2}} \chi_{(s,1)}(r) \, dr,
\end{equation}
parametrized by $s\in[0,1)$. We can even extend to $s=1$ by defining $\mu_1$ as the Dirac at $r=1$.
\end{itemize}
To be more precise with the latter property, let $z_{j,1}'$ be the first positive zero of $J'_{j}$. Then the first positive zero of $R'_{jk}$ is $r^1_{j,k}=\frac{z_{j,1}'}{z_{j,k}}$. The function $r\mapsto R_{j,k}(r)$ is positive and increasing on the interval $(0,r^1_{j,k})$, reaches a (global) maximum at $r^1_{j,k}$, and then oscillates and has $k$ zeros on $(r^1_{j,k},1]$, as can be seen on Figure \ref{fig_Bessel}.
If $j$ and $k$ tend to $+\infty$ with a constant ratio $j/k$ then $r^1_{j,k}$ converges to $s\in[0,1]$, where $s$ is the real number appearing in the formula \eqref{mesure_Burq}. Moreover, in accordance with the two first vague convergence properties recalled above, $s$ tends to $0$ whenever the ratio $j/k$ tends to $0$, and $s$ tends to $1$ whenever the ratio $j/k$ tends to $+\infty$.

\begin{figure}[h]
\begin{center}
\includegraphics[width=4.5cm]{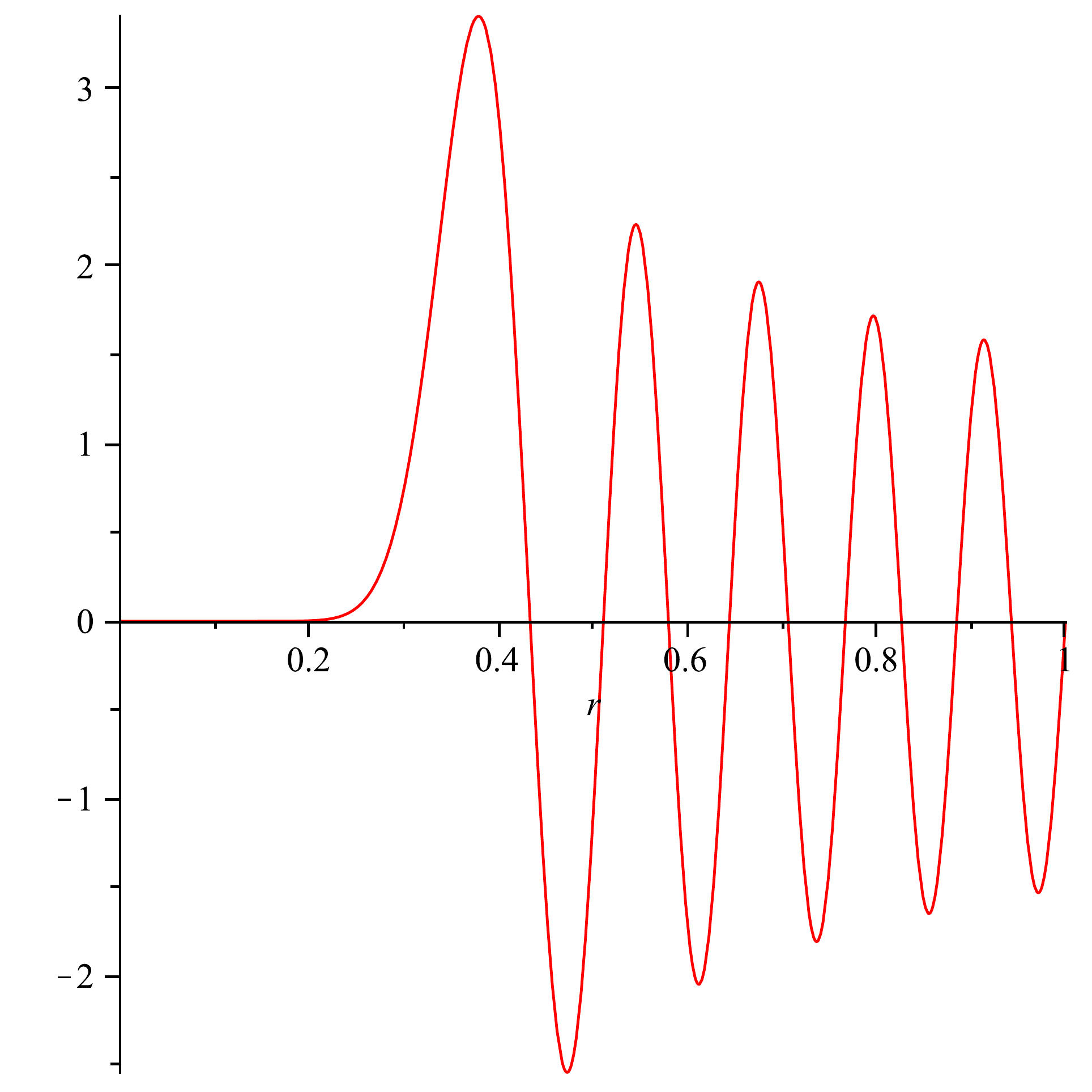}
\end{center}
\caption{The function $r\mapsto R_{j,k}(r)$ on $[0,1]$, with $j=20$, $k=10$.}\label{fig_Bessel}
\end{figure}

Note that these convergence properties provide some semi-classical measures (quantum limits) in the disk. The second one in particular accounts for the phenomenon of whispering galleries. 
The quantum limits \eqref{mesure_Burq} do not seem to be well known. They will be of particular importance in the sequel.

It can be noted that the above limits are in the sense of the vague topology only. Since this topology is weaker than the weak star topology of $L^\infty$, applying these convergence properties will raise some additional difficulties that are not obvious to overcome.
In particular, in the sequel we will need to use the Portmanteau theorem, in combination with a uniform bound in some appropriate Lebesgue space. This is the reason why the following results will be useful.

\begin{lemma}\label{lem_limitsBurq}
For every $h>0$, the family of functions $(f_s)_{0\leq s\leq 1-h}$ (defined by \eqref{mesure_Burq}) is uniformly bounded in $L^{3/2}(0,1)$.
\end{lemma}

\begin{proof}
Easy computations, not reported here, show that the function
$$
s \mapsto \int_0^1 f_s(r)^{3/2} dr = \frac{1}{(1-s^2)^{3/4}} \int_s^1 \frac{r^{3/2}}{(r^2-s^2)^{3/4}} dr
$$
is a continuous increasing function from $[0,1)$ to $[1,+\infty)$. The lemma follows.
\end{proof}

Unfortunately, this lemma alone is not sufficient in order to ensure that the functions $r\mapsto R_{j,k}(r)^2 r$ that are converging vaguely to the $f_s$ are also uniformly bounded in $L^{3/2}(0,1)$ (indeed the convergence is vague only). This uniform bound holds true however, but the proof of this fact requires a particular technical treatment.

\begin{lemma}\label{lem_RjkBurq}
For every $\alpha>0$, the sequence of functions $r\mapsto R_{j,k}(r)^2 r$ with $j\in\N^*$ and $k\in\N^*$ such that $j\leq \alpha k$, is uniformly bounded in $L^{3/2}(0,1)$.
\end{lemma}

Note that this uniform bound depends on $\alpha$ and tends to $+\infty$ as $\alpha$ tends to $0$ (that is, when the functions approach the whisepring galleries modes).

\begin{proof}[Proof of Lemma \ref{lem_RjkBurq}.]
From \eqref{def_Rjk}, we have
\begin{equation*}
R_{j,k}(r)^2 r = 2 \frac{ J_j(z_{j,k} r)^2 r }{ (J_j'(z_{j,k}))^2 } .
\end{equation*}

Let us first provide an asymptotic estimate of $(J_j'(z_{j,k}))^2$. First of all, using \cite[9.3.27 p. 367]{AS}, we have
$J_j'(j+xj^{1/3}) \sim -\frac{2^{2/3}}{j^{2/3}}\mathrm{Ai}'(-2^{1/3} x)$
%+ \mathrm{o}\left(\frac{1}{j^{2/3}}\right),
for every $x>0$, where $\mathrm{Ai}$ is the Airy function. Combining with the fact that $z_{j,k} \sim j + \delta_k j^{1/3}$ with $\delta_k=a_k2^{-1/3}>0$ (see \cite{Olver}), where $a_k>0$ is the $k^{\textrm{th}}$ positive zero of the function $x\mapsto \mathrm{Ai}(-x)$, it follows that 
$J_j'(z_{j,k})\sim -\frac{2^{2/3}}{j^{2/3}}\mathrm{Ai}'(-a_k)$. Now, using \cite{Krasikov}, we have, 
on the one hand, that $\mathrm{Ai}(-x)\sim \frac{1}{\sqrt{\pi}x^{1/4}}\cos\left( \frac{2}{3} x^{3/2} - \frac{\pi}{4} \right)$, from which it follows that $\frac{2}{3} a_k^{3/2} + \frac{\pi}{4}=k\pi$ (and thus, $a_k\sim (\frac{3}{2}\pi k)^{2/3}$), and on the other hand, that
$\mathrm{Ai}'(-x)\sim -\frac{x^{1/4}}{\sqrt{\pi}}\sin\left( \frac{2}{3} x^{3/2} - \frac{\pi}{4} \right)$, from which we infer that 
$\mathrm{Ai}'(-a_k)\sim \frac{a_k^{1/4}}{\sqrt{\pi}}=\left(\frac{3}{2}\right)^{1/6}\frac{k^{1/6}}{\pi^{1/3}}$. We conclude that
$\vert J_j'(z_{j,k})\vert \sim  \frac{3^{1/6}\sqrt{2}}{\pi^{1/3}} \frac{k^{1/6}}{j^{2/3}} $.

Now, using the estimate 
$$
\vert J_j(z_{j,k}r)\vert \leq \frac{2}{\sqrt{\pi}} \frac{1}{\vert z_{j,k}^2r^2-j^2+1/4\vert^{1/4}},
$$
coming from \cite{Krasikov}, and valuable for every $r\geq 0$, and using the inequality $z_{j,k}\geq j$, it follows that
$$
R_{j,k}(r)^2 r \leq C \frac{r}{\vert j^2r^2-j^2+1/4\vert^{1/2}} \frac{j^{4/3}}{k^{1/3}} \leq \frac{4}{\pi C} \frac{r}{\sqrt{\vert r^2-1\vert}} \leq \frac{C}{\sqrt{1-r}} \left(\frac{j}{k}\right)^{1/3} ,
$$
for every $r\in[0,1]$, for some constant $C>0$. The lemma follows easily.
\end{proof}

%Having in mind Remark \ref{remark_minimax}, let us prove a result on the active Lagrange multipliers. We define $\mathcal{I}(a^*)$ as the subset of all indices $(j,k,m)\in\mathcal{I}$ 
%for which the infimum is reached in the functional $J$, that is,
%\begin{equation*}
%\begin{split}
%\forall(j,k,m)\in\mathcal{I}(a^*)\qquad  \gamma_{j,k}(T,\alpha)\int_{\Omega} a^*(r,\theta) R_{j,k}(r)^2 Y_{j,m}(\theta)^2\, r dr d\theta &= J(\chi_{\omega^*}), \\
%\forall(j,k,m)\in\mathcal{I}\setminus\mathcal{I}(a^*)\qquad  \gamma_{j,k}(T,\alpha)\int_{\Omega} a^*(r,\theta) R_{j,k}(r)^2 Y_{j,m}(\theta)^2\, r dr d\theta & > J(\chi_{\omega^*}).
%\end{split}
%\end{equation*}
%According to Remark \ref{remark_minimax}, we have $\beta_{j,k,m}^*=0$ for all $(j,k,m)\in\mathcal{I}\setminus\mathcal{I}(a^*)$.
%
%Actually, we are going to prove that all indices $(j,k,m)$ such that either $k$ is fixed and $j$ tends to $+\infty$, or $j$ and $k$ tend to $+\infty$ with a ratio $j/k$ tending as well to $+\infty$, are not active.
%As explained above, these indices are those for which we obtain the whispering galleries at the limit. With the notation $r^1_{j,k}$, these indices are such that $r^1_{j,k}$ converges to $1$ as $j+k$ tends to $+\infty$.

Having in mind Remark \ref{remark_minimax}, let us prove a result on the active Lagrange multipliers. We define $\mathcal{I}(b^*)$ as the subset of all indices $(j,k)\in\N\times\N^*$ 
for which the infimum is reached in the functional $J_r$, that is,
\begin{equation*}
\begin{split}
\forall(j,k)\in\mathcal{I}(b^*)\qquad  \gamma_{j,k}(T,\alpha)\int_0^1 b^*(r) R_{j,k}(r)^2  r\, dr  &= J_r(b^*), \\
\forall(j,k)\in\N\times\N^*\setminus\mathcal{I}(b^*)\qquad  \gamma_{j,k}(T,\alpha)\int_0^1 b^*(r) R_{j,k}(r)^2  r \, dr & > J_r(b^*).
\end{split}
\end{equation*}
According to Remark \ref{remark_minimax}, we have $\beta_{j,k}^*=0$ for all $(j,k)\in\N\times\N^*\setminus\mathcal{I}(b^*)$.

Actually, we are going to prove that all indices $(j,k)$ such that $j+k$ tends to $+\infty$ with a ratio $j/k$ bounded from above, are not active (in other words, $\beta_{j,k}^*=0$ for such indices).
As explained above, these indices are those for which we avoid the whispering galleries at the limit, and are such that $r^1_{j,k}= \frac{z'_{j,1}}{z_{j,k}} < 1-h$ for some $h>0$ as $j+k$ tends to $+\infty$.

More precisely, let us prove the following lemma.

\begin{lemma}\label{lem_activeindices}
Let $h>0$ arbitrary.
Assume that $j+k$ tends to $+\infty$ with $r^1_{j,k} = \frac{z'_{j,1}}{z_{j,k}} < 1-\frac{h}{2}$. Then
%$$
%\gamma_{j,k}(T,\alpha)\int_{\Omega} a^*(r,\theta) R_{j,k}(r)^2 Y_{j,m}(\theta)^2\, r dr d\theta \rightarrow +\infty.
%$$
\begin{equation}\label{limlem13}
\gamma_{j,k}(T,\alpha)\int_0^1 b^*(r) R_{j,k}(r)^2  r \, dr \rightarrow +\infty.
\end{equation}
%Therefore we have $\mathcal{I}(a^*) \subset \{ (j,k,m)\in \N\times \N^*\times\{1,2\}\mid k=o(j) \}$
\end{lemma}

\begin{proof}[Proof of Lemma \ref{lem_activeindices}.]
Let $h>0$ arbitrary.
To prove \eqref{limlem13}, we are going to use the weak limits \eqref{mesure_Burq}. Since the limits are valuable in vague topology only, they cannot be applied directly to the function $b^*$, which is of class $L^\infty$ only. This is because of this defect in the convergence that we have to deal with indices such that $r^1_{j,k} = \frac{z'_{j,1}}{z_{j,k}} < 1-\frac{h}{2}$, and then we are going to deal with a smooth approximation of $b$ (to which we will be able to apply the weak limits \eqref{mesure_Burq}, in vague topology), and use Lemmas \ref{lem_adherence1} and \ref{lem_RjkBurq}.

From Lemmas \ref{lem_limitsBurq} and \ref{lem_RjkBurq}, there exists $C>0$ such that $\Vert f_s\Vert_{L^{3/2}}\leq C$ for every $s\in[0,1-\frac{h}{2}]$, and $\Vert R_{j,k}^2 r\Vert_{L^{3/2}}\leq C$ for all indices $(j,k)$ such that $r^1_{j,k}<1-\frac{h}{2}$.

From Lemma \ref{lem_adherence1}, the function $b^*$ is nontrivial on the subinterval $[1-\frac{h}{2},1]$.
By an easy computation, we have that $1\leq f_{s_1} < f_{s_2}(r)$ for every $r\in[0,1)$, whenever $0\leq s_1<s_2<1$. Then we get, for every $s\in[0,1-\frac{h}{2}]$,
\begin{equation}\label{18h26}
\int_{0}^1 b^*(r) f_s(r) \, dr \geq \int_{1-\frac{h}{2}}^1 b^*(r) \frac{ 1} { \sqrt{ 1 - s^2}} \frac{ r } { \sqrt{ r^2 - s^2}}  \,  dr  \geq \int_{1-\frac{h}{2}}^1 b^*(r)   \,  dr >0.
\end{equation}

Let $b$ be a nonnegative smooth function defined on $[0,1]$ such that
\begin{equation}\label{b*b}
\Vert b^* - b\Vert_{L^3} \leq \frac{1}{4C}\int_{1-\frac{h}{2}}^1 b^*(r)   \,  dr.
\end{equation}
Such a function $b$ can be obtained by convolution.

Now, we write
\begin{multline}\label{azer4}
\int_0^1 b^*(r) R_{j,k}(r)^2 r \, dr
= 
\int_0^1 b^*(r) f_s(r) \, dr + 
\int_0^1 b(r) \left( R_{j,k}(r)^2 r - f_s(r) \right) \, dr \\
+ \int_0^1 (b^*(r)-b(r)) \left( R_{j,k}(r)^2 r - f_s(r) \right) \, dr .
\end{multline}
Using the H\"older inequality, we have, using \eqref{b*b},
\begin{equation}\label{18h27}
\begin{split}
\left\vert  \int_0^1 (b^*(r)-b(r)) \left( R_{j,k}(r)^2 r - f_s(r) \right) \, dr  \right\vert 
&\leq \Vert b^*-b\Vert_{L^3} \left( \Vert R_{j,k}^2 r\Vert_{L^{3/2}} + \Vert f_s\Vert_{L^{3/2}} \right)  \\
&\leq 2C \Vert b^*-b\Vert_{L^3} \\
&\leq \frac{1}{2}\int_{1-\frac{h}{2}}^1 b^*(r)   \,  dr.
\end{split}
\end{equation}
Then, from \eqref{18h26}, \eqref{azer4} and \eqref{18h27}, we get that
$$
\int_0^1 b^*(r) R_{j,k}(r)^2 r \, dr
\geq \frac{1}{2}\int_{1-\frac{h}{2}}^1 b^*(r) \, dr + 
\int_0^1 b(r) \left( R_{j,k}(r)^2 r - f_s(r) \right) \, dr ,
$$
and then if we take indices $(j,k)$ such that $r^1_{j,k} = \frac{z'_{j,1}}{z_{j,k}} < 1-\frac{h}{2}$, with $j+k$ large enough, then according to the limit \eqref{mesure_Burq} which is valuable in vague topology, it follows (note that $b$ is smooth) that the integral $\int_0^1 b(r) \left( R_{j,k}(r)^2 r - f_s(r) \right) \, dr$ is small, and we get, if $j+k$ is small enough, that
$$
\gamma_{j,k}(T,\alpha) \int_0^1 b^*(r) R_{j,k}(r)^2 r \, dr
\geq \frac{1}{4}\gamma_{j,k}(T,\alpha)\int_{1-\frac{h}{2}}^1 b^*(r) \, dr.
$$
Since the right-hand side of the inequality tends to $+\infty$, the result follows.
\end{proof}

It follows from this lemma that we only need to consider indices such that $r^1_{j,k}$ converges to $1$ as $j+k$ tends to $+\infty$. This case involves indices $(j,k)$ such that $k$ is fixed and $j$ tends to $+\infty$, and indices such that $j$ and $k$ tend to $+\infty$ with a ratio $j/k$ tending as well to $+\infty$. In other words, this case concerns all indices for which we obtain the whispering galleries at the limit.

Therefore, at this step we have obtained that the function $\psi$ defined by \eqref{def_psi} can be written as
\begin{equation}\label{psi_restreinte}
\psi(r)=\sum_{\substack{(j,k)\in \N\times\N^* \\ r^1_{j,k}\geq 1-\frac{h}{2}}}\gamma_{j,k}(T,\alpha)\beta^*_{j,k} R_{j,k}(r)^2 ,
\end{equation}
The function $\psi$ is written as a series of analytic functions. In order to prove the analyticity of $\psi$, we are going to prove that all functions appearing in the sum, and all their derivatives, are bounded from above by some appropriate exponential functions, decreasing with $j$. To reach this objective, the estimate \eqref{estim1_Rjk} derived in Lemma \ref{lem1_Rjk} is not enough, since it was established for $k$ fixed. 
Let us then first extend the result of Lemma \ref{lem1_Rjk}, in order to prove that the estimate \eqref{estim1_Rjk} actually holds true for the set of indices appearing in the sum in the formula \eqref{psi_restreinte}.

\begin{lemma}\label{lem9}
For every $h\in (0,1)$, there exists a constant $C>0$ such that
\begin{equation}\label{estimRjk}
R_{j,k}(r)^2\leq C j^{4/3} \exp ( -Cjh^{3/2} ),
\end{equation}
for every $r\in [0,1-h]$, and for all indices $(j,k)$ such that $r^1_{j,k}\geq 1-\frac{h}{2}$.
\end{lemma}

Note that the estimate Lemma \ref{lem1_Rjk} was weaker (it was established for $k$ fixed), but has been crucial in order to prove that the function $b^*$ was nontrivial on any interval $[1-h,1]$. Note also that, in the statement above, it is important to assume that $r\leq 1-h$ and $r^1_{j,k}\geq 1-\frac{h}{2}$ (the gap of $\frac{h}{2}$ is crucial in the proof below).

\begin{proof}[Proof of Lemma \ref{lem9}]
The beginning of the proof is the same as the proof of Lemma \ref{lem1_Rjk}. In particular, we still have the inequality \eqref{tech19:22}.
Moreover, using the inequalities $z_{j,k}\geq j+k$ (see \cite[Lemma 5]{Kelliher}) and $z'_{j,1} = j+\gamma'_1j^{1/3}+\mathrm{o}(j^{1/3})\leq 2j$ for $j$ large enough, indices such that $r^1_{j,k}=\frac{z'_{j,1}}{z_{j,k}}\geq 1-\frac{h}{2}$ are such that $\frac{k}{j}\leq \frac{2}{1-\frac{h}{2}}-1$.
Then, the inequality \eqref{tech19:22} implies that
\begin{equation}\label{ALR1_bis}
\int_0^1 J_j(z_{j,k}r)^2 r dr
\geq \frac{C_1\gamma_1'}{\pi^2}\Big(1-\frac{h}{2}\Big) \frac{1}{j^{4/3}} .
\end{equation}

However, we cannot use the expansion $z_{j,k} = j + \delta_k j^{1/3} + o(j^{1/3})$, which is valuable for $k$ fixed only.
Then, from that step the proof differs from the one of Lemma \ref{lem1_Rjk}.

Using the assumption that $r^1_{j,k} = \frac{z'_{j,1}}{z_{j,k}}\geq 1-\frac{h}{2}$, we get $z_{j,k}\leq z'_{j,1}/(1-\frac{h}{2})$. Since $z'_{j,1} = j+\gamma'_1j^{1/3}+\mathrm{o}(j^{1/3}) \leq j ( 1+\frac{h}{4} )$ for $j$ large enough, it follows that $\frac{z_{j,k}}{j}\leq (1+\frac{h}{4})/(1-\frac{h}{2})$. 

Then, for every $r\in [0,1-h]$, we write $z_{j,k}r = j y$ with $y=\frac{z_{j,k}}{j}r$, and we get
$$
y = \frac{z_{j,k}}{j}r \leq  \frac{(1-h)(1+\frac{h}{4})}{1-\frac{h}{2}} \leq 1-\frac{h}{4},
$$
whenever $j$ is large enough. Therefore, as in the end of the proof of Lemma \ref{lem1_Rjk}, we get that, if $j$ is large enough, then
$$
\vert J_j(z_{j,k} r)\vert \leq \exp \left( j g \Big( 1-\frac{h}{4} \Big) \right),
$$
for every $r\in[0,1-h]$.
Using an asymptotic expansion of $g$, we get that
\begin{equation}\label{ALR2_bis}
\vert J_j(z_{j,k} r)\vert \leq \exp \left( -\frac{\sqrt{2}}{12} j h^{3/2} + o(j h^{3/2}) \right),
\end{equation}
for every $r\in[0,1-h]$.
Since $R_{j,k}(r)^2 = J_j(z_{j,k} r)^2 / \int_0^1 J_j(z_{j,k}r)^2 r \, dr$, the estimate \eqref{estimRjk} of the lemma finally follows by combining \eqref{ALR1_bis} with \eqref{ALR2_bis}.
\end{proof}

It is now required to estimate also all derivatives of the functions $R_{j,k}(r )^2$.
Let us do that, first, with the four first derivatives (before iterating) of $R_{j,k}(r )$.

\begin{lemma}\label{lem9_4derivatives}
For every $h\in (0,1)$, there exists a constant $C>0$ such that
\begin{equation}\label{estimRjk_n}
\left\vert\frac{d^n}{dr^n}(R_{j,k}(r))\right\vert\leq C j^{n+2} e^{-Cjh^{3/2}}.
\end{equation}
for every $r\in [h,1-h]$, for every $n\in\{0,1,2,3,4\}$, and for all indices $(j,k)$ such that $r^1_{j,k}\geq 1-\frac{h}{2}$.
\end{lemma}

Note that, in the estimates \eqref{estimRjk_n}, the power $3n$ is not optimal. Here, we keep however this power, because it will be enough, and it is simpler in order to be iterated.
Note also that, in the above estimates, we have excluded a neighborhood of $r=0$. This is technically due to the singularity of the polar coordinates, as seen in the proof below.

\begin{proof}[Proof of Lemma \ref{lem9_4derivatives}.]
Let us fix $h$, $C$, $j$ and $k$ as in Lemma \ref{lem9}. 
Using \eqref{def_Rjk}, we have
$$
\frac{d}{dr}(R_{j,k}(r )) = \frac{z_{j,k}J_{j}'(z_{j,k} r)}{\int_0^1 J_j(z_{j,k}r)^2r\, dr}.
$$
Using the well-known identity $2J_{j}'=J_{j-1}-J_{j+1}$ (see \cite{Watson}), 
using the inequality \eqref{ALR1_bis} and the estimate \eqref{estimRjk},
we get
\begin{equation}\label{11:32}
\left\vert\frac{d}{dr}(R_{j,k}(r )) \right\vert \leq \frac{2\pi^2}{C_1\gamma_1'}\frac{1}{1-\frac{h}{2}} j^{4/3} z_{j,k}  \vert J_{j-1}(z_{j,k}r) + J_{j+1}(z_{j,k}r) \vert
\end{equation}
Concerning the term $z_{j,k}$, using \cite[Lemma 5]{Kelliher}, we get
\begin{equation}\label{11:47}
z_{j,k} \leq \pi(j+k) = \pi j ( 1+\frac{k}{j}) \leq \frac{2\pi}{1-\frac{h}{2}} j,
\end{equation}
since $r^1_{j,k} \geq 1-\frac{h}{2}$ implies that $1+\frac{k}{j}\leq 2/(1-\frac{h}{2})$.

Besides, using again the assumption that $r^1_{j,k} = \frac{z'_{j,1}}{z_{j,k}}\geq 1-\frac{h}{2}$, we get $z_{j,k}\leq z'_{j,1}/(1-\frac{h}{2})$. Since $z'_{j,1} = j+\gamma'_1j^{1/3}+\mathrm{o}(j^{1/3}) \leq j ( 1+\frac{h}{4} )$ for $j$ large enough, it follows that $\frac{z_{j,k}}{j-1}\leq (1+\frac{h}{8})(1-\frac{h}{2})$. 
Then, for every $r\in [0,1-h]$, we write $z_{j,k}r = (j-1) y$ with $y=\frac{z_{j,k}}{j}r$, and we get
$$
y = \frac{z_{j,k}}{j-1}r \leq  \frac{(1-h)(1+\frac{h}{8})}{1-\frac{h}{2}} \leq 1-\frac{3h}{8},
$$
whenever $j$ is large enough. Therefore, as in the end of the proof of Lemma \ref{lem1_Rjk} or of Lemma \ref{lem9}, using the Kapteyn inequality we get that, if $j$ is large enough, then
\begin{equation}\label{Jj-1}
\vert J_{j-1}(z_{j,k} r)\vert \leq \exp \left( j g \Big( 1-\frac{3h}{8} \Big) \right),
\end{equation}
for every $r\in[0,1-h]$.

In a completely similar way, we get as well that, if $j$ is large enough, then
\begin{equation}\label{Jj+1}
\vert J_{j+1}(z_{j,k} r)\vert \leq \exp \left( j g \Big( 1-\frac{3h}{8} \Big) \right),
\end{equation}
for every $r\in[0,1-h]$.

Using an asymptotic expansion of $g$, it follows from \eqref{11:32}, \eqref{11:47}, \eqref{Jj-1} and \eqref{Jj+1} that
$$
\vert R_{j,k}'(r)\vert \leq C j^{7/3} \exp (-C j h^{3/2}).
$$
The estimate \eqref{estimRjk_n} with $n=1$ follows.

In order to derive the estimate \eqref{estimRjk_n} with $n=2$, we use the differential equation satisfied by $R_{j,k}(r)$, which is
\begin{equation}\label{odeRjk}
r^2R_{j,k}''(r) + r R_{j,k}'(r) + (z_{j,k}^2r^2-j^2) R_{j,k}(r) = 0.
\end{equation}
It follows from \eqref{odeRjk} that
$$
\vert R_{j,k}''(r)\vert \leq \frac{1}{r} \vert R_{j,k}'(r)\vert + \frac{\vert z_{j,k}^2r^2-j^2\vert}{r^2}\vert R_{j,k}(r)\vert.
$$
At this step, we can see that there is a difficulty in the neighborhood of $r=0$. We are thus obliged, in what follows, to distinguish between what happens in a neighborhood of $r=0$, and in the rest.

Let us assume that $r\geq h$. Then, we get easily that
\begin{equation}\label{estimRjkseconde}
\vert R_{j,k}''(r)\vert \leq C j^{4} \exp (-C j h^{3/2}),
\end{equation}
for some constant $C>0$, and for $r\in[h,1-h]$.

The estimates \eqref{estimRjk_n} with $n=3$ and $n=4$ are established in a similar way, by derivating with respect to $r$ the differential equation \eqref{odeRjk} and then proceeding as above. We do not give the details.
\end{proof}

Now, we are going to extend the estimates of Lemma \ref{lem9_4derivatives} to all possible derivatives, using an induction argument. The result is the following.

\begin{lemma}\label{lem9_allderivatives}
For every $h\in (0,1)$, there exists a constant $C>0$ such that
\begin{equation}\label{estimRjk_n_2}
\left\vert\frac{d^n}{dr^n}(R_{j,k}(r))\right\vert\leq C e^{2n} j^{n+2} \exp ( -C j h^{3/2} ),
\end{equation}
for every $r\in [h,1-h]$, for every $n\in\N$, and for all indices $(j,k)$ such that $r^1_{j,k}\geq 1-\frac{h}{2}$.
\end{lemma}

\begin{proof}[Proof of Lemma \ref{lem9_allderivatives}.]
%For $r$ such that $z_{j,k}r\leq 1$, we use the well-known analytic expansion of the Bessel function $J_j$, given by
%\begin{equation}\label{analytexpJj}
%J_j(x) = \sum_{n=0}^{+\infty} \frac{(-1)^n}{n!(n+j)!2^{2n+j}}x^{2n+j}.
%\end{equation}
%From this expansion, it follows the very rough estimate (which will be enough in what follows)
%\begin{equation}\label{roughestim}
%\vert J_j^{(n)}(x)\vert \leq \frac{ n! }{ (\frac{n}{2})! 2^n },
%\end{equation}
%valuable for every $n\in\N$ and for every $x>0$. 
%Since
%$$
%R_{j,k}^{(n)}(r) = z_{j,k}^n \frac{J_j^{(n)}(z_{j,k}r)}{\int_0^1 J_j(z_{j,k}r)^2 r \, dr}
%$$
%
%We infer from that estimate, from the inequality \eqref{ALR1_bis}, and from the fact that $z_{j,k}\leq 2\pi j/(1-\frac{h}{2})$ (due to the assumption that $r^1_{j,k}\geq 1-\frac{h}{2}$), that
%\begin{equation}
%\vert R_{j,k}^{(n)}(r)\vert \leq C j^{n+2} \frac{n!}{(\frac{n}{2})!2^n},
%\end{equation}
%for every $n\in\N$, for every $r\in[0,1]$ such that $z_{j,k}r\leq 1$, with a uniform constant $C>0$.

We are going to proceed with an induction argument.
Let us assume that
\begin{equation}\label{erti}
\vert R_{j,k}^{(i)}(r) \vert \leq C c_i j^{i+2} \exp ( -C j h^{3/2} ),
\end{equation}
for every $i\in\{0,\ldots,n-1\}$, and for every $r\in[h,1-h]$, and let us determine what can be an estimate of $c_n$. Derivating the differential equation $\eqref{odeRjk}$ $n-2$ times, we get
\begin{multline*}
r^2 R_{j,k}^{(n)}(r) + 2nrR_{j,k}^{(n-1)}(r) + n(n-1)R_{j,k}^{(n-2)}(r) \\
= -r R_{j,k}^{(n-1)}(r) - R_{j,k}^{(n-2)}(r) - (z_{j,k}^2r^2-j^2) R_{j,k}^{(n-2)}(r) - 2n z_{j,k}^2 r R_{j,k}^{(n-3)}(r) - n(n-1) z_{j,k}^2 R_{j,k}^{(n-4)}(r) ,
\end{multline*}
from which it follows that
\begin{multline*}
\vert R_{j,k}^{(n)}(r) \vert \leq
\frac{2n+1}{r} \vert R_{j,k}^{(n-1)}(r) \vert
+ \frac{ n(n-1)+1 + \vert z_{j,k}^2r^2-j^2\vert }{r^2} \vert R_{j,k}^{(n-2)}(r) \vert \\
+ \frac{2n z_{j,k}^2}{r}\vert R_{j,k}^{(n-3)}(r)\vert + n(n-1) \frac{z_{j,k}^2}{r^2} \vert R_{j,k}^{(n-4)}(r) \vert .
\end{multline*}
Now, using the inequalities \eqref{erti}, we see that we can define $c_n$ by
$$
c_n = n c_{n-1} + (n^2+1) c_{n-2} + n c_{n-3} + n^2 c_{n-4}.
$$
An easy study of this recurrence relation leads to the estimate
$
c_n \leq 2 \exp \left( \frac{1+\sqrt{5}}{2} n \right).
$
We finally conclude that
$$
\vert R_{j,k}^{(n)}(r) \vert \leq C e^{2n} j^{n+2} \exp ( -C j h^{3/2} ).
$$
The lemma is proved.
\end{proof}

As an immediate corollary, we have the following result.

\begin{lemma}\label{lem_Rjkcarre}
For every $h\in (0,1)$, there exists a constant $C>0$ such that
\begin{equation}\label{estimRjk_ncarre}
\left\vert\frac{d^n}{dr^n}(R_{j,k}(r)^2)\right\vert\leq C (2e^{2})^n j^{n+4} \exp ( -2C j h^{3/2} ),
\end{equation}
for every $r\in [h,1-h]$, for every $n\in\N$, and for all indices $(j,k)$ such that $r^1_{j,k}\geq 1-\frac{h}{2}$.
\end{lemma}

\begin{proof}[Proof of Lemma \ref{lem_Rjkcarre}.]
Using the estimates \eqref{estimRjk_n_2} of Lemma \ref{lem9_allderivatives}, we have
\begin{equation*}
\begin{split}
\left\vert\frac{d^n}{dr^n}(R_{j,k}(r)^2)\right\vert
& \leq \sum_{i=0}^n \begin{pmatrix} n\\ i\end{pmatrix} R_{j,k}^{(i)}(r)R_{j,k}^{(n-i)}(r) \\
& \leq C\sum_{i=0}^n \begin{pmatrix} n\\ i\end{pmatrix} e^{2n} j^{n+4}  \exp ( -2C j h^{3/2} )\\
& \leq C 2^n e^{2n} j^{n+4}  \exp ( -2C j h^{3/2} )
\end{split}
\end{equation*}
and the estimate follows.
\end{proof}

Let us finally finish the proof of Proposition \ref{technicalLemma_AnaDisk}.
Recall that the function $\psi$ is given by the formula \eqref{psi_restreinte}.
From the expression \eqref{gammajkdisk}, and using the inequality \eqref{11:47} which is valuable because of the assumption that $r_{j,k}^1\geq 1-\frac{h}{2}$, we have
$$
\gamma_{j,k}(\alpha,T)\leq \exp\left(  \left(\frac{2\pi}{1-\frac{h}{2}}\right)^{2\alpha} j^{2\alpha}   \right)
$$
Using the estimates \eqref{estimRjk_ncarre} of Lemma \ref{lem_Rjkcarre}, It follows that, for every $r\in(h,1-h)$, we have
\begin{equation*}
\begin{split}
\vert \psi^{(n)}(r)\vert 
&\leq \sum_{\substack{(j,k)\in \N\times\N^* \\ r^1_{j,k}\geq 1-\frac{h}{2}}}\gamma_{j,k}(T,\alpha)\beta^*_{j,k} \left\vert\frac{d^n}{dr^n}(R_{j,k}(r)^2)\right\vert \\
&\leq C (2e^2)^n \sum_{j=0}^{+\infty} j^{n+4} \exp \left( C j^{2\alpha} - 2 C j h^{3/2} \right),
\end{split}
\end{equation*}
where the constant $C>0$ is independent of $n$. Now, using the easy fact that
$$
\sum_{j=0}^{+\infty} j^n e^{-j} \leq \sum_{j=0}^{+\infty} (j+1)(j+2)\cdots(j+n) e^{-j} = \frac{n!}{(1-e^{-1})^{n+1}} ,
$$
we finally get that
$
\vert \psi^{(n)}(r)\vert \leq C^n n!,
$
for every $n\in\N$ and every $r\in(h,1-h)$, for some $C>0$ independent of $n$.
From these estimates, and from standard theorems on analytic functions, we finally infer that the function $\psi$ is analytic on $(h,1-h)$.
Since $h$ was taken arbitrary, Proposition \ref{technicalLemma_AnaDisk} is proved.

\paragraph{End of the proof of Theorem \ref{propLBSC1}.}

First of all, we are going to tackle the problem that has arised at $r=0$. Actually, due to polar coordinates, we have created a spurious singularity at $r=0$.
Let us then prove directly that the optimal domain $\omega^*$ is has a finite number of connected components, also in a neighborhood of the center of the disk.
Actually, let us prove directly that $\omega^*$ is semi-analytic in a neighborhood of the center of the disk.

To this aim, we are going to apply a minimax argument again. Using the notation $\mathcal{I}=\N\times \N^*\times \{1,2\}$ (already introduced) and setting
$$
\mathcal{T}=\left\{\beta=(\beta_{j,k,m})_{(j,k,m)\in\mathcal{I}}\in \ell^1(\R_{+}) \mid \sum_{(j,k,m)\in\mathcal{I}}\beta_{j,k,m}=1\right\} ,
$$
we have the equality
$$
J(a) = \inf_{(j,k,m)\in\mathcal{I}}  \gamma_{j,k}(T,\alpha)\int_{\Omega}a(x)\phi_{j,k,m}(x)^2  \, dx
=  \inf_{\beta\in\mathcal{S}}G(a,\beta),
$$
with
$$
G(a,\beta)=\sum_{(j,k,m)\in \mathcal{I}}\gamma_{j,k}(T,\alpha)\beta_{j,k,m}\int_{\Omega}a(x)\phi_{j,k,m}(x)^2 \, dx ,
$$
for every $a\in \overline{\mathcal{U}}_L$.
Therefore, we have
$$
\sup_{a\in \overline{\mathcal{U}}_{L}}J(a)=\sup_{a\in \overline{\mathcal{U}}_{L}}\inf_{\beta\in\mathcal{T}}G(a,\beta).
$$
We can then apply the minimax theorem of \cite{hartung} as previously, noticing that $G$ satisfies the same assumptions as the function $F$ of the radial case (in particular, the inf-compactness property must be underlined and its proof is similar).
Then, there exists a saddle point $(a^*,\beta^*)\in\overline{\mathcal{U}}_L\times\mathcal{T}$, and according to the whole analysis that has been done, we have $a^*=\chi_{\omega^*}$. Moreover, there exists a unique $\Xi>0$ such that
$$ 
a^*(x) = \left\{ \begin{array}{rcl}
1 & \textrm{if} & \Psi(x)>\Xi,\\
0 & \textrm{if} & \Psi(x)<\Xi,
\end{array}\right.  
$$
with
\begin{equation}\label{def_Psi}
\Psi(x) = \sum_{(j,k,m)\in \mathcal{I}}\gamma_{j,k}(T,\alpha)\beta_{j,k,m}^* \phi_{j,k,m}(x)^2.  %R_{j,k}(r)^2 Y_{j,m}(\theta)^2 .
\end{equation}
The only fact that remains to be proved is the fact that $\omega^*$ is semi-analytic in a neighborhood of the center of the disk. To this aim, it suffices to prove that the function $\Psi$ defined by \eqref{def_Psi} is analytic in the open unit disk. The proof of that fact follows exactly the same lines as the proof of Proposition \ref{technicalLemma_AnaDisk}. The only difference is that, instead of using polar coordinates, we keep the "intrinsic" variable $x\in\Omega$, and, in order to iterate the estimates of the derivatives of the eigenfunctions, we use the fact that, by definition, $\triangle \phi_{j,k,m} = -z_{j,k}^2\phi_{j,k,m}$ and we make use of Sobolev embeddings (as in the proof of Theorem 2 in \cite{PTZobspb1}).
We do not provide all details here, because the proof is already very lengthy and the arguments used here are similar to those given before.
It can be noted that we could as well have applied the latter minimax argument at the beginning of the proof, instead of focusing on the radial problem. But then in order to prove that the optimal domain is radial it would have been required anyway to study the radial problem. Here, we made the choice of beginning with the radial problem.

Now, combining all facts that have been proved earlier, we can assert that, if $0<\alpha <1/2$, or if $\alpha=1/2$ and $T$ is small enough, then there exists a unique optimal domain $\omega^*$, which is radial, and for which neither the assumption $\Hdeux$ nor its weakened version \eqref{LBSC_weakened} are satisfied. Moreover, the number of connected components of $\omega^*$ intersected with any proper compact subset of $\Omega$ is finite.

It remains to prove that the optimal set $\omega^*$ has an infinite number of concentric rings accumulating at the boundary.
Assume now by contradiction that the number of connected components of $\omega_{r}^*$ is finite. Then, there must exist $\eta\in (0,1)$ such that $(1-\eta,1)\subset \omega_{r}$ and we deduce from \cite[Lemma 3.1, (3.11)]{Lagnese} that\footnote{Note that this can be as well straightforwardly inferred from the quantum limits mentioned at the beginning of the proof of Lemma \ref{lemma:tech_ineq}.}
$$
\inf_{(j,k)\in \N\times \N^*}\int_{\omega_r^*} R_{j,k}(r )^2  r\, dr >0.
$$
Since the coefficients $\gamma_{j,k}$ grow exponentially, it follows immediately that the assumption \eqref{LBSC_weakened} is satisfied, which raises a contradiction. 

The proof of Theorem \ref{propLBSC1} is complete.

\subsection{Proof of Theorem \ref{theo:stokes}}\label{Sec:proofstokes}
First of all, it is clear that the assumptions $\Hun$ and $\Htrois$ are satisfied.
Using the Hilbert basis of eigenfunctions given by \eqref{Stokes_eig0} and \eqref{Stokes_eig1}, we have
$$
J(a)=\inf_{\substack{j\in\N \\ k\in \N^*}}\min_{m=1,2} \int_{0}^{2\pi}\!\!\!\int_{0}^1 a(r\cos \theta,r\sin \theta)\phi_{j,k,m}(r,\theta)^2  r \, drd\theta ,
$$
for every $a\in  \overline{\mathcal{U}}_{L}$. First of all, by a straightforward adaptation of the proof of Lemma \ref{lem6}, we prove that the problem of maximizing $J$ over $\overline{\mathcal{U}}_L$ is equivalent to the problem of maximizing $J$ over the radial functions of $\overline{\mathcal{U}}_L$. In other words, we have
$$
\max_{a\in \overline{\mathcal{U}}_{L}}J(a)=\max_{\substack{b\in L^\infty(0,1;[0,1])\\ \int_{0}^1b( r)\, rdr=\frac{L}{2}}}J_{r}(b),
$$
where 
$$
J_{r}(b)=\inf_{\substack{j\in\N \\ k\in \N^*}}\frac{\gamma_{j,k}(T)}{2}\int_{0}^1 b( r)\left(\frac{j^2}{r^2}f_{j,k}(r )^2+f_{j,k}'(r )^2\right)  r\, dr,
$$
and
$$
f_{j,k}(r )=\frac{J_j(\sqrt{\lambda_{j,k}} r)-J_j(\sqrt{\lambda_{j,k}})r^j}{\lambda_{j,k}\vert J_j(\sqrt{\lambda_{j,k}})\vert r}.
$$
We are going to prove that
\begin{equation}\label{15:17}
\lim_{j+k\to +\infty}\gamma_{j,k}(T)\int_{0}^1 b( r)\left(\frac{j^2}{r^2}f_{j,k}(r )^2+f_{j,k}'(r )^2\right) \, rdr=+\infty,
\end{equation}
for every $b\in L^\infty(0,1;[0,1])$ such that $\int_{0}^1b( r) r\, dr=\frac{L}{2}$
(this implies $\Hdeux$ for the radial problem).

Indeed, for such a function $b$, there exists $\varepsilon>0$ and a nontrivial subinterval $[\alpha,\beta]\subset [0,1]$, with $\alpha>0$, such that the restriction of $b$ to the interval $(\alpha,\beta)$ is nontrivial. More precisely, we assume that the restriction of $b$ to the interval $(\alpha,\frac{\alpha+\beta}{2})$ is nontrivial and also that the restriction of $b$ to the interval $(\frac{\alpha+\beta}{2},\beta)$ is nontrivial

In order to prove \eqref{15:17}, it suffices to prove that
\begin{equation}\label{tech_ineq}
%\lim_{j+k\rightarrow+\infty} \gamma_{jk}(T) \int_\alpha^\beta R_{j,k}(r)^2 r\,  dr = +\infty.
\lim_{j+k\to +\infty}\gamma_{j,k}(T)\int_{\alpha}^\beta b(r) \left(\frac{j^2}{r^2}f_{j,k}(r )^2+f_{j,k}'(r )^2\right) r\, dr=+\infty.
\end{equation}
It suffices to prove that fact for $j\geq 0$ (and then we assume $j\geq 0$ in what follows, which avoids to write $\vert j\vert$ in the sequel).
To prove \eqref{tech_ineq}, let us first note, that, using the fact that $(x^{-n} J_n(\alpha x))' = -\alpha x^{-n}J_{n+1}(\alpha x)$ for every $\alpha>0$, every $x\geq 0$ and every $n\in\N$ (see \cite{Watson}), we first get
$$
r^{-j} J_j(\sqrt{\lambda_{j,k}} r)-J_j(\sqrt{\lambda_{j,k}})  =  \sqrt{\lambda_{j,k}} \int_{r}^1 \frac{J_{j+1}(\sqrt{\lambda_{j,k}}x)}{x^j}\, dx
$$
and then
$$
f_{j,k}(r ) = \frac{r^j}{\lambda_{j,k}\vert J_j(\sqrt{\lambda_{j,k}})\vert r} \left( r^{-j} J_j(\sqrt{\lambda_{j,k}} r)-J_j(\sqrt{\lambda_{j,k}}) \right)
=
\frac{r^{j-1}}{\sqrt{\lambda_{j,k}}}\int_{r}^1 \frac{J_{j+1}(\sqrt{\lambda_{j,k}}x)}{x^j|J_{j}(\sqrt{\lambda_{j,k}})|}\, dx.
$$
This implies that
\begin{multline*}
\int_{\alpha}^\beta b(r) f'_{j,k}(r )^2 r\, dr + 
2\frac{j-1}{\lambda_{j,k}}\int_{\alpha}^\beta b(r) r^{j-2} \frac{J_{j+1}(\sqrt{\lambda_{j,k}}r)}{|J_{j}(\sqrt{\lambda_{j,k}})|} \int_{r}^1\frac{J_{j+1}(\sqrt{\lambda_{j,k}}x)}{x^j|J_{j}(\sqrt{\lambda_{j,k}})|}\, dx dr \\
 \geq \frac{1}{\lambda_{j,k}}\int_{\alpha}^\beta b(r) \frac{J_{j+1}(\sqrt{\lambda_{j,k}}r)^2}{rJ_{j}(\sqrt{\lambda_{j,k}})^2}\, dr.
\end{multline*}
Noticing that for every $j\geq 2$,
\begin{equation*}
\begin{split}
\frac{2(j-1)}{\sqrt{\lambda_{j,k}}}\int_{\alpha}^\beta b(r) \frac{J_{j+1}(\sqrt{\lambda_{j,k}}r)}{|J_{j}(\sqrt{\lambda_{j,k}})|}  \frac{f_{j,k}(r )}{r}\, dr  
&=  2(j-1)^2\int_{\alpha}^{\beta} b(r) \frac{f_{j,k}(r )^2}{r^2}\, dr \\
& \qquad - 2(j-1)\int_{\alpha}^\beta b(r) \frac{f_{j,k}(r )f'_{j,k}(r )}{r}\, dr \\
& \leq  \frac{j}{\alpha}\int_{\alpha}^\beta b(r) \left(\frac{j^2}{r^2}f_{j,k}(r )^2+f_{j,k}'(r )^2\right)  r\, dr,
\end{split}
\end{equation*}
we obtain
\begin{eqnarray}\label{eq:fjk}
\int_{\alpha}^\beta b(r) \left(\frac{j^2}{r^2}f_{j,k}(r )^2+f_{j,k}'(r )^2\right) \, rdr& \geq & \frac{\alpha}{\lambda_{j,k}(j+\alpha)} \int_{\alpha}^\beta b(r) \frac{J_{j+1}(\sqrt{\lambda_{j,k}}r)^2}{rJ_{j}(\sqrt{\lambda_{j,k}})^2}\, dr \nonumber\\
& \geq & \frac{\alpha}{2\pi^2\beta^2(j+k)^2(j+\alpha)}\int_{\alpha}^\beta b(r)  R_{j+1,k}(r )^2\, rdr,
\end{eqnarray}
where the functions $R_{j,k}$ are defined by \eqref{def_Rjk}, and where we have used the estimate of \cite[Lemma 5]{Kelliher}, as already done before.

Now, in order to derive \eqref{tech_ineq}, with the estimate \eqref{eq:fjk}, it suffices to prove the following lemma.

\begin{lemma}\label{lemma:tech_ineq}
We have
\begin{equation}\label{lemm11}
\lim_{j+k\to +\infty}\frac{\gamma_{j-1,k}(T)}{(j+k)^2(j+\alpha)}\int_{\alpha}^\beta b(r) R_{jk}(r )^2\, rdr=+\infty.
\end{equation}
\end{lemma}

\begin{proof}[Proof of Lemma \ref{lemma:tech_ineq}]
We are going to use the quantum limits recalled at the beginning of the proof of Proposition \ref{technicalLemma_AnaDisk}. Recall in particular the notation $r^1_{j,k}=\frac{z'_{j,1}}{z_{j,k}}$ used in that proof, and its role.

Let $j\in\N$ and $k\in\N^*$ be such that $j+k$ is large. We distinguish between two cases, in function of the value of $r^1_{jk}$ with respect to $\frac{\alpha+\beta}{2}$.

\medskip

If  $r^1_{j,k}\geq\frac{\alpha+\beta}{2}$ then
$$
R_{jk}(r)>  \frac{J_j(x)}{\sqrt{\int_0^1 J_j(z_{jk}r)^2 r \, dr}} ,
$$
for $x>0$ small enough, for every $r\in [x,r^1_{jk}]$, due to the fact that the function $J_j$ is increasing on $[0,r^1_{jk}]$. Now, using the facts that $J_j(y)=\frac{1}{\pi}\int_0^\pi \cos(j\tau-\sin\tau)d\tau$ (see \cite{AS}) and thus $\vert J_j(y)\vert\leq 1$, and that $J_j(x)\sim \frac{x^j}{2^j j!}$ for $x>0$ small (see \cite{Watson}), it follows that $R_{jk}(r)>  \frac{x^j}{2^j j!}$ for some given $x>0$, for every $r\in[\max(\alpha,x),\beta]$ (assuming $x<\beta$).
Besides, there holds $\gamma_{jk}(T)=\frac{e^{2\lambda_{jk}T}-1}{2\lambda_{jk}}$, and $\lambda_{jk}=z_{jk}^2>j^2$. Then \eqref{lemm11} follows easily, using the fact that $b$ is nontrivial along $(\alpha,\frac{\alpha+\beta}{2})$.

\medskip

If $r^1_{j,k}<\frac{\alpha+\beta}{2}$, then as in the proof of Lemma \ref{lem_activeindices},  we are going to use weak limits that are established for the vague convergence only. It is then required to approximate $b$ with a smooth function $b_1$.
Let then $\varepsilon>0$ arbitrary, and let $b_1$ be a nonnegative smooth function defined on $[0,1]$ (and obtained for instance by convolution) such that
\begin{equation}\label{approxb}
\Vert b - b_1 \Vert_{L^3} \leq \varepsilon.
\end{equation}
Moreover, since $b$ is nontrivial along $(\frac{\alpha+\beta}{2},\beta)$, we can assume as well that the restriction of $b_1$ to $(\frac{\alpha+\beta}{2},\beta)$ is nontrivial.

Now we proceed as in the proof of Lemma \ref{lem_activeindices}. 
We write
\begin{multline}\label{azer4_bis}
\int_\frac{\alpha+\beta}{2}^\beta b(r) R_{j,k}(r)^2 r \, dr
= 
\int_\frac{\alpha+\beta}{2}^\beta b(r) f_s(r) \, dr + 
\int_\frac{\alpha+\beta}{2}^\beta b_1(r) \left( R_{j,k}(r)^2 r - f_s(r) \right) \, dr \\
+ \int_\frac{\alpha+\beta}{2}^\beta (b(r)-b_1(r)) \left( R_{j,k}(r)^2 r - f_s(r) \right) \, dr .
\end{multline}
From Lemmas \ref{lem_limitsBurq} and \ref{lem_RjkBurq}, there exists $C>0$ such that $\Vert f_s\Vert_{L^{3/2}}\leq C$ for every $s\in[0,\frac{\alpha+\beta}{2}]$, and $\Vert R_{j,k}^2 r\Vert_{L^{3/2}}\leq C$ for all indices $(j,k)$ such that $r^1_{j,k}<\frac{\alpha+\beta}{2}$.
Using the H\"older inequality, we have, using \eqref{approxb},
\begin{equation}\label{18h27_bis}
\begin{split}
\left\vert  \int_\frac{\alpha+\beta}{2}^\beta (b(r)-b_1(r)) \left( R_{j,k}(r)^2 r - f_s(r) \right) \, dr  \right\vert 
&\leq \Vert b-b_1\Vert_{L^3} \left( \Vert R_{j,k}^2 r\Vert_{L^{3/2}} + \Vert f_s\Vert_{L^{3/2}} \right)  \\
&\leq 2C \Vert b^*-b\Vert_{L^3} \\
&\leq 2C\varepsilon.
\end{split}
\end{equation}
Then, from \eqref{azer4_bis} and \eqref{18h27_bis}, we get that
$$
\int_\frac{\alpha+\beta}{2}^\beta b(r) R_{j,k}(r)^2 r \, dr
\geq  \int_\frac{\alpha+\beta}{2}^\beta b(r) \, dr +  \int_\frac{\alpha+\beta}{2}^\beta b_1(r) \left( R_{j,k}(r)^2 r - f_s(r) \right) \, dr - 2C \varepsilon,
$$
and then, for indices $(j,k)$ such that $r^1_{j,k} = \frac{z'_{j,1}}{z_{j,k}} < \frac{\alpha+\beta}{2}$, with $j+k$ large enough, according to the limit \eqref{mesure_Burq} which is valuable in vague topology, it follows (note that $b_1$ is smooth) that the integral $\int_\frac{\alpha+\beta}{2}^\beta b_1(r) \left( R_{j,k}(r)^2 r - f_s(r) \right) \, dr$ is small, and we get, if $j+k$ is small enough, and if $\varepsilon$ is small enough, that
$$
\frac{\gamma_{j-1,k}(T)}{(j+k)^2(j+\alpha)} \int_\alpha^\beta b(r) R_{j,k}(r)^2 r \, dr
\geq \frac{1}{4}\frac{\gamma_{j-1,k}(T)}{(j+k)^2(j+\alpha)})\int_\frac{\alpha+\beta}{2}^\beta b(r) \, dr.
$$
Since the right-hand side of the inequality tends to $+\infty$, the result follows.
\end{proof}

The proof of Theorem \ref{theo:stokes} is complete.

\section{Conclusion}\label{sec_ccl}
Considering general parabolic equations on a bounded open connected subset $\Omega$ of $\R^n$, we have modeled the problem of optimal shape and location of the observation domain having a prescribed measure, in terms of maximizing a spectral functional over all measurable subsets of fixed Lebesgue measure. This spectral functional has been interpreted as a randomized observability constant which, in contrast to the classical deterministic one, corresponds to an average version of the classical observability inequality over random initial data.

We have given sufficient spectral assumptions under which we are able to prove that the resulting optimal design problem has a unique solution. Moreover, the optimal domain can be built from a truncated version of the spectral functional, thus with a finite number of modes only.
The optimal domain is semi-analytic and hence has a finite number of connected components, which is in strong contrast with previous results obtained for conservative wave and Schr\"odinger equations.

We have proved that our results cover the case of the Stokes equation in the disk and of anomalous diffusion equations in which the operator is given by an arbitrary positive power $\alpha$ of the negative of the Dirichlet-Laplacian.
Using a refined and highly technical analysis, we have been able to prove that, for anomalous diffusion equations, the complexity of the optimal domain may depend both on the geometry of the domain and on the value of $\alpha$.
In particular we have proved that, in the unit square of $\R^2$, the optimal domain has a finite number of connected components, and this, independently on the value of $\alpha$. In contrast, in the unit disk of $\R^2$, the optimal domain consists of a finite number of rings if $\alpha>1/2$, and of an infinite number of rings accumulating at the boundary if $\alpha<1/2$ or if $\alpha=1/2$ and $T$ is small enough.
These properties have been illustrated on several numerical simulations.

To conclude, let us provide several further comments and open problems.

\subsection{Exponential concentration properties of eigenfunctions}
In Section \ref{ex4}, in order to prove Theorem \ref{theoLBSC3}, we have used in particular a minimax argument.
It can be noted that this argument can be applied as well in the general case, and will lead us to comment on the possible concentration properties of eigenfunctions.

Let us provide the general minimax argument.
Setting
$$
\mathcal{S}=\Big\{\beta=(\beta_{j})_{j\in\N^*}\in \ell^1(\R_{+}) \mid \sum_{j\in\N^*}\beta_{j}=1\Big\} ,
$$
as in Section \ref{ex4} we have the equality
$$
J(a) = \inf_{j\in\N^*}  \gamma_{j}(T)\int_\Omega a(x) \phi_j(x)^2\, dx
=  \inf_{\beta\in\mathcal{S}}F(a,\beta),
$$
with
$$
F(a,\beta)=\sum_{j\in\N^*}\gamma_{j}(T)\beta_{j}\int_\Omega a(r) \phi_j(x)^2\, dx ,
$$
for every $a\in \overline{\mathcal{U}}_L$.
Therefore, we have
$$
\sup_{a\in \overline{\mathcal{U}}_L}J(a)=\sup_{a\in \overline{\mathcal{U}}_L}\inf_{\beta\in\mathcal{S}}F(a,\beta).
$$
The function $F$ is upper semi-continuous with respect to its first variable, lower semi-continuous with respect to its second variable and concave-convex.
The set $\overline{\mathcal{U}}_L$ is ($L^\infty$ weakly star) compact, and besides, although the set $\mathcal{S}$ is not compact, the function $F$ is however inf-compact. Indeed for $a(\cdot)=L$, one has $F(L,\beta)=L\sum_{j\in\N^*}\gamma_{j}(T)\beta_{j,k}$, and then, since $\Real(\lambda_{j})$ tends to $+\infty$ as $j$ tends to $+\infty$ by assumption, it follows that the coefficients $\gamma_{j}(T)$ defined by \eqref{defgammaj} have an exponential increase, and therefore the set $\{\beta \in \mathcal{S}\mid F(L,\beta)\leq \lambda\}$ is compact in $\ell^1(\R)$, for every $\lambda\in\R$. 
Then, it follows from \cite[Theorem 1]{hartung} that there exists a saddle point $(a^*,\beta^*)\in\overline{\mathcal{U}}_L\times\mathcal{S}$ of the functional $F$, which implies in particular that
\begin{equation*}
J(a^*)= \max_{a\in \overline{\mathcal{U}}_L}F(a,\beta^*) 
= \max_{a\in \overline{\mathcal{U}}_L} \int_\Omega  \psi(x) a(x)\, dx ,
\end{equation*}
with the function $\psi$ defined by
\begin{equation}\label{def_psi_gen}
\psi(x) = \sum_{j\in\N^*}\gamma_{j}(T)\beta_{j}^* \phi_j(x)^2 .
\end{equation}
In other words, the function $a^*$ has to maximize a given integral (under a volume constraint), and therefore is characterized in terms of the level sets of the function $\psi$. More precisely, there exists a unique $\xi>0$ such that
$$ a^*(x) = \left\{ \begin{array}{rcl}
1 & \textrm{if} & \psi(x)>\xi,\\
0 & \textrm{if} & \psi(x)<\xi,
\end{array}\right.$$
and the values of $a^*(x)$ are not determined by such first-order conditions on subsets of positive measures along which $\psi(x)=\xi$. Equivalently, there holds $\psi(x)=\xi$ on the set $\bigcup_{\varepsilon\in (0,1)}\{\varepsilon \leq a^*\leq 1-\varepsilon\}$.
%
%\begin{remark}\label{remark_minimax_gen}
%Another consequence of the minimax theorem is that
%$$
%J_r(b^*) = \min_{\beta\in\mathcal{S}} \sum_{(j,k)\in\N\times\N^*} \beta_{j,k} \gamma_{j,k}(T,\alpha) \int_0^1 b^*(r) R_{j,k}(r)^2 r\, dr ,
%$$
%from which it follows that, if $\beta_{j,k,m}^*>0$, then necessarily there must hold 
%$$\gamma_{j,k}(T,\alpha) \int_0^1 b^*(r) R_{j,k}(r)^2 r \, dr  = J_r(b^*),$$
%and conversely if $\gamma_{j,k}(T,\alpha) \int_0^1 b^*(r) R_{j,k}(r)^2 r \, dr > J_r(b^*)$ then there must hold $\beta_{j,k}^*=0$.
%In other words, the support of the Lagrange multipliers $\beta_{j,k}^*$ coincides with the set of active constraints, as is well known in constrained optimization. This remark will be useful in the sequel.
%\end{remark}
%
As a consequence, if we are able to establish that, under appropriate assumptions, the function $\psi$ defined by \eqref{def_psi_gen} is analytic in $\Omega$ (or at least, cannot be constant on any subset of positive measure), then the function $a^*$ can only take the values $0$ and $1$, and therefore $a^*=\chi_{\omega^*}$ is the characteristic function of some subset $\omega^*$ such that $\chi_{\omega^*}\in\mathcal{U}_L$.

We have seen in Section \ref{ex4} with the example of the unit disk that, in order to prove that the switching function $\psi$ is analytic, fine asymptotic properties of the eigenfunctions have to be used, in combination with a study of active and inactive indices.

As explained in Section \ref{sec_LBSC}, if every quantum limit associated with the eigenfunctions $\phi_j$ contains a density which is positive over the whole $\Omega$, then the uniform lower bound assumption $\Hdeux$ of Theorem \ref{mainTheo} is satisfied, and the existence and uniqueness of a semi-analytic optimal domain follows. Therefore the worst possible case is when there exist quantum limits which are completely concentrated, such as a Dirac along the projection on the configuration space of a closed geodesic $\mathcal{C}$ of the phase space (scarring eigenfunctions). 
%In Section \ref{sec_LBSC}, this worst possible case was given by the Dirac along the boundary of the disk (whispering galleries). 
In accordance with the proof in Section \ref{ex4} (see in particular Lemma \ref{lem_activeindices}), indices $j$ associated with subsequences of $\phi_j^2\, dx$ converging to a completely singular measure along $\mathcal{C}$ may be active, that is, $\beta_j^*> 0$ a priori. Then, in order to ensure that $\psi$ is analytic, it is required to know that the scarring subsequences of eigenfunctions $\phi_j$ enjoy exponential concentration properties. More precisely it is required to know that, outside of any neighborhood of $\mathcal{C}$, the subsequences $\phi_j$'s which concentrate on $\mathcal{C}$ can be bounded from above by exponentials decreasing with $j$.
For completely integrable systems, such concentration properties are expected to occur on invariant tori, but up to our knowledge no general result is known.
For ergodic systems, in relation with the Shnirelman theorem, the situation is widely open.
In any case, any new result establishing some concentration features for eigenfunctions could certainly help to analyze the analyticity properties of the switching function $\psi$.

\subsection{Several open problems}

In this section we briefly comment on some open problems and subjects for possible future research related with the contents of this paper.

\paragraph{On the Strong Conic Independence Property $\Hdeux$.} 
Our methods apply to a wide class of parabolic problems allowing a spectral decomposition. One of the key subtle issues is to verify the property $\Hdeux$. The way we have addressed this property is by combining the analyticity of the eigenfunctions (which allows to extend the dependence condition to the whole domain $\Omega$) with the boundary conditions. This applies to the Dirichlet-Laplacian. The same proof would apply for other elliptic equations with analytic coefficients. 

But the analysis of this issue is widely open in two directions. First, in what concerns elliptic equations with non-analytic coefficients, and, second, for other boundary conditions. For instance, the proof above does not work for the Neumann boundary conditions except in some particular cases (as in the square domain where the eigenfunctions are explicit in separated variables).

\paragraph{On the concentration of eigenfunctions.} Another spectral property that plays a key role in our analysis is the lower bounds of the form
\begin{equation}\label{deconcentration}
\int_E \phi_j(x)^2 \, dx \geq \frac{e^{-2C\sqrt{\mu_j}}}{C^2\vert E\vert} ,
\end{equation}
valid uniformly over measurable sets $E$, proved in \cite{AEWZ} and extending previous results in \cite{LebeauRobbiano} on open sets.

Extending the results of this paper to other parabolic models would require the extension of these concentration inequalities for the corresponding spectra: other boundary conditions, Stokes problem, fourth order problems, reaction-diffusion systems, etc.

Note also that one could also consider optimal placement problems for these inequalities. For instance, in view of \eqref{deconcentration}, it would be natural to analyze the problem of determining the measurable set $E$ in $\Omega$ of a given measure so that the constant in the exponential lower bound in (\ref{deconcentration}) is minimized. This would correspond, in some sense, to determining the set $E$ where the deconcentration is maximized, which would be a good candidate for being an optimal observation set of the whole dynamics. Note however that this problem involves the whole spectrum of the Laplacian, contrarily to the problem considered here for the heat equation in which the intrinsic strong dissipative effect makes the whole problem to be governed by a finite number of eigenfunctions.

\paragraph{Heat equations with lower order potentials.}
The results of this paper could be applied to heat equations with lower order terms, of the form
\begin{equation}\label{heatEq_intropot}
\partial_t y-\triangle y + p(x) y=0, \quad (t,x)\in (0,T)\times\Omega.
\end{equation}
The techniques of this paper allow to introduce, by the randomization procedure, the spectral observability criterion. The proof of the main results of this paper requires exponential lower bounds as in \cite{AEWZ} on the possible concentration of the eigenfunctions of the associated operator $- \triangle + p\,\mathrm{id}$ on measurable sets. The analysis of these concentration inequalities, and their dependence on the regularity of the potential $p=p(x)$ is, as far as we know, an open problem.

The topics considered in this paper are even more widely open in the case where the potential $p$ is also time-dependent, in particular because of the lack of the existence of a spectral basis to perform the Fourier decomposition leading to the spectral criterion that we have considered throughout.

\paragraph{Heat equations with convective potentials.} It would be also interesting to analyze these issues for heat equations with convective terms, of the form
\begin{equation}\label{heatEq_intropotcon}
\partial_t y-\triangle y + V(x) \cdot \nabla y=0, \quad (t,x)\in (0,T)\times\Omega.
\end{equation}
In dimension one, with a simple change of variables, the problem can be reduced to the form \eqref{heatEq_intropot}. Of course this is no longer true in the multi-dimensional case. The main difficulty is then that one cannot expand the solutions in a spectral basis. 
Note that null controllability and observability properties of these models has been analyzed in \cite{Leautaud} by means of resolvent estimates. But nothing is known about the optimal location of sensors or actuators.

\paragraph{Analysis of the full Gramian.} 
As explained in the paper, our analysis corresponds to focus on the diagonal terms of the Gramian operator \eqref{defGT}. The analysis of the full Gramian is a widely open problem. Note that, as pointed out in \cite{MZ1}, the off-diagonal terms in the Gramian operator play a fundamental role when dealing with the full observation/control problems. One cannot exclude that, when dealing with the full Gramian, the optimal domains be not fully determined by a finite number of Fourier modes.

\paragraph{Other fractional models.} 
In this paper we have considered the fractional Laplacian defined in spectral terms. In this way we have taken advantage of the properties that are well known for the spectrum of the Laplacian such as, for instance, $\Hdeux$ or \eqref{deconcentration}. Significant added developments would be needed to consider parabolic equations involving other versions of the fractional Laplacian, such as the non-local one in \cite{CS}. As far as we know, very little is known about the spectrum of this operator, except in the 1D case. The analogue of $\Hdeux$ and \eqref{deconcentration}, and of the results of this paper constitute interesting open problems for these alternative fractional models. The same can be said for models involving fractional derivatives in time (see \cite{KLW}).

\bigskip

\noindent{\bf Acknowledgment.}

We are indebted to Nicolas Burq for fruitful discussions.

The first author was partially supported by the ANR project OPTIFORM.

The third author was partially supported by the Grant MTM2011-29306-C02-00 of the MICINN (Spain), project PI2010-04 of the Basque Government, the ERC Advanced Grant FP7-246775 NUMERIWAVES, the ESF Research Networking Program OPTPDE.

This work was done while the third author was visiting the CIMI (Centre International de Mathématiques et Informatique) of Toulouse (France) in the frame of the CIMI Excellence Chair on ``Control, PDEs, Numerics and Applications".

\end{document}